\documentclass[reqno]{amsart} 

\usepackage{blindtext}
\usepackage[margin=1.5in]{geometry}
\usepackage[dvipsnames]{xcolor}
\usepackage{tikz-cd}
\usepackage[utf8]{inputenc} 
\usepackage[english]{babel}
\usepackage{graphicx} 
\usepackage{amsmath,amssymb}

\usepackage{amsthm} 
\usepackage{color} 
\usepackage[T1]{fontenc} 
\usepackage{etoolbox} 
\usepackage{mathtools}

\usepackage{blindtext} 
\usepackage{enumerate,comment} 
\usepackage{enumerate,amssymb, xcolor,mathrsfs} 
\usepackage[colorlinks]{hyperref} 

\newtheorem{mainthm}{Theorem} 
\newtheorem{maincor}[mainthm]{Corollary}

\usepackage{accents}

\newtheorem{thm}{Theorem}[section]
\newtheorem*{theorem*}{Theorem}
\newtheorem{cor}[thm]{Corollary}
\newtheorem{lem}[thm]{Lemma}
\newtheorem{prop}[thm]{Proposition}

\theoremstyle{definition}
\newtheorem{defn}[thm]{Definition}
\theoremstyle{remark}

\numberwithin{equation}{section}

\theoremstyle{definition}

\newcommand{\calH}{\mathcal{H}}

\newcommand{\calC}{\mathcal{C}}
\newcommand{\calB}{\mathcal{B}}

\newcommand{\calM}{\mathcal{M}}
\newcommand{\calU}{\mathcal{U}}
\newcommand{\calL}{\mathcal{L}}

\newcommand{\calK}{\mathcal{K}}
 
\newcommand{\calR}{\mathcal{R}}

\newcommand{\calQ}{\mathcal{Q}}
\newcommand{\calI}{\mathcal{I}}
\newcommand{\calT}{\mathcal{T}}

\newcommand{\Z}{\mathcal Z}

\makeatletter
\def\bign#1{\mathclose{\hbox{$\left#1\vbox to8.5\p@{}\right.\n@space$}}\mathopen{}}
\makeatother

\usepackage{graphicx}

\let\svthefootnote\thefootnote
\newcommand\freefootnote[1]{%
  \let\thefootnote\relax%
  \footnotetext{#1}%
  \let\thefootnote\svthefootnote%
}

\hypersetup{
    colorlinks,
    citecolor=black,
    filecolor=black,
    linkcolor=black,
    urlcolor=black
}

\begin{document}  
\title{Uniqueness for embeddings of nuclear $C^*$-algebras into type II$_{1}$ factors}  
\author{Shanshan Hua}
\address{Mathematisches Institut, Fachbereich Mathematik und
Informatik der Universität Münster, Einsteinstrasse 62, 48149 Münster, Germany.} 
\email{shua@uni-muenster.de} 
\author{Stuart White} 
\address{Mathematical Institute, University of Oxford, Andrew Wiles Building, Radcliffe Observatory Quarter, Woodstock Road, Oxford, OX2 6GG, United Kingdom}
\email{stuart.white@maths.ox.ac.uk}
\thanks{The first author was funded by the Deutsche Forschungsgemeinschaft (DFG, German Research Foundation) under Germany’s Excellence Strategy EXC 2044-390685587, Mathematics Münster: Dynamics–Geometry–Structure, by the SFB 1442 of the DFG. The second author was supported by the Engineering and Physical Sciences Research Council [EP/X026647/1]. For the purpose of open access, the authors have applied a CC-BY license to any author accepted manuscript arising from this submission. } 
\maketitle 

\begin{abstract} 
     Let $A$ be a separable, unital and exact $C^*$-algebra satisfying the universal coefficient theorem. We prove uniqueness theorems up to unitary conjugacy for unital, full and nuclear maps from $A$ into norm-ultraproducts of finite von Neumann factors: any two such maps agreeing on traces and total $K$-theory are unitarily equivalent. There are two consequences. Firstly if one takes the factors to be a sequence $(M_{k_n})_{n}$ of matrix algebras, we obtain a uniqueness result for quasidiagonal approximations of $A$. Secondly, when $(\mathcal M,\tau_{\calM})$ is a II$_1$ factor, a pair  $\phi,\psi:A\to\mathcal M$ of unital, injective and nuclear maps are norm approximately unitarily equivalent if and only if $\tau_{\calM}\circ\phi=\tau_{\calM}\circ\psi$.

     The main strategy is to use Schafhauser's classification of lifts along the trace-kernel extension from \cite{Schafhauser.Annals}.  Since our codomains may lack the tensorial absorption properties needed in \cite{Schafhauser.Annals}, the main new ingredient is a suitable $KK$-uniqueness theorem tailored to our situation. This is inspired by $KK$-uniqueness theorems of Loreaux, Ng and Sutradhar (\cite{ess-codim}).
\end{abstract} 

\section*{Introduction}

\subsection*{Uniqueness of maps into II$_1$ factors} The Gelfand--Naimark--Segal (GNS) construction (\cite{GN,Segal:GNS}) is a foundational result in the subject of operator algebras, showing that $C^*$-algebras have faithful representations as bounded operators on a Hilbert spaces. Precisely, for any $C^*$-algebra $A$ there exist a Hilbert space $\mathcal H$, which can be taken to be the infinite dimensional separable Hilbert space when $A$ is separable, and an injective $^*$-homomorphism $A\hookrightarrow\mathcal B(\mathcal H)$.  Over 30 years later, Voiculescu's non-commutative Weyl--von Neumann theorem (\cite[Theorem 1.5]{Voiculescu}; see also Arveson's account in \cite{Arveson:DJM}), gave a uniqueness counterpart to the GNS construction. Famously, any two essential representations of a $C^*$-algebra on a Hilbert space are approximately unitarily equivalent (and moreover the approximate unitary equivalence can be taken relative to the compacts).  Hadwin was working on related topics in his PhD thesis (\cite{Hadwin:thesis}) obtaining the following version of the uniqueness statement in terms of ranks (\cite[Theorem 2.5]{Hadwin:TAMS}).

\begin{theorem*}[Hadwin's formulation of Voiculescu's theorem] Let $A$ be a $C^*$-algebra and let $\mathcal H$ be a Hilbert space. Let $\phi,\psi\colon A\to\mathcal B(\mathcal H)$ be non-degenerate $^*$-homomorphisms. Then $\phi$ and $\psi$ are approximately unitarily equivalent if and only if $\phi(a)$ and $\psi(a)$ have the same rank in $\mathcal B(\mathcal H)$ for all $a\in A$.
\end{theorem*}

The building blocks of von Neumann algebras are \emph{factors} (those von Neumann algebras with trivial centre). Murray and von Neumann's fundamental work divides the factors into types based on the behaviour of their projections. The type I factors are those with minimal non-zero projections. These are all of the form $\mathcal B(\mathcal H)$ for some Hilbert space $\mathcal H$. Type II$_1$ factors have the property that all projections are finite, but have no minimal projections. These have a continuous dimension range measured by a trace. The II$_\infty$ factors have both infinite and finite projections, but again no minimal projections. These are always von Neumann algebra tensor products of a type II$_1$ factor with an infinite type I factor. Finally, the type III factors are characterised by the property that all non-zero projections are infinite (and equivalent in the separable predual case). This paper is in a line of work instigated by Ding and Hadwin in their 2005 article \cite{DingHadwin} (which contains the results from Ding's 1993 PhD thesis \cite{Ding:Thesis} supervised by Hadwin), which asks what happens when $\mathcal B(\mathcal H)$ is replaced by another von Neumann algebra $\mathcal M$, and in particular with a II$_1$ factor?  A special case of our main result is as follows (the hypotheses, particularly the universal coefficient theorem (UCT), will be discussed later in the introduction). This will be proved as Theorem \ref{II1_factor}. 

\begin{mainthm} 
\label{II1_factor:intro} 
Let $A$ be a separable, unital and nuclear $C^*$-algebra satisfying Rosenberg and Schoquet's universal coefficient theorem and let $\calM$ be a II$_{1}$-factor with trace $\tau_{\calM}$. Let $\phi, \psi: A\rightarrow \calM$ be unital faithful $^*$-homomorphisms such that $\tau_{\calM} \circ \phi = \tau_{\calM} \circ \psi$, then there exists a sequence of unitaries $(u_{n})_{n}$ in $\calM$ such that 
\begin{equation}\label{II1_factor.1} 
\|u_n \phi(a)u_n^*-\psi(a)\|\to 0, \quad a\in A. 
\end{equation} 
\end{mainthm} 

The previous state of the art result in this direction is the following theorem of Hadwin, Li and Liu (which is covered by \cite[Theorem 5]{LHL:OM}) building on Ding and Hadwin's result for approximately homogeneous (AH) algebras from \cite{DingHadwin}.  This encompasses domains which are inductive limits of type I $C^*$-algebras (those whose bidual is a type I von Neumann algebra). Every subhomogeneous $C^*$-algebra (those whose irreducible representations have uniformly bounded finite dimension) is type I, and hence the class of approximately subhomogeneous (ASH) algebras consisting of inductive limits of subhomogenous $C^*$-algebras is covered by the Hadwin--Li--Liu theorem.

\begin{theorem*}[Ding--Hadwin, Hadwin--Li--Liu; Uniqueness for maps from inductive limits of type I $C^*$-algebras to II$_1$ factors]
    Let $A$ be a separable and unital $C^*$-algebra which is an inductive limit of type I $C^*$-algebras (such as an ASH algebra) and let $\mathcal M$ be a II$_1$-factor.  Any two unital $^*$-homomorphisms $\phi,\psi\colon A\to\mathcal M$ with $\tau_\mathcal M\circ\phi=\tau_\mathcal M\circ\psi$ are approximately unitarily equivalent in the norm topology. 
\end{theorem*} 

The differences between this result and our Theorem \ref{II1_factor:intro} parallel the developments on the stably finite side of Elliott's classification programme for simple separable nuclear $C^*$-algebras over the last 30 years. Both the ASH hypothesis and the UCT can be viewed as approximation conditions, the former by means of concrete internal structure, while the latter lives at the level of $KK$-theory. 
Indeed, a $C^*$-algebra satisfies the UCT precisely when it is equivalent in the $KK$-category to a commutative $C^*$-algebra.  It is a major open question going back to \cite{RS:UCT} whether all separable and nuclear $C^*$-algebras satisfy the UCT, and it is also open whether all separable and nuclear $C^*$-algebras $A$ with the property that all ideal-quotients have a densely defined trace are ASH.  However, Tu's theorem that $C^*$-algebras associated to amenable \'etale groupoids satisfy the UCT (building on Higson and Kasparov's work on the Baum--Connes conjecture in \cite{HK:Invent}) provides a powerful tool for obtaining the UCT in examples (which has subsequently been extended to twisted groupoids, and hence nuclear $C^*$-algebras containing Cartan masas in \cite{BarlakLi}).  So in concrete examples it is usually possible to verify the UCT in practice. Indeed to our knowledge there are no known explicit examples of separable nuclear $C^*$-algebras for which the UCT is not known to hold.

The other difference is that Theorem \ref{II1_factor:intro} requires injective $^*$-homomorphisms and Hadwin--Li--Liu's ASH theorem does not. Injectivity is crucial in the proof of Theorem \ref{II1_factor:intro}, and we can not just pass to the quotient of $A$ by the common kernel of the $^*$-homomorphisms, since the UCT does not generally pass to quotients (whether it does for nuclear $C^*$-algebras is equivalent to the UCT problem). In contrast the class of algebras covered by the Hadwin--Li--Liu theorem is closed under taking quotients and so it can be recaptured from our Theorem \ref{II1_factor:intro} (see Corollary \ref{cor:HLL_reproof}). 

We next discuss what is known outside the $\mathcal B(H)$ and type II$_1$ factor setting, and the role of amenability in these uniqueness theorems. Given non-degenerate $^*$-homomorphisms $\phi,\psi\colon A\to \mathcal M$, from a $C^*$-algebra $A$ into a von Neumann algebra $\mathcal M$, say that $\phi$ and $\psi$ have \emph{the same $\mathcal M$-rank}, when for each $a\in A$, the range projections of $\phi(a)$ and $\psi(a)$ are Murray--von Neumann equivalent. In the case when $\mathcal M$ is a II$_1$-factor, $\phi$ and $\psi$ have the same $\mathcal M$-rank precisely when they agree on trace, i.e. $\tau_{\mathcal M}\circ\phi=\tau_{\mathcal M}\circ\psi$. Ding and Hadwin asked when having the same $\mathcal M$-rank gives rise to norm approximate unitary equivalence? 

For a separably acting type III factor $\mathcal M$, as all non-zero projections are equivalent, two maps $A\to\mathcal M$ have the same $\mathcal M$-rank precisely when they have the same kernel.  In this setting one can obtain a uniqueness theorem using a sledgehammer: every type III factor is a simple purely infinite $C^*$-algebra and Kirchberg's classification theorems apply! We will briefly explain how to obtain the following result in Appendix \ref{Appendix_A}.

\begin{theorem*}[A consequence of Gabe's treatment of Kirchberg's work]
Let $A$ be a separable, exact and unital $C^*$-algebra and let $\mathcal M$ be a type III factor with separable predual.  Then any two unital and nuclear $^*$-homomorphisms $A\to\mathcal M$ are approximately unitarily equivalent if and only if they have the same kernel.
\end{theorem*} 

Analogous to the usual statement of Voiculescu's theorem in terms of essential representations, there are also results for \emph{full} maps from nuclear $C^*$-algebras to type II$_\infty$ factors.\footnote{A map is full, if the image of every non-zero element of the domain, generates the codomain as an ideal. A separably acting semifinite factor $\mathcal M$ has a unique non-trivial ideal -- the \emph{Brauer ideal} -- generated by all the projections of finite trace. In the case of $\mathcal B(\mathcal H)$, this ideal is just the compacts. Accordingly a full map $A\to\mathcal B(\mathcal H)$ is an injective essential representation.} The following theorem follows from results of Li, Shen and Shi (\cite[Theorem~5.2.2]{Li_Shen_Shi}); see also the more explicit formulation of Hadwin, Li and Liu (\cite[Theorem~3]{LHL:OM}). 

\begin{theorem*}[Li--Shen--Shi; Uniqueness for full maps into II$_{\infty}$-factors] 
Let $A$ be a unital, separable and nuclear $C^*$-algebra and $\calM$ be a II$_{\infty}$-factor with separable predual. If $\phi, \psi$ are unital and full $^*$-homomorphisms, then they are approximately unitarily equivalent. 
\end{theorem*} 

Outside the setting of full maps, Hadwin, Li and Liu obtained a uniqueness theorem for ASH domains (\cite[Theorem 8]{LHL:OM}).

\begin{theorem*}[Hadwin--Li--Liu; Uniqueness for maps into from ASH-algebras to II$_{\infty}$-factors] 
Let $A$ be a separable and unital ASH-algebra, and let $\calM$ be a II$_{\infty}$-factor with separable predual. Then any two unital $^*$-homomorphisms $\phi, \psi: A\rightarrow \calM$ are approximately unitarily equivalent if and only if $\phi$ and $\psi$ have the same $\mathcal M$-rank.
\end{theorem*}  

This is obtained by combining the absorption theorems for full maps from \cite{Li_Shen_Shi} with uniqueness results  for maps from ASH-algebras into the Brauer ideal, in the spirit of Voiculescu's original work. In fact, just as with Voiculescu's theorem, they reach the stronger conclusion that the approximate unitary equivalence can be taken relative to the Brauer ideal.  We hope that our work in the II$_1$ setting will open the door to abstract results for maps into type II$_\infty$ factors in the future.

Returning to the setting of II$_1$ factors, Connes' celebrated equivalence of injectivity and hyperfiniteness (\cite{Connes:Annals}) gives the following uniqueness theorem with the weaker conclusion of approximate unitary equivalence in the $2$-norm coming from the trace.  

\begin{theorem*}[A folklore consequence of Connes' theorem]\label{Connes} 
Let $A$ be a separable, unital and nuclear $C^*$-algebra and let $\calM$ be a II$_{1}$-factor. Let $\phi, \psi: A\rightarrow \calM$ be unital $^*$-homomorphisms such that $\tau_{\calM} \circ \phi = \tau_{\calM} \circ \psi$, then there exists a sequence of unitaries $(u_{n})_n$ in $\calM$ such that 
\begin{equation} \label{Intro.Connes}
\|u_n \phi(a)u_n^*-\psi(a)\|_{2,\tau_{\calM}}\to 0, \quad a\in A. 
\end{equation} 
\end{theorem*}
The point is that the finite part $A^{**}_{\text{fin}}$ of the bidual of a nuclear $C^*$-algebra is injective, so hyperfinite.  This allows maps from a nuclear $C^*$-algebra into $\mathcal M$ to be approximated (in $\|\cdot\|_2$) by maps from finite dimensional algebras, where one immediately has uniqueness. Ding and Hadwin's idea for commutative $C^*$-algebras (and then AH-algebras) was of a similar vein: factor through a larger algebra where one can find operator norm finite-dimensional approximations. For example, Ding and Hadwin prove their uniqueness theorem for $A=C([0,1])$, by noting that a map  $C([0,1])\to\mathcal M$ extends to the $C^*$-algebra $B([0,1])$ of bounded Borel functions on $[0,1]$, and step functions can be used to approximate continuous functions in norm. By contrast, in our proof of Theorem \ref{II1_factor:intro}, the only internal approximations used are those coming from Connes' theorem.
 
Some amenability is necessary for uniqueness theorems.  To our knowledge, this was first noticed by Hadwin who used free entropy methods to give examples of pairs of $^*$-homomorphisms from certain non-nuclear separable and tracial $C^*$-algebras (those with a family of generators satisfying a certain free entropy dimension ratio condition, for example the reduced group $C^*$-algebras of free groups) into II$_1$ factors which agree on traces, but are not $2$-norm approximately unitarily equivalent (\cite[Corollary 3.5]{Hadwin:FreeEntropy}). This can be viewed as a precursor of Jung's result (\cite{Jung:MA}), that for a Connes embeddable II$_1$-factor $\mathcal N$, uniqueness of embeddings $\mathcal N\to\mathcal R^\omega$ up to unitary conjugacy implies that $\mathcal N$ is semidiscrete. Hadwin's work also contains one of the earliest versions of the folklore consequences of Connes' theorem we are aware of (see \cite[Theorem 2.1]{Hadwin:FreeEntropy}). More generally, Ciuperca, Giordano, Ng and Niu's work on these `weak$^*$-uniqueness' type problems further demonstrates the role of amenability.

\begin{theorem*}[Ciuperca, Giordano, Ng and Niu] Let $A$ be a $C^*$-algebra.
\begin{enumerate}
    \item
    Let $\mathcal M$ be a von Neumann algebra, and suppose that either $A$ is nuclear, or $\mathcal M$ is injective. Then any two non-degenerate $^*$-homomorphisms $\phi,\psi\colon A\to\mathcal M$ with the same $\mathcal M$-rank are weak$^*$-approximately unitarily equivalent in the sense of \cite[Definition 1.1(c)]{CGNN:Adv}.\footnote{In the infinite setting, this means that there are unitaries $(u_n)_{n}$ and $(v_n)_{n}$ with $u_n\phi(a)u_n^*\to \psi(a)$ and $v_n\psi(a)v_n^*\to \phi(a)$ in the weak$^*$ topology, for all $a\in A$. One can not guarantee that $v_n=u_n^*$ for infinite von Neumann algebras.} 
    \item Suppose that $A$ is separable and has the property that for all von Neumann algebras $\mathcal M$, any two non-degenerate $^*$-homomorphisms into $\mathcal M$ with the same $\mathcal M$-rank are weak$^*$-approximately unitarily equivalent.  Then $A$ is nuclear.
    \end{enumerate}
\end{theorem*}

\subsection*{Uniqueness of quasidiagonal approximations}

While we stated Theorem \ref{II1_factor:intro} as a uniqueness theorem for $^*$-homomorphisms into II$_1$ factors, our main theorem applies more broadly to sequences of maps which become $^*$-homomorphisms in the limit. These are best encoded using ultraproducts (we will recall the definitions in Section \ref{Sect:sep}), and the version of our main theorem in this setting is as follows. The proof will be given in Section \ref{classification_chapter}. 

\begin{mainthm}\label{intro:ultraproductthm}
Let $A$ be a separable, unital and exact $C^*$-algebra satisfying the UCT. Let $(\mathcal M_n)_{n}$ be a sequence of finite von Neumann factors, and let $\mathcal M_\omega=\prod_\omega \mathcal M_n$ denote the $C^*$-algebra ultraproduct which has a unique trace $\tau_\omega$.  Given full, unital, nuclear $^*$-homomorphisms $\phi, \psi: A\rightarrow \mathcal M_\omega$ such that $\tau_{\omega} \circ \phi = \tau_{\omega} \circ \psi$, $K_0(\phi) = K_0(\psi)$, and $K_0(\phi;\mathbb Z/n\mathbb Z)=K_0(\psi;\mathbb Z/n\mathbb Z)$ for all $n\geq 2$, there exists a unitary $u \in \mathcal M_\omega$ such that $\phi = \text{Ad} (u) \circ \psi$. 
\end{mainthm} 

Here $K_0(\cdot;\mathbb Z/n\mathbb Z)$ is the $K_0$-group with coefficients in $\mathbb Z/n\mathbb Z$ as developed by Schochet in \cite{Schochet:PJM}.  When each $\mathcal M_n$ is a II$_1$ factor, we have $K_0(\mathcal M;\mathbb Z/n\mathbb Z)=0$ and we only need to consider $K_0$, and (as we explain in Section \ref{classification_chapter}) Theorem \ref{II1_factor:intro} is a special case of Theorem \ref{intro:ultraproductthm}. 

Our motivation for the generality of allowing the factors in Theorem \ref{intro:ultraproductthm} to vary is to obtain a uniqueness theorem for quasidiagonal approximations. Recall that according to a characterisation due to Voiculescu, a separable $C^*$-algebra $A$ is quasidiagonal if and only if there is a sequence of approximately multiplicative, and approximately isometric completely positive maps $\phi_n\colon A\to M_{k_n}$. Packaging such maps together, one obtains an embedding $\phi\colon A\rightarrow\prod_{\omega}M_{k_n}$ into the norm ultraproduct of these matrix algebras. 

For the existence of such embeddings, every separable and nuclear $C^*$-algebra satisfying the UCT and with a faithful tracial state is quasidiagonal (\cite{TWW:Annals}). This was extended to the setting of exact $C^*$-algebras with the UCT by Gabe in \cite{Gabe:JFA} and Schafhauser gave a new proof of both  results in \cite{new-TWW}. In fact the proofs give nuclear embeddings from exact $C^*$-algebras into $\mathcal Q_\omega$, the ultraproduct of the universal UHF algebra $\mathcal Q$ (which is equivalent to the existence of the embeddings into $\prod_{\omega}M_{k_n}$). Subsequently, in his breakthrough AF-embedding paper, Schafhauser gave a uniqueness counterpart to this form of quasidiagonality theorem.\footnote{Schafhauser's theorem allowed for somewhat more general codomains than $\mathcal Q_\omega$.} 

\begin{theorem*}[Schafhauser's uniqueness of embeddings into $\mathcal Q_\omega$]
    Let $A$ be a separable, unital, exact $C^*$-algebra satisfying the UCT. Then unital, full and nuclear embeddings $\phi,\psi\colon A\to \mathcal Q_\omega$ are approximately unitarily equivalent if and only if $K_0(\phi)=K_0(\psi)$ and \text{$\tau_{\calQ_\omega}\circ\phi=\tau_{\calQ_\omega}\circ \psi$}.
\end{theorem*} 

Schafhauser's theorem does not give uniqueness for maps from $A$ into an ultraproduct $\prod_{\omega} M_{k_n}$. But given two full maps $\phi,\psi:A\to\prod_{\omega} M_{k_n}$ with $\tau_{\omega}\circ\phi=\tau_{\omega}\circ\phi$ (where $\tau_{\omega}$ is the unique trace on this ultraproduct), and $K_0(\phi)=K_0(\psi)$, one can embed each $M_{k_n}$ into $\calQ$ and obtain a unitary conjugacy in $\mathcal Q_\omega$.  Approximating by matrices, one can  find a sequence $(r_n)_{n}$ such that enlarging $M_{k_n}$ to $M_{k_n}\otimes M_{r_n}$, the induced maps $\tilde\phi,\tilde\psi:A\to\prod_{\omega} (M_{k_n}\otimes M_{r_n})$ become unitarily equivalent, i.e. we do get uniqueness after enlarging the matrix sizes.

It is natural to ask for unitary equivalence without enlarging the size of the matrices. For this it is inevitable that we will need to keep track of $\mathbb Z/n\mathbb Z$-coefficients in $K_0$ (since $K_0(\mathcal Q_\omega;\mathbb Z/n\mathbb Z)=0$ for all $n$, this data is not required when the codomain is $\mathcal Q_\omega$). To our knowledge, uniqueness of quasidiagonality results of this form were first obtained by Lin as part of his work classifying $^*$-homomorphisms from unital AH-algebras (which are inductive limit of homogeneous $C^*$-algebras) into unital and simple $C^*$-algebras with good tracial internal approximations; see \cite[Theorem 5.8]{Lin_2017} for AH domains (this also allows for ultraproducts of simple $C^*$-algebras with tracial rank at most $1$ as the codomain, with ultraproducts of matrices being a special case).\footnote{Lin's results are stated for approximately multiplicative maps; an ultraproduct version of Lin's result for commutative domains (\cite[Theorem 2.10]{Lin_2017}, which is a special case of \cite[Theorem 4.6]{Lin07}) that is much more directly comparable with Theorem \ref{intro:ultraproductthm} and Corollary \ref{intro:qdcor} is set out as \cite[Theorem 2.4]{Nuclear_dim_extension}.} Taking the factors $\mathcal M_n$ in Theorem \ref{intro:ultraproductthm} to be matrix algebras, we get the uniqueness of quasidiagonal approximations in the presence of the UCT.  Again this amounts to moving a type I approximation hypothesis from an internal structural condition to the level of $KK$-theory. 

\begin{maincor}\label{quasidiagonality}\label{intro:qdcor}
Let $A$ be a separable, unital, exact $C^*$-algebra satisfying the UCT. Let $\phi, \psi: A\rightarrow \prod_{\omega} M_{k_{n}}$ be unital, full and nuclear $^*$-homomorphisms such that $\tau_{\omega} \circ \phi = \tau_{\omega} \circ \psi$, $K_0(\phi) = K_0(\psi)$, $K_0(\phi;\mathbb Z/n\mathbb Z)=K_0(\psi;\mathbb Z/n\mathbb Z)$ for all $n\geq 2$. Then there exists a unitary $u \in \prod_{\omega} M_{k_{n}}$ such that $\phi = \text{Ad} (u) \circ \psi$. 
\end{maincor}

\subsection*{Classification via lifting along the trace-kernel extension and $KK$-uniqueness theorems} 

Our approach to proving Theorem \ref{II1_factor:intro} follows Schafhauser's pioneering uniqueness of embeddings into $\mathcal Q_\omega$ from \cite{Schafhauser.Annals}.   Schafhauser divided this problem into two pieces by means of the trace-kernel extension:
\begin{equation}\label{intro:TK}
    0\longrightarrow J_{\mathcal Q}\longrightarrow \mathcal Q_\omega\stackrel{q}{\longrightarrow} \mathcal R^\omega\longrightarrow 0.
\end{equation}
Here, the canonical embedding of $\mathcal Q$ into the hyperfinite II$_1$-factor $\mathcal R$ induces a quotient map from $\mathcal Q_\omega$ onto the von Neumann ultraproduct $\mathcal R^\omega$, and the unique trace $\tau_{\calQ_{\omega}}$ descends to $\tau_{\calR_{\omega}}$, i.e. $\tau_{\calR_{\omega}}\circ q = \tau_{\calQ_{\omega}}$.  Given a nuclear $C^*$-algebra $A$ and two unital full $^*$-homomorphisms $\phi,\psi:A\to\mathcal Q_\omega$ with $\tau_{\mathcal Q_{\omega}}\circ\phi=\tau_{\mathcal Q_{\omega}}\circ\psi$, the folklore consequence of Connes' theorem ensures that $q\circ\phi,q\circ\psi:A\to\mathcal R^\omega$ are unitarily equivalent. Moreover, this unitary lifts to $\mathcal Q_\omega$ and so we can assume that $q\circ\phi=q\circ\psi$. Then $(\phi,\psi)$ forms a $(A,J_{\mathcal Q})$-\emph{Cuntz pair} and hence gives a class $[\phi,\psi]$ in $KK(A,J_\calQ)$. There are things to worry about here, notably that the kernel $J_\calQ$ is not $\sigma$-unital and also not stable. But one of Schafhauser's big insights (from his earlier work \cite{new-TWW} reproving the quasidiagonality theorem) was that one can nevertheless productively use $KK$-theory and extension theory by working with suitable separable subextensions of \eqref{intro:TK} (which we will suppress in the rest of this discussion). 

Assuming now that $K_0(\phi)=K_0(\psi)$, the UCT can be used to show that $[\phi,\psi]$ vanishes in $KK(A,J_\calQ)$.  We would like to deduce that there is a path of unitaries $(u_t)_{t\geq 0}$ in the minimal unitisation of $J_B$ with $u_t\phi(a)u_t^*\to \psi(a)$; this is known as a \emph{proper asymptotic unitary equivalence}. Once this happens $\phi$ and $\psi$ will be approximately unitarily equivalent in $\mathcal Q_\omega$, and hence (using reindexing) unitarily equivalent.  However, $[\phi,\psi]=0$ does not directly give a proper asymptotic unitary equivalence.  Instead, Dadarlat--Eilers' \emph{stable uniqueness theorem} (\cite{Dadarlet-Eilers}) gives such a path of unitaries asymptotically conjugating $\phi\oplus\theta$ and $\psi\oplus\theta$, after adding any unitally absorbing map $\theta$. Importantly using the Elliott--Kucerovsky's noncommutative Weyl--von Neumann Theorem from \cite{EK.PJM}, Schafhauser realised that as $\phi$ and $\psi$ are full, either of them can be used as $\theta$, giving a unitary equivalence between $\phi\oplus\phi$ and $\psi\oplus\psi$.\footnote{Nuclearity enters play here. The Elliott--Kucerovsky theorem characterises nuclearly absorbing maps, i.e. those maps absorbing all weakly nuclear maps. In order to be able to take $\theta$ to be $\phi$ and $\psi$ in turn these need to be weakly nuclear as maps into the multiplier of $J_{\mathcal Q}$, which is a consequence of nuclearity of $\phi$ and $\psi$ (which in this outline follows from nuclearity of $A$). In the main body of the paper (and just as in \cite{Schafhauser.Annals}), we work with Skandalis' $KK_{\mathrm{nuc}}(A,J)$ and its Hausdorff quotient $KL_{\mathrm{nuc}}(A,J)$ so as to handle nuclear maps from exact domains.}  Using the fact that $\mathcal Q$ absorbs the CAR-algebra $M_{2^\infty}$ tensorially, Schafhauser removed the unwanted $2\times 2$ matrix amplification to obtain the desired unitary equivalence of $\phi$ and $\psi$. 

The passage from $[\phi,\psi]=0$ in $KK$, together with the unital absorption of $\phi$ and $\psi$, to a proper asymptotic unitary equivalence between $\phi$ and $\psi$ was coined a \emph{$KK$-uniqueness theorem} in \cite{CGSTW}, and we informally say the pair $(A,J)$ has $KK$-uniqueness when this applies.\footnote{As a warning, in this paper we work with unital maps and the notion of unitally absorbing representations, whereas in \cite{Schafhauser.Annals} and \cite{CGSTW}, absorbing representations where used in these arguments. This enables us to avoid some unnecessary unitisation and deunitisation arguments. See Footnote \ref{unital_absorb-fn} and the comments before Theorem \ref{KK_KL_uniqueness}.} This part of Schafhauser's argument works generally for any $\mathcal Q$-stable and stable $KK$-codomain $J$ (see \cite[Proposition 2.7]{Schafhauser.Annals}). In the later abstract approach to the unital classification theorem, Schafhauser's $\mathcal Q$-stable $KK$-uniqueness theorem was replaced by a $\mathcal Z$-stable $KK$-uniqueness theorem (see \cite[Theorem 1.4]{CGSTW} and the discussion there). 

To prove Theorem \ref{II1_factor:intro}, we replace $\mathcal Q$ by a II$_1$-factor $\mathcal M$. The trace-kernel extension
\begin{equation}\label{Intro:TK2}
0\longrightarrow J_{\mathcal M}\longrightarrow \mathcal M_\omega\stackrel{q}{\longrightarrow} \mathcal M^\omega \longrightarrow 0
\end{equation}
arises from the natural quotient of the $C^*$-ultrapower $\mathcal M_\omega$ onto the tracial von Neumann algebra ultrapower $\mathcal M^\omega$. Schafhauser's approach readily adapts to this setting and the only missing step is a $KK$-uniqueness theorem. We can not use the $\Z$-stable $KK$-uniqueness theorem, as not only does $\mathcal M$ fail to be $\mathcal Z$-stable (no II$_1$ factor can be a tensor product of two infinite dimensional $C^*$-algebras; see \cite{Ghasemi}), but in general $\mathcal M$ will not even be separably $\mathcal Z$-stable\footnote{Separable $\Z$-stability is the appropriate version of $\Z$-stability for non-separable $C^*$-algebras: $D$ is separably $\Z$-stable if every separable $D_0\subseteq D$ is contained in a $\Z$-stable separable $C^*$-subalgebra $D_1$ of $D$.}  (for example if $\mathcal M$ fails to have property $\Gamma$). While there are II$_1$-factors which are separably $\mathcal Z$-stable (such as $\mathcal R^\omega$), to use $\Z$-stable $KK$-uniqueness to prove Theorem \ref{II1_factor:intro} for $\mathcal M=\mathcal R$, we would need $\mathcal R$ (or equivalently its norm ultraproduct $\mathcal R_\omega$) to be separably $\Z$-stable, and this is open (see \cite[Problem XCIV]{99Q}). Accordingly the main technical goal of the paper is to prove a $KK$ and $KL$-uniqueness theorem designed for use in Theorem \ref{II1_factor:intro} (where $J$ will be taken to be a suitable separable subalgebra of the trace-kernel ideal $J_{\mathcal M}$ from \eqref{Intro:TK2}).  Once we have the $KK$-uniqueness theorem in place, the rest of the sketch argument above from \cite{Schafhauser.Annals} works (the details are set out in Section \ref{classification_chapter}).  We state our $KK$-uniqueness result informally here for nuclear domains; in the main body we work with weakly nuclear maps and give $KK_{\mathrm{nuc}}$ and $KL_{\mathrm{nuc}}$-uniqueness theorems. Since our argument is $C^*$-algebraic, the rather eclectic set of hypotheses on the codomain $J$ is designed to collect the relevant $C^*$-algebraic properties of (suitable separabalisations) of $J_{\calM}$.  We will discuss these hypotheses and their consequences in Section \ref{Sect:sep}.

\begin{mainthm}[See Theorem \ref{KL-uniqueness}] 
\label{Intro:KKUnique}
Let $A$ be a separable, unital and nuclear $C^*$-algebra, and $J$ be a separable and stable $C^*$-algebra which has real rank zero, stable rank one, $K_1(J)=0$ and a totally ordered Murray and von Neumann semigroup. Then the pair $(A,J)$ has $KK$-uniqueness.
\end{mainthm}

These $KK$-uniqueness theorems originate in the work of Dadarlat and Eilers, who in the language we used showed that for any unital separable $A$, the pair $(A,\mathcal K)$ satisfies $KK$-uniqueness where $\mathcal K$ is the compacts (\cite[Theorem 3.12]{Dadarlet-Eilers}). This is powered by asymptotic characterisations of absorption together with the analysis of asymptotically inner derivations and automorphisms going back to \cite{KadisonRingrose} (see \cite[Section 8]{Pedersen:Book}). There are two key ingredients:
\begin{itemize}
    \item 
Paschke duality $KK(A, \mathbb{C}) \cong K_{1}(\calC \cap \overline{\phi}(A)')$, where $\mathcal C$ is the Calkin algebra $\mathcal B(\mathcal H)/\mathcal K(\mathcal H)$ and $\overline{\phi}:A\to\mathcal C$ is induced by a unital and unitally absorbing map $\phi:A\to\mathcal B(\mathcal H)$.
\item $K_1$-injectivity of the Paschke dual algebra $\calC \cap \overline{\phi}(A)'$, i.e. unitaries in this algebra which are trivial in $K_1$ are homotopic to the identity without the need for further matrix amplification.
\end{itemize}

The argument of Dadarlat and Eilers extends to more general $J$, by means of Thomsen's generalisation of Paschke duality (\cite{Thomsen}), whenever the Paschke dual algebra is $K_1$-injective.  This was set out by Loreaux, Ng, and Sutradhar in \cite[Theorem 2.6]{K1-injective} (both in the framework we give here of unital maps which are unitally absorbing, and also for absorbing maps), and in the abstract approach to classification (\cite[Section 5]{CGSTW}), which explicitly raised the $KK$-uniqueness problem (see also \cite[Section 17]{99Q}). Following this overarching strategy, Loreaux, Ng, and Sutradhar obtained $KK$-uniqueness for pairs $(A,J)$ where $A$ is a unital, separable, simple and nuclear $C^*$-algebra and $J$ is separable, simple and stable satisfying strict comparison, with $T(J)$ having finitely many extreme points. Their  $KK$-uniqueness theorem relies on a delicate argument to establish $K_{1}$-injectivity of Paschke dual algebras.

We follow the approach of Loreaux, Ng, and Sutradhar to obtain the necessary $K_1$-injectivity, establishing the following  theorem, from which Theorem \ref{Intro:KKUnique} and then Theorem \ref{II1_factor:intro} follow. The major difference is that in our case, $J$ needs to be (a suitable separable subalgebra of) the trace-kernel ideal, and so is inevitably highly non-simple: with ideals coming from different rates at which elements can shrink to zero in trace. In contrast, \cite{K1-injective} handles only simple $J$ with reasonably small and well behaved trace spaces.  (though this is not the main difficulty), we also extend from simple and nuclear $A$ to handle general separable domains. 

\begin{mainthm}[See Corollary \ref{RR0-K1-inj}]  
\label{K1_inj_main_4.2_intro}  
    Let $A$ be a unital, separable $C^*$-algebra. Let $J$ be a separable and stable $C^*$-algebra which has real rank zero, stable rank one, $K_{1}(J) = 0$, which has a totally ordered Murray--von Neumann semigroup. Let ${\phi: A\rightarrow \calM(J)}$ be a unital, full and weakly nuclear $\ast$-homomorphism. Then the Paschke dual algebra $\calC(J)\cap \overline{\phi}(A)'$ is $K_{1}$-injective. 
\end{mainthm} 

We set up an abstract version of this result in Theorem \ref{property_thm} in terms of relatively purely large extensions, which will both capture Theorem \ref{K1_inj_main_4.2_intro} and the unique trace codomain case of \cite[Theorem 3.28]{K1-injective}.  The argument for this is quite involved, and so we defer a detailed outline of how it works until Section \ref{overall_plan_section}.

The abstract condition we use is potentially of independent interest. Given a stable and separable $C^*$-algebra $J$, there has been considerable interest going back to \cite{ideals_multiplier_finite} in finding ideals in the corona algebra of $J$, or equivalently ideals of the multiplier algebra $\calM(J)$ which contain $J$, and then analysing their structure. For example, a densely defined lower semicontinuous trace on $J$ can be extended to a lower semicontinuous (but not densely defined) trace on $\calM(J)$, which in turn produces a \emph{trace-class} ideal $\mathcal T$ generated by all elements of finite trace which lies between $J$ and $\calM(J)$. Work going back to R{\o}rdam (\cite{ideals_multiplier_finite}) shows that when $J$ has a unique densely defined lower semicontinuous trace (up to scaling), and has strict comparison with this trace, the trace class ideal $\mathcal T$ is the unique ideal between $J$ and $\calM(J)$, and furthermore $\calM(J)/\mathcal T$ is simple and purely infinite.  The existence of such an ideal is the abstract condition underpinning our proof of Theorem \ref{K1_inj_main_4.2_intro}, and we show that after passing to a suitable separable subalgebra of the trace-kernel extensions, these large ideals can always be found. In fact, these ideals satisfy the stronger condition of pure largeness. 

\begin{mainthm}[See Theorem \ref{posi_ideal}] 
\label{pure_large_intro} 
Let $J$ be a separable and stable $C^*$-algebra which has real rank zero, stable rank one, $K_{1}(J) = 0$ and has a totally ordered Murray--von Neumann semigroup. Then there exists a maximal proper ideal $\mathcal I$ of $\calM(J)$ containing $J$ with the property that $\calI\lhd \calM(J)$ is purely large, i.e. for any $a\in \calI_{+}$ and $b\in \calM(J)_{+} \setminus \calI$, we have $a\precsim b$ in the Cuntz semigroup. Moreover, $\calM(J)/\mathcal I$ is simple and purely infinite. 
\end{mainthm} 

Given the role separabilisations of the trace-kernel ideal and the corresponding multiplier algebra play in abstract classification, we believe it will be valuable to analyse ideals in the multipliers of these separabilisations more generally, and not just with the restrictive hypotheses here covering trace-kernel ideals from ultraproducts of factors.

The material in this paper is based on the DPhil thesis of the first author \cite{Hua_thesis} supervised by the second author.  In \cite{Hua_thesis} the results are set up for maps out of nuclear $C^*$-algebras, whereas here we work in the more general framework of nuclear maps from exact domains. During the (fairly long) period in which this paper was written, Gabor Szab\'o has informed us that he has obtained very general $KK$ and $KL$-uniqueness theorems, which can be used to obtain our main results as well as other consequences for classification. 

\subsection*{Structure of the paper} 
The preliminary section provides a brief overview of $KK$-theory and unital absorption for maps, ending up with our adaptation of a sufficient condition for $K_{1}$-injectivity of the Paschke dual algebra. Section \ref{Sect:sep} is devoted to our codomain hypotheses: the abstract combination of $C^*$-algebraic properties satisfied by finite factors which power our later arguments.  This section also reviews ultrapowers, the trace-kernel extension, and separabilisation arguments. 

In Section \ref{pure_large_sec}, we introduce relative pure largeness for extensions and prove Theorem \ref{pure_large_intro}. Using this abstract property, Section \ref{K1_Paschke_chapter} establishes our main $K_{1}$-injectivity result (Theorem \ref{K1_inj_main_4.2_intro}), and the corresponding $KK$-uniqueness result (Theorem \ref{Intro:KKUnique}). An outline of the proof of Theorem \ref{K1_inj_main_4.2_intro} is provided in Section \ref{overall_plan_section} to illustrate the argument. 

Section \ref{classification_chapter} proves the main classification theorem (Theorem \ref{intro:ultraproductthm}) by inserting our $KK$-uniqueness theorem suitably in Schafhauser's argument.  We give the details for completeness and record Theorem \ref{II1_factor:intro} as a direct consequence. Finally, the appendix contains a short proof of uniqueness of maps into type III factors. 

\subsection*{Acknowledgements} We thank Jamie Gabe, Chris Schafhauser, Gábor Szabó and Aaron Tikuisis for helpful conversations regarding the material in this paper, and the referee for their careful reading of the paper and helpful suggestions. 

\section{Preliminaries} 
\subsection{Multiplier algebras of $C^*$-algebras with real rank zero}
\label{Multi_rr0_sec}
The $C^*$-algebras appearing as codomains in our work have \emph{real rank zero}: the invertible self-adjoint elements are dense in the self-adjoints, or equivalently, self-adjoint elements of finite spectrum are dense in the self-adjoint elements (\cite[Theorem 2.6]{BrownPedersen:JFA}). Such $C^*$-algebras have an abundance of projections, and thus have a well-understood structure theory. 

We first recall Zhang's description of the ideal lattice of a \emph{multiplier algebra} of a $C^*$-algebra with real rank zero via the Murray--von Neumann semigroup. Given a $C^*$-algebra $D$, the Murray--von Neumann semigroup $V(D)$ is the preordered monoid consisting of equivalence classes of projections in $D\otimes \calK$, with the addition given by $[p]+ [q] = [\mathrm{diag}(p, q)]$. The preorder (given by Murray--von Neumann subequivalence) is algebraic: $x\leq y$ if there exists $z\in V(D)$ with $x+z=y$. To state Zhang's characterisation, and results in Section \ref{pure_large_sec}, we recall the following definitions. 

\begin{defn} 
Let $S$ be a preordered abelian monoid. A nonempty subset $I \subseteq S$ is 
\begin{enumerate} [(i)]
\item \emph{hereditary} if for any $x, y\in S$, such that $x\in I$ and $y\leq x$, we have $y\in I$; 
\item an \emph{order ideal} if it is hereditary and closed under addition. 
\end{enumerate} 
Given nonempty subsets $I_{1}$ and $I_{2}$ of $S$, define 
\begin{equation} \label{addition_subsets} 
I_{1}+I_{2} \coloneqq \{x+y: x\in I_{1}, y\in I_{2}\}. 
\end{equation} 
If $I_{1}$ and $I_{2}$ are order ideals of $S$, then $I_{1}+I_{1}$ is closed under addition. Additionally, if $S$ has the Riesz decomposition property,\footnote{$S$ has the Riesz decomposition property if, whenever $x, y, z\in S$ have $x\leq y+z$, there exists $y', z'\in S$ such that $x = y'+ z'$, $y'\leq y$ and $z'\leq z$.} the sum of two order ideals is hereditary, and is therefore also an order ideal. In this case, the set of order ideals in $S$ forms a lattice with the set intersection and the addition defined as in \eqref{addition_subsets}. 
\end{defn} 

For a $C^*$-algebra of real rank zero, every ideal has real rank zero and so is generated by the projections it contains. This leads to the well-known result that the ideal lattice of a real rank zero $C^*$-algebra is given by the lattice of order ideals in its Murray--von Neumann semigroup. Zhang's characterisation obtains the same result for the multiplier algebra $\mathcal M(D)$ of a real rank zero $C^*$-algebra $D$.\footnote{As it happens, when we apply this in Section \ref{pure_large_sec}, the totality of the hypotheses on the algebras $D$ can be used to show $\mathcal M(D)$ also has real rank zero, but this would use results from (\cite{RR-LHX}) which came after Zhang's characterisation.} 

Note that Zhang's statement was given in terms of the \emph{dimension range} of $\mathcal M(D)$ (the subset of $V(\mathcal M(D))$ realised by classes of projections from $\mathcal M(D)$, with a partially defined addition). When $D$ is stable, $\calM(D)$ contains a copy of $\calB(\calH)$, and hence admits isometries $s_{1}, \cdots, s_{n}$ satisfying $s_{1}s_{1}^* + \cdots + s_{n}s_{n}^{*} = 1$ for every $n\in \mathbb{N}$. These isometries implement canonical isomorphisms $\mathcal M(D)\otimes M_n \cong \mathcal M(D)$ via $a\mapsto V_{n} a V_{n}^*$, where $V_{n} = (s_{1},\cdots, s_{n})\in M_{1, n}(\calM(D))$. As a consequence, every projection in $M_{n}(\calM(D))$ is Murray--von Neumann equivalent to a projection in $\calM(D)$, and hence $V(\mathcal M(D))$ agrees with the dimension range of $\mathcal M(D)$ used by Zhang. 

\begin{prop}[{\cite[Theorem 2.3]{Riesz-dec}}]
\label{ideal_corr} 
Let $D$ be a separable and stable $C^*$-algebra with real rank zero. Then $V(\calM(D))$ has the Riesz decomposition property. Moreover, the lattice of ideals in $\calM(D)$ is isomorphic to the lattice of order ideals in $V(\calM(D))$. The correspondence is given as follows: For an order ideal $\calL \subseteq V(\calM(D))$, 
\begin{equation} \label{ideal-def}
\calI_{\calL} = \overline{\mathrm{Ideal}} \{e\in \calM(D):e \text{ is a projection with } [e]\in \calL\}, 
\end{equation} 
is the corresponding closed ideal in $\calM(J)$ and conversely, $\calL$ is recovered from $\calI$ by 
\begin{equation} \label{order_ideal-def}
\calL_{\calI} = \{[e]\in V(\calM(D)): e\text{ is a projection in } \calI\}. 
\end{equation}  
\end{prop} 

In Zhang's argument, a central ingredient is that every projection in $\calM(D)$ is Murray--von Neumann equivalent to a projection in $\calM(D)$ which is in diagonal form with respect to an approximate unit of projections in $D$ (see \cite[Corollary 2.1]{Riesz-dec}). Representing or approximating elements of such special forms is frequently useful in studying structural properties of multiplier algebras (see \cite{Zhang_RR0_Calkin, RR-LHX, Zhang_1991}, for instance). 

In our setting,  to obtain commutativity with separable subalgebras in $\calM(D)$ in Section \ref{K1_Paschke_chapter}, we will need to work with a quasicentral approximate unit $(e_{n})_{n}$ of $D$ (see Proposition \ref{unitary_commutant} and the paragraph above it). Even in the real rank zero setting, we can not guarantee to find such $(e_{n})_{n}$ consisting of projections, but when $D$ is $\sigma$-unital, we can take $(e_{n})_{n}$ to be \emph{almost idempotent}, meaning that $e_{n+1}e_{n} = e_{n}$ for all $n\in \mathbb{N}$. To approximate elements in multiplier algebras, while maintaining commutativity, we make heavy use of the more general notion of bidiagonal series. For a sequence of elements $(a_{n})_{n}$ in $D$, we say that the series $\sum_{n=1}^{\infty} a_{n}$ is \emph{bidiagonal} in $\calM(D)$ if the series converges in the strict topology and $a_{n} a_{m} = 0$ for $|n-m|\geq 2$. The following standard lemma provides a method of forming bidiagonal series with respect to $(e_{n})_{n}$. The proof resembles that for diagonal series in Zhang's work (\cite[Proposition 1.7]{Zhang_1991}) and is relatively standard, and a detailed proof can be found in first author's thesis  as \cite[Lemma 2.9.11]{Hua_thesis}.

\begin{lem} \label{almost_ortho_approx}
Let $D$ be a non-unital $C^*$-algebra with an almost idempotent approximate unit $(e_{n})_{n}$. Let $(n_{k,1})_{k}$, $(n_{k,2})_{k}$ be strictly increasing sequences of natural numbers, and let $(a_{k})_{k}$ be a bounded sequence with $a_{k}\in \overline{(e_{n_{k, 1}} - e_{n_{k-1, 1}})D(e_{n_{k, 2}} - e_{n_{k-1, 2}})}$. Then $a = \sum_{k=1}^{\infty}a_{k}$ is a bidiagonal series in $\calM(D)$, and if $(a_{k})_{k}$ is bounded by some $M$, then $\|a\|\leq  2M$. 
\end{lem} 

\subsection{Cuntz equivalence} 
We will occasionally use results regarding Cuntz comparison of positive elements and the Cuntz semigroup. Let $D_{+}$ be the set of positive elements in a $C^*$-algebra $D$. For $a, b\in D_{+}$, $a$ is said to be \emph{Cuntz subequivalent} to $b$ in $D$, denoted by $a\precsim b$, if there exists a sequence $(x_{n})_{n}$ in $D$ such that $x_{n}bx_{n}\rightarrow a$. Then $a, b$ are \emph{Cuntz equivalent} if $a\precsim b$ and $b\precsim a$.  The \emph{Cuntz semigroup} $\mathrm{Cu}(D)$ consists of Cuntz equivalence classes of positive elements in $D\otimes \calK$, and for any $a\in (D\otimes \calK)_{+}$, we denote its Cuntz equivalence class by $\langle a \rangle$. Then $\mathrm{Cu}(D)$ is an abelian partially ordered monoid, with the addition $\langle a \rangle + \langle b \rangle = \langle \mathrm{diag}(a, b)\rangle$ and the partial order given by Cuntz subequivalence. 

For projections $p, q$ in $D\otimes \calK$, Murray--von Neumann subequivalence and Cuntz subequivalence coincide, and for this reason, we use the notation $p\precsim q$ for both subequivalences. But in general, Murray--von Neumann equivalence and Cuntz equivalence are not the same,\footnote{For every $C^*$-algebra $D$, there is a natural monoid map $V(D)\rightarrow \mathrm{Cu}(D)$ given by $[p]\mapsto \langle p\rangle$ for any projection $p\in D\otimes \calK$. It is an order embedding if $D$ is stably finite or has cancellation of projections, but this fails in general, as $p\precsim q$ and $q\precsim p$ does not imply $p,q$ are Murray--von Neumann equivalent.} and thus we use different notation: $[p]$ for the Murray--von Neumannn equivalence class in $V(D)$ and $\langle p\rangle$ for the Cuntz equivalence class in $\mathrm{Cu}(D)$. 

Every increasing sequence in $\mathrm{Cu}(D)$ has a supremum (c.f. \cite[Theorem 3.8]{modern-cuntz}). The reason the Cuntz semigroup plays a limited role in this paper is that we are primarily concerned with the real rank zero setting, when information about the Cuntz semigroup is entirely carried by projections.  The following theorem was first obtained as \cite[Theorem 2.8]{Perera_RR0_SR1} when $D$ was additionally assumed to have stable rank one, and the general case is \cite[Theorem 5.7]{RR0_Cuntz}. 

\begin{prop}[{\cite[Theorem 5.7]{RR0_Cuntz}}] 
\label{RR0_Cu} 
Let $D$ be a $\sigma$-unital $C^*$-algebra with real rank zero. For any $a\in (D\otimes \calK)_{+}$, there exists a sequence of projections $(p_{n})_{n}$ in $D\otimes \calK$ such that 
\begin{equation} 
\langle a\rangle = \sup_{n} \langle p_{n} \rangle \in \mathrm{Cu}(D). 
\end{equation}  
\end{prop} 

\subsection{$KK$-theory} 
As is standard for applications to $C^*$-algebra classification, we work with the Cuntz pair picture of Kasparov's $KK$-theory which goes back to \cite{Cuntz:GeneralizedHM}. Here we give a short overview to set out notation and discuss Skandalis' $KK_{\mathrm{nuc}}$, and the Hausdorff quotients $KL$ and $KL_{\mathrm{nuc}}$. The standard references \cite{Elements_KK, Blackadar_K} contain full details (the latter shows how the Cuntz pair picture reduces to the standard Fredholm module picture, and the former both compares to other pictures and also expands \cite{Thomsen:Homotopy} to develop bifunctoriality and Kasparov product directly in the Cuntz pair picture).

\begin{defn} 
\label{defn_KK} 
Let $A$ be a separable $C^*$-algebra, and let $J$ be a $\sigma$-unital and stable $C^*$-algebra.  An \emph{$(A, J)$-Cuntz pair} is a pair of $^*$-homomorphisms $\phi, \psi: A\rightarrow E$, for some $C^*$-algebra $E$ containing $J$ as an ideal, such that $\phi(a) - \psi(a)\in J$ for all $a\in A$. We will denote such a Cuntz pair by $(\phi, \psi): A\rightrightarrows E\rhd J$.   Note that as $J\lhd E$, there is a canonical map $E\to\mathcal M(J)$ (which preserves $J$), and so from a Cuntz pair taking values in $E$, we can naturally construct one taking values in $\mathcal M(J)$. 

Cuntz pairs $(\phi_{0}, \psi_{0}), (\phi_{1}, \psi_{1}): A\rightrightarrows \calM(J)\rhd J$ are \emph{homotopic} if there exists a Cuntz pair 
\begin{equation} 
(\Phi, \Psi): A\rightrightarrows \calM(C([0,1],J)) \cong C_{\sigma} ([0,1], \calM(J)) \rhd C([0,1], J),\footnote{Here $C_{\sigma} ([0,1], \calM(J))$ is the $C^*$-algebra of norm bounded and strictly continuous functions from $[0,1]$ to $\mathcal M(J)$.} 
\end{equation}  
such that the evaluations at $0$ and $1$ produce $(\phi_{0}, \psi_{0})$ and $(\phi_{1}, \psi_{1})$ respectively. Then $KK(A, J)$ is defined to be the homotopy equivalence classes of $(A, J)$-Cuntz pairs, where we denote by $[\phi, \psi]_{KK(A, J)}$ the equivalence class of the Cuntz pair $(\phi, \psi): A\rightrightarrows \calM(J)\rhd J$ in $KK(A, J)$. Moreover, given a Cuntz pair $(\phi, \psi): A\rightrightarrows E\rhd J$, we write $[\phi,\psi]_{KK(A,J)}$ for the class obtained by following each of $\phi$ and $\psi$ by the canonical map $E\to \mathcal M(J)$. Every $^*$-homomorphism  $\phi: A\rightarrow J$ gives a class in $KK(A,J)$ corresponding to the Cuntz pair $(\phi, 0)$; we denote this class by $[\phi]_{KK(A, J)}$. 
\end{defn} 

The addition in $KK$ is defined using the usual diagonal direct sum after canonically identifying $\mathcal M(J)$ with $M_2(\mathcal M(J))$ (using stability of $J$). Precisely, fix isometries $s_{1}, s_{2} \in \calM(J)$ such that $s_{1}s_{1}^{*} + s_{2}s_{2}^{*} = 1$. Given $^*$-homomorphisms $\phi_{1}, \phi_{2}: A\rightarrow \calM(J)$, define the \emph{Cuntz sum} of $\phi_{1}$ and $\phi_{2}$ by 
\begin{equation} 
\label{Cuntz sum} 
(\phi_{1} \oplus \phi_{2})(a) = s_{1} \phi_{1}(a) s_{1}^{*} + s_{2} \phi_{2}(a) s_{2}^{*}, \quad \forall a\in A. 
\end{equation}
This induces a well defined addition on $KK(A,J)$ given by\footnote{The $KK$-equivalence class for the sum is 
independent of the choice of isometries (as any two different Cuntz sums are unitarily conjugate, and the unitary group of $\calM(J\otimes \calK)$ is strictly path connected).} 
\begin{equation} 
[\phi_{1}, \psi_{1}]_{KK(A, J)} + [\phi_{2}, \psi_{2}]_{KK(A, J)} = [\phi_{1}\oplus \phi_{2}, \psi_{1}\oplus \psi_{2}]_{KK(A, J)}. 
\end{equation} 
Then $KK(A,J)$ is an abelian group with the addition given by the Cuntz sum and the zero element is represented by $[\phi, \phi]$ for any $^*$-homomorphism $\phi: A \rightarrow \calM(J\otimes \calK)$. 

While $KK$-theory provides a powerful framework to study $^*$-homomorphisms up to homotopy, approximately unitarily $^*$-homomorphisms need not have the same $KK$-class. This is remedied by working with the quotient $KL(A,J)$ of $KK(A,J)$ by the closure of $\{0\}$ in a suitable topology, as set out in \cite{Dadarlat_topo_KK}.

\begin{defn} 
Let $A$ and $J$ be $C^*$-algebra with $A$ separable and $J$ $\sigma$-unital. Define 
\begin{equation} 
Z_{KK(A, J)} = \left\{ KK(A, \text{ev}_{\infty}) (\kappa): 
\begin{array}{l}
\kappa \in KK(A, C(\overline{\mathbb{N}}, J)) \text{ and }\\
KK(A, \text{ev}_{n}) (\kappa) = 0, \forall n\in \mathbb{N}
\end{array}
\right\}, 
\end{equation} 
where $\overline{\mathbb{N}} = \mathbb{N}\cup \{\infty\}$ is the one-point compactification of $\mathbb{N}$ and $\text{ev}_{n}\colon C(\overline{N},J)\to J$ is the evaluation function at $n\in \overline{\mathbb{N}}$. Then one defines 
\begin{equation} 
KL(A,J) = KK(A, J) / Z_{KK(A,J)}. 
\end{equation} 
\end{defn} 

Just as in \cite{Schafhauser.Annals}, when we work with nuclear maps out of exact domains, it is more appropriate to use Skandalis' nuclear version of $KK$-theory, $KK_{\mathrm{nuc}}(A,J)$, from \cite{Skandalis_88} and the corresponding quotient  $KL_{\mathrm{nuc}}(A,J)$ from \cite{CFP_KN}, defined by requiring all maps and homotopies to be weakly nuclear. A $^*$-homomorphism $\varphi: A\rightarrow \calM(J)$ is \emph{weakly nuclear} if for every $b\in J$, the completely positive map $A\rightarrow J$ given by $a\mapsto b\varphi(a) b^{*}$ is nuclear.

\begin{defn} 
Let $A$ be a separable $C^*$-algebra, and let $J$ be a $\sigma$-unital and stable $C^*$-algebra. The abelian group $KK_{\mathrm{nuc}}(A, J)$ is defined as the classes of weakly nuclear Cuntz pairs $A\rightrightarrows \mathcal M(J)\rhd J$ up to weakly nuclear homotopies. Likewise, define 
\begin{equation} 
Z_{KK_{\mathrm{nuc}}(A, J)} = \left\{ KK_{\mathrm{nuc}}(A, \text{ev}_{\infty}) (\kappa): 
\begin{array}{l}
\kappa \in KK_{\mathrm{nuc}}(A, C(\overline{\mathbb{N}}, J)) \text{ and }\\
KK_{\mathrm{nuc}}(A, \text{ev}_{n}) (\kappa) = 0, \forall n\in \mathbb{N}
\end{array}
\right\}, 
\end{equation} 
and $KL_{\mathrm{nuc}}(A,J) = KK_{\mathrm{nuc}}(A, J) / Z_{KK_{\mathrm{nuc}}(A,J)}$. 
\end{defn} 

In the case when $A$ is separable and satisfies the UCT, it is $KK$-equivalent to a commutative $C^*$-algebra (\cite{RS:UCT}), and thus the canonical homomorphism $KK_{\mathrm{nuc}}(A, J)\rightarrow KK(A, J)$ is an isomorphism (\cite{Skandalis_88}) and so induces an isomorphism $KL_{\mathrm{nuc}}(A, J)\cong KL(A, J)$ as explained before \cite[Theorem 8.11]{Gabe:MAMS}. Although the UCT is a standing assumption in our final theorems, we choose to work with $KL_{\mathrm{nuc}}$ as this is the precise invariant which determines uniqueness of lifts from the trace-kernel quotient (from \cite{Schafhauser.Annals}). The UCT only enters to compute this invariant. 

\subsection{Total $K$-theory and the universal coefficient theorem} 
\label{Intro:totalKtheory} 
Our uniqueness theorems use $K_0$ with coefficients in $\mathbb Z/n\mathbb Z$ as well as the usual $K_0$-groups.  These were introduced by Schochet in \cite{Schochet:PJM}, as
\begin{equation}
\label{Defn_total_K} 
K_i(A; \mathbb{Z} / n\mathbb{Z})\coloneqq K_{i}(A\otimes C_n),
\end{equation}
where $C_n$ is any separable nuclear $C^*$-algebra satisfying the UCT with $K_{*}(C_{n}) = (\mathbb{Z} / n\mathbb{Z}, 0)$ (so for example one can take $C_n=\mathcal O_{n+1}$). The \emph{total $K$}-theory functor $\underline{K}$ associates to each $C^*$-algebra $A$ the collection of groups $K_0(A)$, $K_1(A)$, $K_0(A;\mathbb Z/n\mathbb Z)$ and $K_1(A;\mathbb Z/n\mathbb Z)$ for all $n\geq 2$, together with natural connecting maps between these groups, called  \emph{Bockstein maps} (the exact form of these maps is not important to this paper). There are various exact sequences satisfied by the components of total $K$-theory. Of most relevance to us is the exact sequence
\begin{equation} 
\label{eq:bockstein-new}
0 \longrightarrow K_i(A)\otimes  \mathbb{Z} / n\mathbb{Z}\rightarrow  K_i(A; \mathbb{Z} / n\mathbb{Z}) \rightarrow \mathrm{Tor}(K_{1-i}(A), \mathbb{Z} / n\mathbb{Z}) \longrightarrow 0, 
\end{equation} 
for $i=0,1$, where $\mathrm{Tor}(K_{1-i}(A),\mathbb{Z} / n\mathbb{Z})$ consists of those elements $x\in K_{1-i}(A)$ with $nx=0$, from \cite[Proposition 1.8]{Schochet:PJM}. In fact this sequence splits unnaturally (see \cite[Proposition 2.4]{Schochet:PJM}). In particular, in the case of interest to us, \eqref{eq:bockstein-new} shows that total $K$-theory is concentrated in $K_0$ with coefficients, as recorded in the following proposition.
\begin{prop} 
\label{trivial_K1_coef} 
Let $B$ be a $C^*$-algebra such that $K_1(B)=0$ and $K_0(B)$ is torsion free.  Then $K_1(B;\mathbb Z/n\mathbb Z)=0$ for all $n\geq 2$.
\end{prop}

The way the UCT is used in the main uniqueness result of the paper is through Dadarlat and Loring's universal multicoefficient theorem, which computes $KL$ in terms of its behaviour on total $K$-theory.
\begin{thm}[Dadarlat and Loring {\cite{Dadarlet-Loring}}]\label{UMCT}
    Let $A$ be a separable $C^*$-algebra satisfying the UCT. The canonical map\footnote{This descends from the map $KK(A,B)\to\mathrm{Hom}_{\Lambda}(\underline{K}(A),\underline{K}(B))$ given by the Kasparov product} $KL(A,B)\to \mathrm{Hom}_{\Lambda}(\underline{K}(A),\underline{K}(B))$ is an isomorphism for all $\sigma$-unital $C^*$-algebras $B$.
\end{thm}

\subsection{Absorption and uniqueness theorems for $KK$-theory} 
\label{KL_section} 
For a $C^*$-algebra $J$, we denote the corona algebra $\calM(J)/J$ by $\calC(J)$, and write $\pi: \calM(J)\rightarrow \calC(J)$ for the canonical quotient map. For $a\in \calM(J)$, we often denote $\pi(a)$ by $\overline{a}$, and similarly for a $^*$-homomorphism $\phi: A\rightarrow \calM(J)$, we use $\overline{\phi}$ for $\pi \circ \phi$. When $J$ is stable, we can form the direct sum of $\phi_{1}, \phi_{2}: A\rightarrow \calM(J)$ as in \eqref{Cuntz sum}, and likewise, the direct sum of $\overline{\phi}_{1}, \overline{\phi}_{2}: A\rightarrow \calC(J)$ is also defined using isometries from $\calM(J)$.\footnote{Again, the direct sum is independent of the choice of isometries up to unitary equivalence, witnessed by unitaries from $\calM(J)$.} 

The concept of absorption can be traced back to Voiculescu's non-commutative Weyl–von Neumann Theorem (\cite{Voiculescu}) and is critical in the theory of extensions. The notion was later generalized and abstractly studied by Kasparov (\cite{Kasparov_1980}) and Thomsen (\cite{Thomsen}). For the purpose of this paper, we are only concerned with unital and weakly nuclear $^*$-homomorphisms $\phi: A\rightarrow \calM(J)$ and the associated maps $\overline{\phi}:A\to\mathcal C(J)$. For this reason the right version of absorption for us consists of the \emph{unitally nuclearly absorbing maps} as formalised by Elliott and Kurcerovsky (\cite[Definition 5]{EK.PJM}).\footnote{\label{unital_absorb-fn}Comparing to Definition \ref{unital_nuclear_absorb}, a map is (nuclearly) absorbing when it absorbs all (weakly nuclear) $^*$-homomorphisms. Note that a unital $^*$-homomorphism can never be (nuclearly) absorbing, as it does not absorb the zero map. This is why we need unitally absorbing maps.} 

\begin{defn} 
\label{unital_nuclear_absorb} 
Let $A$ be a unital and separable $C^*$-algebra, and let $J$ be a $\sigma$-unital and stable $C^*$-algebra. 
\begin{enumerate} 
\item A unital $^*$-homomorphism $\theta: A\rightarrow \calC(J)$ is \emph{unitally nuclearly absorbing} if for every unital and weakly nuclear $^*$-homomorphism $\psi: A\rightarrow \calM(J)$, there exists a unitary $u\in \calM(J)$ such that $\mathrm{Ad}(\overline{u}) (\theta\oplus \overline{\psi}) = \theta$; 
\item A unital $^*$-homomorphism $\phi: A\rightarrow \calM(J)$ is \emph{unitally nuclearly absorbing} if $\overline{\phi}$ is unitally nuclearly absorbing.\footnote{Note that this is not the standard definition of unital (nuclear) absorption, which is usually formulated in terms of absorbing every unital (respectively, weakly nuclear) $^*$-homomorphisms up to approximate unitary equivalence modulo $J$. These notions agree (see for example  \cite[Proposition 5.9]{CGSTW}).} 
\end{enumerate} 
\end{defn} 

In the fundamental case $J = \calK$, every unital $^*$-homomorphism $A\rightarrow \calM(\calK)$ is weakly nuclear since $\calK$ is nuclear. Consequently, Voiculescu’s non-commutative Weyl–von Neumann Theorem implies that a unital extension $\theta: A\rightarrow \calC(\calK)$ is unitally nuclearly absorbing if and only if it is faithful (\cite{Voiculescu}; see also \cite{Arveson:DJM}). For general separable and stable $J$, Elliott and Kucerovsky's abstract version of Voiculescu's theorem characterises unitally nuclear absorbing maps, in terms of a condition (\emph{pure largeness}) which in hindsight turns out to live at the level of Cuntz comparison of positive elements as we will discuss in Section \ref{pure_large_sec}. The precise definition is not important to us at this point, as whenever we use the Elloitt--Kucerovsky theorem, we will always do so in the presence of Kucerovsky and Ng's corona factorisation property, a mild regularity condition which allows pure largeness to be conveniently verified.  A  $\sigma$-unital $C^*$-algebra $J$ is said to have the \emph{corona factorisation property} if every full projection in $\calM(J\otimes \calK)$ is properly infinite (see \cite[Definition 2.1]{CFP_KN}). In our setting, the corona factorisation property will follow from the combination of real rank zero and the total ordering of the Murray--von Neumann semigroup  (see Lemma \ref{separableideals}).  In the presence of the corona factorisation property, the Elliott--Kucerovsky theorem takes the following form (recall that a $^*$-homomorphism $\varphi: A\rightarrow B$ is \emph{full} if $\varphi(a)$ is full in $B$ for any nonzero positive element $a\in A$).

\begin{thm}[Elliott--Kucerovsky--Ng, {\cite{EK.PJM, CFP_KN}}] 
\label{nuc_absorb_equiv} 
Let $A$ be a unital and separable $C^*$-algebra, and let $J$ be a $\sigma$-unital and stable $C^*$-algebra satisfying the corona factorisation property. A unital $^*$-homomorphism $\theta: A\rightarrow \calC(J)$ (or $\phi: A\rightarrow \calM(J)$) is unitally nuclearly absorbing if and only if it is full. 
\end{thm}  

\begin{proof} 
When $J$ has the corona factorisation property, it is shown in the proof of \cite[Theorem 3.5]{CFP_KN} that every full extension is purely large. Thus, $\theta$ (or $\overline{\phi}$) is purely large. Unital nuclear absorption of $\theta$ (or $\overline{\phi}$, and thus $\phi$) follows from \cite[Theorem 6]{EK.PJM}. 

The other implication is straightforward, and we include a short proof for completeness. Let $\Phi: A\rightarrow \calB(\calH)$ be a unital essential representation. We obtain a unital and full map 
\begin{equation} 
\label{full_absorb_map_prototype} 
\psi: A\rightarrow B(\calH)\otimes \calM(J) \subseteq \calM(J), \quad a\mapsto \Phi(a)\otimes 1_{\calM(J)}, 
\end{equation} 
which is weakly nuclear (as set out in the proof of \cite[Lemma 12]{EK.PJM}).  Since $\theta$ (or $\phi$) is unitally nuclearly absorbing, then $\theta$ (or $\overline{\phi}$) is unitary conjugate to the full map $\theta \oplus \overline{\psi}$ (or $\overline{\phi}\oplus \overline{\psi}$), which implies the fullness of $\theta$ (or $\phi$). 
\end{proof} 

As discussed in the introduction, the work of Dadarlat and Eilers in \cite{Dadarlet-Eilers} gives rise to $KK$-uniqueness theorems in the presence of $K_1$-injectivity (as set out in \cite{ess-codim, K1-injective, CGSTW}). In the weakly nuclear and unitally absorbing setting, the main result is Theorem \ref{KK_KL_uniqueness} below. This is identical to the corresponding results for $KK$ and $KL$ in which $\phi$ and $\psi$ are unitally absorbing (i.e. they absorb all unital representations, not just the weakly nuclear ones), and are not assumed weakly nuclear. The $KK$-version of this can be found as \cite[Theorem 2.6]{K1-injective}, an account is also given as \cite[Theorem 4.3.5]{Hua_thesis}. See also the discussion of these results between Questions 5.17 and 5.18 of \cite{CGSTW}, and in \cite[Section 17]{99Q}. Recall that a unital $C^*$-algebra $D$ is \emph{$K_{1}$-injective} if the canonical
group homomorphism $\pi_{0}(\calU(D))\rightarrow K_{1}(D)$, $[u]\mapsto [u]_{1}$ is injective. 

\begin{thm} 
\label{KK_KL_uniqueness} 
Let $A$ be a separable and unital $C^*$-algebra, and $J$ be a $\sigma$-unital and stable $C^*$-algebra. Let $(\phi, \psi) \colon A\rightrightarrows \calM(J)\rhd J$ be a Cuntz pair where both $\phi$ and $\psi$ are weakly nuclear and unitally nuclearly absorbing. Suppose that $\calC(J) \cap \overline{\phi} (A)'$ is $K_{1}$-injective. 
\begin{enumerate} [(i)]
\item If $[\phi, \psi]_{KK_{\mathrm{nuc}}} = 0$, there exists a norm-continuous path $(u_{t})_{t\geq 0}$ of unitaries in $J^\dagger$ such that 
\begin{equation} 
\|u_t(\phi(a))u_t^*-\psi(a)\|\to 0,\quad a\in A. 
\end{equation} 
\item If $[\phi, \psi]_{KL_{\mathrm{nuc}}} = 0$, there exists a sequence $(u_{n})_{n=1}^{\infty}$ of unitaries in $J^\dagger$ such that 
\begin{equation} 
\|u_n(\phi(a))u_n^*-\psi(a)\|\to 0,\quad a\in A. 
\end{equation} 
\end{enumerate} 
\end{thm} 

The proof follows the corresponding results for $KK$ and $KL$ set out in \cite[Theorem 4.3.5]{Hua_thesis} checking that at all points in the argument, everything that needs to be weakly nuclear is indeed weakly nuclear. These checks are essentially akin to how the $\mathcal Z$-stable $KK_{\mathrm{nuc}}$ and $KL_{\mathrm{nuc}}$-uniqueness theorems in \cite[Section 5.5]{CGSTW} follow the same proofs as the $\mathcal Z$-stable $KK$ and $KL$ uniqueness theorems from \cite[Section 5.4]{CGSTW}, and indeed the argument of \cite[Theorem 4.3.5]{Hua_thesis} is based on the presentation in \cite[Section 5.4]{CGSTW} (but working with unitally absorbing maps, rather than absorbing maps). For this reason \cite{Hua_thesis} avoids the forced unitisation, and deunitisation, which appears in \cite[Section 5.4]{CGSTW}. 

\subsection{Sufficient condition for $K_{1}$-injectivity}  Results on $K_{1}$-injectivity for relative commutants of the form $\calC(J) \cap \overline{\phi}(A)'$ trace back to Paschke’s original work in the case $J = \calK$ (see \cite[Lemma 3]{Paschke}). More recently, Loreaux, Ng, and Sutradhar generalized Paschke’s argument in \cite[Theorem 2.15]{K1-injective}, making crucial use of the Elliott--Kucerovsky theorem, to give a sufficient condition for $K_{1}$-injectivity of the relative commutant $\overline{\phi}(A)'\cap\mathcal C(J)$ of unital and unitally absorbing maps $\phi$ from simple nuclear $C^*$-algebras $A$. Such relative commutant is called the \emph{Paschke dual algebra} for a unital and unitally absorbing map $\phi$, as there is a canonical isomorphism $KK(A, J) \cong K_{1}(\overline{\phi}(A)'\cap\mathcal C(J))$ (see \cite{Paschke} for $J = \calK$, \cite{Thomsen} for Paschke duality for absorbing maps and \cite[Proposition 2.5]{K1-injective} for unitally absorbing maps).  

We extend the result of Loreaux, Ng, and Sutradhar  to drop the simplicity requirement, and transfer the nuclearity hypothesis to the map.  The proof of the following lemma is very similar to the original from \cite[Theorem 2.15]{K1-injective}, but we give the details for completeness. We denote the set of unitaries in $M_{2}(D)$ by $\calU_{2}(D)$. 

\begin{lem} 
\label{sufficient_K1} 
Let $A$ be a unital and separable $C^*$-algebra, and let $J$ be a separable and stable $C^*$-algebra satisfying the corona factorisation property. Let $\phi: A\rightarrow \calM(J)$ be a unital, full and weakly nuclear $^*$-homomorphism. Suppose that every unitary in $\calC(J)\cap \overline{\phi}(A)'$ is homotopic through unitaries in $\calC(J)\cap \overline{\phi}(A)'$ to a unitary $u$ such that the inclusion map $\iota: C^*(\overline{\phi}(A), u) \rightarrow \calC(J)$ is full. Then $\calC(J) \cap \overline{\phi}(A)'$ is $K_{1}$-injective. 
\end{lem} 

\begin{proof} 
By Theorem \ref{nuc_absorb_equiv}, the unital and full map $\phi$ is unitally nuclearly absorbing. It is then standard that $\calC(J) \cap \overline{\phi}(A)'$ is properly infinite (see \cite[Proposition 6.1.1]{Hua_thesis} for a proof in the case that $\phi$ is unitally absorbing, and again the proof works in just the same way when $\phi$ is weakly nuclear and unitally nuclearly absorbing). As such, by \cite[Proposition 5.1]{properly_infinite_C(X)}, we only need to show that for a fixed unitary $u_{0}\in \calC(J) \cap \overline{\phi}(A)'$ satisfying 
\begin{equation} 
\label{2_homotopy_3} 
\begin{pmatrix}
u_{0} & 0 \\
0 & 1
\end{pmatrix} 
\sim_{h}
\begin{pmatrix}
1 & 0 \\
0 & 1
\end{pmatrix} 
\end{equation} 
in $\calU_{2} (\calC(J) \cap \overline{\phi}(A)')$, we have $u_{0}\sim_{h}1$ in $\calU(\calC(J) \cap \overline{\phi}(A)')$. By hypothesis, $u_{0}$ is homotopic in $\calC(J) \cap \overline{\phi}(A)'$ to a unitary $u$ with the property that the inclusion $\iota: C^{\ast} (\overline{\phi}(A), u) \rightarrow \calC(J)$ is full. In particular, $u$ satisfies \eqref{2_homotopy_3} and it suffices to show that $u\sim_{h} 1$ in $\calU(\calC(J) \cap \overline{\phi}(A)')$. 

We claim that there exists a unital and weakly nuclear $^*$-homomorphism $\psi: C^{*}(\overline{\phi}(A), u)\rightarrow \calM(J)$ such that $\overline{\psi}(\overline{\phi}(a)) = \overline{\phi}(a)$ for any $a\in A$ and $\overline{\psi}(u)\sim_{h}1$ in $\calU(\calC(J) \cap \overline{\phi}(A)')$.  To see this, take any unital essential representation $\Psi: C^{*}(\overline{\phi}(A), u)\rightarrow \calB(\calH)$, which gives arise to a unital and full $^*$-homomorphism $\psi_{0}: C^{*}(\overline{\phi}(A), u)\rightarrow \calM(J)$ as defined in \eqref{full_absorb_map_prototype}. Then (as noted in the proof of \cite[Lemma 12]{EK.PJM}), $\psi_{0}$ is weakly nuclear. Since $u$ commutes with $\overline{\phi}(A)$, it follows that $\Psi(u)$ belongs to the unitary group of the von Neumann algebra $\Psi(\overline{\phi}(A))' \subseteq \calB(\calH)$, which implies that $\Psi(u) \sim_{h}1$ in $\calU(\Psi(\overline{\phi}(A))')$. As $\psi_{0}$ is defined as in \eqref{full_absorb_map_prototype}, we have $\overline{\psi_{0}}(u)\sim_{h}1$ in $\calU(\calC(J)\cap \overline{\psi_{0}}(\overline{\phi}(A))')$.  We will now use absorption to modify $\psi_0$ to obtain a map which fixes $\overline{\phi}(A)$.

Since fullness of $\psi_{0}$ passes to its restriction to $\overline{\phi}(A)$, it follows that $\psi_{0}|_{\overline{\phi}(A)}: \overline{\phi}(A)\rightarrow \calM(J)$ is full and thus unitally nuclearly absorbing by Theorem \ref{nuc_absorb_equiv}. Weak nuclearity of $\psi_{0}$ is also inherited by $\psi_{0}|_{\overline{\phi}(A)}$. By hypothesis $\phi$ is full, so $\phi(A)\cap J = \{0\}$, and so the quotient map $\mathcal M(J)\to\mathcal C(J)$ is injective when restricted to $\phi(A)$. In this way, we obtain a map 
\begin{equation} 
\theta: \overline{\phi}(A) \rightarrow \calM(J), \quad \overline{\phi}(a)\mapsto \phi(a), 
\end{equation}
which factors through $\phi\colon A\to\mathcal M(J)$ so is weakly nuclear and full, and hence nuclearly unitally absorbing  by Theorem \ref{nuc_absorb_equiv}. Then $\psi_{0}|_{\overline{\phi}(A)}$ and $\theta$ are unitarily equivalent modulo $J$, witnessed by some unitary $w\in \calM(J)$. Take $\psi \coloneqq \mathrm{Ad}(w) \psi_{0}$, which is unital, weakly nuclear, and satisfies
\begin{equation} 
\overline{\psi}(\overline{\phi}(a)) = \mathrm{Ad}(\overline{w}) \overline{\psi_{0}}(\overline{\phi}(a)) = \overline{\theta} (\overline{\phi}(a)) = \overline{\phi}(a),\quad a\in A. 
\end{equation} 
Then, applying $\mathrm{Ad}(w)$ to the homotopy $\overline{\psi_{0}}(u)\sim_{h}1$ in $\calU(\calC(J)\cap \overline{\psi_{0}}(\overline{\phi}(A))')$, we get $\overline{\psi}(u)\sim_{h}1$ in $\calU(\calC(J)\cap \overline{\psi}(\overline{\phi}(A)) ') = \calU(\calC(J)\cap \overline{\phi}(A)')$, which proves the claim. 

As $u$ satisfies \eqref{2_homotopy_3} and $\overline{\psi}(u)\sim_{h} 1$ in $\calU(\overline{\phi}(A)'\cap \calC(J))$, we have
\begin{equation} \label{2_homotopy} 
\begin{pmatrix}
u & 0 \\
0 & \overline{\psi}(u) 
\end{pmatrix} 
\sim_{h} 
\begin{pmatrix} 
u & 0 \\
0 & 1
\end{pmatrix} 
\sim_{h}
\begin{pmatrix} 
1 & 0 \\
0 & 1
\end{pmatrix} 
\end{equation} 
in $\calU_{2}(\overline{\phi}(A)'\cap \calC(J))$. The inclusion map $\iota: C^*(\overline{\phi}(A), u) \rightarrow \calC(J)$ is full and thus unitally nuclearly absorbing by Theorem \ref{nuc_absorb_equiv}. Since $\psi$ is weakly nuclear, there exists some unitary $v\in \calM(J)$ such that $\iota = \text{Ad}(\overline{v})\circ (\iota\oplus \overline{\psi})$. Take a $^*$-isomorphism $\Phi: M_{2}\otimes \calC(J) \rightarrow \calC(J)$ with $\Phi(\text{diag}(a, b)) = a\oplus b$ for any $a, b \in \calM(J)$. Then 
\begin{equation} 
\text{Ad}(\overline{v}) \circ \Phi \bigg(
\begin{pmatrix}
u & 0 \\
0 & \overline{\psi}(u) 
\end{pmatrix}
\bigg) = \text{Ad}(\overline{v}) (u \oplus \overline{\psi}(u)) = u, \quad \text{and}
\end{equation} 
\begin{equation} \label{uni_conj_3} 
\text{Ad}(\overline{v}) \circ \Phi \bigg(
\begin{pmatrix}
\overline{\phi}(a) & 0 \\
0 & \overline{\phi}(a)  
\end{pmatrix}
\bigg) = \text{Ad}(\overline{v}) (\overline{\phi}(a)\oplus \overline{\phi}(a)) = \overline{\phi}(a), 
\end{equation} 
where \eqref{uni_conj_3} is true since $\overline{\psi} (\overline{\phi}(a)) =\overline{\phi}(a)$ for any $a\in A$. By applying the map $\text{Ad}(\overline{v}) \circ \Phi$ to the homotopy in \eqref{2_homotopy}, we have $u\sim_{h} 1$ in $\mathcal U(\calC(J))$ as $\phi$ is unital. It remains to verify that $\text{Ad}(\overline{v}) \circ \Phi (w) \in \overline{\phi}(A)'\cap \calC(\calK)$, for any $w\in\calU_{2}(\overline{\phi}(A)'\cap \calC(\calK))$. This follows from the following computation: 
\begin{align} 
&(\text{Ad}(\overline{v}) \circ \Phi (w)) \overline{\phi}(a) \\
\stackrel{\eqref{uni_conj_3}}{=} &\overline{v} \Phi (w)\Phi \bigg(
\begin{pmatrix}
\overline{\phi}(a) & 0 \\
0 & \overline{\phi}(a)
\end{pmatrix}
\bigg)\overline{v}^{*} \\
= \ \ & \overline{v} \Phi \bigg(
\begin{pmatrix}
\overline{\phi}(a) & 0 \\
0 & \overline{\phi}(a)
\end{pmatrix}
\bigg)\Phi (w)\overline{v}^{*} \\
\stackrel{\eqref{uni_conj_3}}{=} &\overline{\phi}(a)(\text{Ad}(\overline{v}) \circ \Phi (w)). \qedhere 
\end{align} 
\end{proof} 

\section{Codomain hypotheses, the trace-kernel extension and separabilisation arguments} 
\label{Sect:sep}  

In this section we set out the abstract framework we will use to handle $C^*$-algebra ultraproducts of finite factors, and examine the  resulting trace-kernel extensions.  The relevant $C^*$-algebraic properties of a finite von Neumann factor for our argument are:
\begin{enumerate}
    \item real rank zero;\label{codomain1}
    \item stable rank one;\label{codomain2}\footnote{Recall that stable rank one gives rise to cancellation of projections by \cite{Rieffel-stablerank} (see also \cite[Proposition V.3.1.24]{Blackadar}, and hence that $V(B_{n})$ is partially ordered (rather than preordered). Moreover in the presence of real rank zero, cancellation of projections is equivalent to stable rank one (see \cite[Proposition V.3.2.15]{Blackadar}).  } 
    \item there exists a tracial state;\label{codomain3}
    \item trivial $K_1$-group;\label{codomain4}
    \item the Murray--von Neumann semigroup is totally ordered.\label{codomain5}
\end{enumerate}
We will refer to these collection of conditions as our \emph{codomain hypotheses}. The abstract version of our classification theorem will apply to ultraproducts of $C^*$-algebras $(B_n)_{n}$ satisfying these hypotheses as the codomain (see Theorem \ref{MainUniqueness}), subject to a condition ensuring that the tracial ultraproduct is type II$_1$ (discussed just after Proposition \ref{prop:TKquotient}).

Comparing our codomain hypotheses with those used by Schafhauser in \cite{Schafhauser.Annals}, the essential difference is that we do not require tensorial absorption of the universal UHF-algebra, as finite von Neumann factors fail to satisfy this. Instead we impose the condition (\ref{codomain5}). While covering the main case of interest, this is of course a very stringent hypothesis from a $C^*$-algebraic point of view.  The results in this section are all minor modifications of the relevant technical results from \cite{new-TWW,Schafhauser.Annals} to replace the UHF-stability hypothesis with condition (\ref{codomain5}).

So that we can conveniently appeal to results in the literature, we note that the combination of properties \eqref{codomain1}, \eqref{codomain3} and \eqref{codomain5} above yield strict comparison and that the tracial state is the unique quasitrace (both well known properties of finite von Neumann factors). Recall that for a unital $C^*$-algebra $B$ and a tracial state $\tau$ on $B$, define 
\begin{equation}
d_{\tau} (b) = \lim_{n\rightarrow \infty} \tau (b^{1/n}), \quad b\in (B\otimes \calK)_{+}, 
\end{equation} 
where $\tau$ is canonically extended to $(B\otimes \calK)_{+}$. Then $d_{\tau}$ induces a \emph{state} on $\mathrm{Cu}(B)$, which is an order-preserving, additive and lower-semicontinuous map $\mathrm{Cu}(B)\rightarrow [0,\infty]$, sending $\langle 1_{B} \rangle$ to $1$. 

\begin{prop}\label{prop.codomain.strictcomp}
    Let $B$ be a unital $C^*$-algebra of real rank zero, which has a tracial state $\tau$. Suppose that $V(B)$ is totally ordered. Then $B$ has strict comparison by bounded traces and $\tau$ is the unique quasitracial state on $B$. 
\end{prop} 
\begin{proof} 
We claim that $B$ has strict comparison with respect to $\tau$. Firstly, given projections $p,q\in B\otimes \mathcal K$ with $\tau(p)<\tau(q)$, as $V(B)$ is totally ordered we must have $p\prec q$ (since $q\precsim p$ is evidentally impossible). Now take positive elements $a, b \in B\otimes \calK$ with $d_{\tau}(a) < d_{\tau}(b)$. Since $B$ has real rank zero, by Proposition \ref{RR0_Cu}, it follows that $\langle a\rangle = \sup_{n}\langle p_{n}\rangle$ and $\langle b\rangle = \sup_{n}\langle q_{n}\rangle$ for some sequences of projections $(p_{n})_{n}$ and $(q_{n})_{n}$ in $B\otimes \calK$ which give increasing sequences in the Cuntz semigroup. By lower-semicontinuity of $d_{\tau}$, there exists $n_{0}$ such that for any $n\geq n_{0}$, 
\begin{equation} 
\tau(p_n)=d_{\tau} (p_{n}) \leq d_{\tau} (a) < d_{\tau} (q_{n})=\tau(q_n) \leq d_{\tau} (b).
\end{equation} 
Thus, $p_{n}\precsim q_{n}\precsim b$ for any $n\geq n_{0}$ and this implies $a\precsim b$ in $B\otimes \calK$.

Lastly, we note that the total ordering hypothesis ensures that there is a unique state on $V(B)$. Indeed, given states $\phi_1,\phi_2$ on $V(B)$ and $x\in V(B)$ with $\phi_1(x)<\phi_2(x)$. Then there exists $m, n\in \mathbb{N}$ such that $n\phi_{1}(x) < m < n\phi_{2}(x)$. By the total ordering of $V(B)$, we have either $m[1_B]\leq nx$ (which contradicts $n\phi_1(x)<m$) or $nx\leq m[1_B]$ (which contradicts $m<n\phi_2(x)$). Since $B$ has real rank zero, Proposition \ref{RR0_Cu} then shows that $\mathrm{Cu}(B)$ has a unique state, and hence $B$ has a unique quasitrace as states on the Cuntz semigroup correspond to quasitraces (see \cite[Theorem 6.9]{modern-cuntz}, for example).  
\end{proof}

We now turn to ultraproducts and the trace-kernel extension. Throughout the paper, $\omega$ denotes a fixed free ultrafilter on the natural numbers.  Given a sequence $(B_n)_n$ of $C^*$-algebras, one has the product $C^*$-algebra,  $\prod_{n=1}^\infty B_n$ consisting of all bounded sequences $(b_n)_n$ with $b_n\in B_n$. The ultraproduct $\prod_{n\to\omega} B_n$ is defined by 
\begin{equation}
\prod_{n\to\omega}B_n \coloneqq \prod_{n=1}^\infty B_n\bigg/\left\{(x_n)_{n}\in\prod_{n=1}^\infty B_n:\lim_{n\to\omega}\|x_n\|=0\right\}.
\end{equation} 
To ease notation, we will often write $B_\omega$ as shorthand for the ultraproduct $\prod_{n\to\omega}B_n$. As is standard we will typically denote elements in $B_\omega$ using representing sequences $(b_n)_n$ from $\prod_{n=1}^\infty B_n$.  A sequence $(\tau_n)_n$ of tracial states on each $B_n$, induces a \emph{limit trace} $\tau$ on $\prod_{n\to\omega}B_n$, defined on representative sequences by $\tau((x_n)_{n})\coloneqq\lim_{n\to\omega}\tau_n(x_n)$.  The following standard lemma records that most of our codomain hypotheses pass to ultraproducts: the only thing missing is that ultraproducts of simple $C^*$-algebras need not be simple (for example in our setting they contain the trace-kernel ideal as an ideal).

\begin{lem}\label{prop:B_omega} Let $(B_n)_{n}$ be a sequence of unital $C^*$-algebras each of which has real rank zero, stable rank one, trivial $K_1$ group and totally ordered Murray--von Neumann semigroup. Then the ultraproduct $B_\omega\coloneqq \prod_{n\to\omega}B_n$ also has real rank zero, stable rank one, totally ordered Murray--von Neumann semigroup, trivial $K_1$ and trivial $K_{1}$ with coefficients. When each $B_n$ has a tracial state $\tau_n$ (necessarily unique by Proposition \ref{prop.codomain.strictcomp}), then the induced limit trace is the unique quasitracial state on $B_\omega$, and this ultraproduct has strict comparison by bounded traces.
\end{lem}

\begin{proof}
It is standard that taking ultraproducts preserves the properties of real rank zero and stable rank one.\footnote{For real rank zero this is easily seen from the equivalence of (i) and (vi) in \cite[Theorem 2.6]{BrownPedersen:JFA} (as noted in \cite[Example 2.4.2]{ModelMemoir}). Preservation of stable rank one for ultraproducts uses \cite[Theorem 5]{Pedersen:JOT}, and is given as \cite[Lemma 19.2.2(1)]{Loring:Book}.} The third paragraph of the proof of \cite[Proposition 3.2 (2)]{TWW:Annals} then shows that $K_{1}(B_{\omega}) = 0$. 

To see that $V(B_\omega)$ is totally ordered, take projections $p, q\in M_{m}(B_{\omega})$ for some $m\in \mathbb{N}$, and choose representing  sequences of projections $(p_{n})_n$ and $(q_{n})_n$, where each $p_{n}, q_{n}\in M_{m}(B_{n})$. Since $V(B_{n})$ is totally ordered and $\omega$ is an ultrafilter, either $\{n\in\mathbb{N}: [p_{n}]\leq [q_{n}]\text{ in } V(B_{n})\}\in \omega$ or $\{n\in\mathbb{N}: [q_{n}]\leq [p_{n}]\text{ in } V(B_{n})\}\in \omega$. Thus either $[p]\leq [q]$ or $[q]\leq [p]$ in $V(B_{\omega})$. 

Under the hypothesis that each $B_n$ has a tracial state $\tau_n$, then the induced limit trace $\tau$ is a tracial state on $B_\omega$. Proposition \ref{prop.codomain.strictcomp} then shows\footnote{Using the proposition is unnecessary; one can alternatively use the fact that strict comparison by bounded traces passes to ultrapowers (see \cite[Lemma 1.23]{BBSTWW}, for example).} that $B_\omega$ has strict comparison by bounded traces and $\tau$ is the unique quasitrace on the ultraproduct. 

Finally, we prove that $K_{1}(B_{\omega}; \mathbb{Z}/n\mathbb{Z}) = 0$ for any $n\geq 2$. By Proposition \ref{trivial_K1_coef}, it suffices to show that $K_{0}(B_{\omega})$ has no torsion. For this note that any element $x\in K_0(B_\omega)$ can be written as a difference $x=[p]_0-[q]_0$ for projections $p,q$ in matrices over $B_\omega$. As $V(B_\omega)$ is totally ordered, either $[p]\leq [q]$ or $[q]\leq [p]$ in $V(B_\omega)$ and hence there is a projection $r$ in matrices over $B_\omega$ with either $x=[r]_0$ or $-x=[r]_0$. As $B_\omega$ is stably finite (since it has stable rank one), there are no non-zero projections $s$ in matrices over $B_\omega$ with $[s]_0=0$ in $K_0(B_\omega)$. In particular $nx\neq 0$ for all $n\geq 1$, i.e. $K_0(B_\omega)$ is indeed torsion free.
\end{proof} 

Henceforth, we assume that each $B_n$ has a unique tracial state $\tau_n$, since this follows from our codomain hypotheses. We denote the induced limit trace on the ultraproduct $B_\omega$ by $\tau_{B_\omega}$, and write $\|x\|_{2,\omega}=\tau_{B_{\omega}}(x^*x)^{1/2}$, for $x\in B_\omega$ for the associated seminorm. Then the \emph{trace-kernel ideal} of $B_\omega$ is given by 
\begin{equation}
J_B\coloneqq \{x\in B_\omega:\tau_{B_{\omega}}(x^*x)=0\}, 
\end{equation} 
giving rise to the \emph{trace-kernel extension} in this unique trace setting: 
\begin{equation} \label{Trace_kernel}
0\longrightarrow J_B\stackrel{j_B}{\longrightarrow} B_\omega\stackrel{q_B}{\longrightarrow} B^\omega \longrightarrow 0.
\end{equation} 
Here the \emph{trace-kernel quotient} is $B^\omega\coloneqq B_\omega/J_B$, and $\tau_{B_\omega}$ descends to a faithful trace $\tau_{B^\omega}$ on $B^\omega$, i.e., $\tau_{B_{\omega}} = \tau_{B^{\omega}}\circ q_{B}$.  Central sequence versions of these trace-kernel extensions where first developed by Matui and Sato in \cite{Matui-Sato}. Since we are in the unique trace setting, each $\pi_{\tau_n}(B_n)''$ is a finite von Neumann factor (see \cite[Theorem 6.7.4]{Dixmier}), and (as an application of Kaplanskzy density's theorem) $B^\omega$ canonically identifies with the tracial von Neumann algebra ultraproduct $\prod_{n\to\omega}\pi_{\tau_n}(B_n)''$, which is a finite factor (\cite[Remark 4.7]{KR_SI}). 

\begin{prop}\label{prop:TKquotient} 
    Let $(B_n)_{n}$ be a sequence of unital $C^*$-algebras, each with a unique trace $\tau_n$. Then the trace-kernel quotient $B^\omega$ is a finite von Neumann factor. It is type $I_m$ precisely when there is some $m\in\mathbb N$ such that $\{n:\pi_{\tau_n}(B_n)''\text{ is type I}_m\}\in\omega$. 
\end{prop} 

It will be important in the arguments of Section \ref{classification_chapter} that the trace-kernel quotient is type II$_1$. For this reason, in our main abstract result (Theorem \ref{MainUniqueness}), we additionally require that $\dim_{n\to\omega}\pi_{\tau_n}(B_n)''=\infty$. This hypothesis will first appear in Lemma \ref{seplem}. 

Next we collect the relevant properties of the trace-kernel ideals. In \cite{new-TWW,Schafhauser.Annals}, Schafhauser used the notion of an \emph{admissible kernel} for a $C^*$-algebra $J$ with real rank zero, stable rank one, $K_1(J)=0$, $K_0(J)$ is divisible, $V(J)$ is almost unperforated, and every projection in $J\otimes\mathcal K$ is Murray--von Neumann equivalent to a projection in $J$. Although our trace-kernel ideal $J_{B}$ need not have divisible $K_0$-groups,\footnote{and will not, in the case corresponding to an ultraproduct of matrices of growing sizes, where the trace-kernel ideal has a minimal projection.} all of the other properties of the admissible kernels hold in just the same way as in \cite{Schafhauser.Annals}. 

\begin{lem}
 Let $(B_n)_{n}$ be a sequence of unital $C^*$-algebras each of which has real rank zero, stable rank one, trivial $K_1$ group, and totally ordered Murray--von Neumann semigroup. Suppose each $B_n$ has a unique tracial state $\tau_n$. Then the trace-kernel quotient $J_B$ has real rank zero, stable rank one, $K_1(J_B)=0$, totally ordered Murray--von Neumann semigroup, and every projection in $J_B\otimes\mathcal K$ is Murray--von Neumann equivalent to a projection in $J_B$. 
\end{lem}

\begin{proof}
As the properties of having real rank zero (\cite[Corollary 2.8]{BrownPedersen:JFA}), stable rank one (\cite[Theorem 4.3]{Rieffel-stablerank}) pass to ideals, $J_B$ inherits all these properties from the ultraproduct $B_\omega$ through Lemma \ref{prop:B_omega}. Likewise, since $V(B_\omega)$ is totally ordered, so too is $V(J_B)$.\footnote{Recall that if $p,q$ are projections in matrices over $J_B$ with $p\precsim q$ in matrices over $B_\omega$, then the partial isometry witnessing such a subequivalence lies in matrices over $J_B$.}  The argument from the first paragraph of the proof of \cite[Proposition 3.2(3)]{Schafhauser.Annals} applies to show that projections in $J_B\otimes\mathcal K$ are Murray--von Neumann equivalent to projections in $J_B$, as this only uses strict comparison of $B_\omega$ (noted in Lemma \ref{prop:B_omega}). 

The proof that $K_{1}(J_{B}) = 0$ also goes in the same fashion as the relevant part of \cite[Proposition 3.2(3)]{Schafhauser.Annals}. Consider the following fragment of the $6$-term exact sequence induced by the trace-kernel extension:
\begin{equation} 
\begin{tikzcd} K_0(B_\omega) \arrow{r}{K_0(q_B)} & K_0(B^\omega) \arrow{r}{\partial_0} & K_1(J_B) \arrow{r}{K_1(j_B)} & K_1(B_\omega). 
\end{tikzcd}
\end{equation} 
As noted in Lemma \ref{prop:B_omega}, $K_{1}(B_{\omega}) = 0$, so by exactness it suffices to check that $K_{0}(q_{B})$ is surjective. Given a projection $p\in B^\omega$, with $p\neq 0,1$, choose a contractive positive lift $x\in B_\omega$ of $p$ so that $d_{\tau_{B_\omega}}(x)=\tau_{B^\omega}(p)$. Since $B_\omega$ has strict comparison and real rank zero (as noted in Lemma \ref{prop:B_omega}), there is an increasing sequence of projections $(q_n)_{n}$ in $B_\omega$ such that $\sup_{n}\langle q_n\rangle=\langle x\rangle$ in $\mathrm{Cu}(B_\omega)$ by Proposition \ref{RR0_Cu} (note strict comparison allows these projections to be taken in $B_\omega$ as $\tau_{B_\omega}(q_n)\leq \tau_{B^{\omega}}(p)<1$).  Since $d_{\tau_{B_\omega}}$ preserves suprema, we have $\tau_{B_\omega}(q_n)\to d_{\tau_{B_\omega}}(x) =\tau_{B^\omega}(p)$.  Hence for any $\epsilon>0$, there is a projection $q\in B_\omega$ with $|\tau_{B_\omega}(q)-\tau_{B^\omega}(p)|<\epsilon$. Using Kirchberg's $\epsilon$-test, or reindexing, there is a projection $q\in B_\omega$ with $\tau_{B_\omega}(q)=\tau_{B^\omega}(p)$.  Since projections in the finite factor $B^\omega$ are determined up to unitary conjugacy by the trace $\tau_{B^{\omega}}$, it follows that a unitary conjugacy of $q$ lifts $p$. As $K_0(B^\omega)$ is generated by the projections in $B^\omega$, $K_0(q_B)$ is indeed surjective, as required.  
\end{proof}

We now turn to describe the `separabilisation' results we need in Section \ref{classification_chapter}.  The point behind these is that the trace-kernel ideal is not separable, nor even $\sigma$-unital.  This causes technical difficulties working with $J_B$ as a codomain in $KK$-theory, and it must be replaced  by a suitable separable subalgebra. 

We first recall that separable $C^*$-algebras satisfying the properties of our trace-kernel ideals are stable and have the corona factorisation property (see the discussion in Section \ref{KL_section}). This lemma is essentially contained in \cite[Theorem 2.2]{new-TWW} and \cite[Proposition 3.3(3)]{Schafhauser.Annals}, as the condition that admissible kernels have divisible $K_0$-groups is not used in the proof.

\begin{lem}\label{separableideals}
Let $J$ be a separable $C^*$-algebra which has real rank zero, stable rank one, totally ordered Murray--von Neumann semigroup and every projection in $J\otimes\mathcal K$ is Murray--von Neumann equivalent to a projection in $J$. Then $J$ is stable and has the corona factorisation property.
\end{lem}
\begin{proof}
For a separable $C^*$-algebra with stable rank one (giving cancellation of projections), and real rank zero (giving an approximate unit of projections), stability is equivalent to the condition that every projection in $J\otimes\mathcal K$ is Murray--von Neumann equivalent to a projection in $J$. This is \cite[Proposition 3.4(i)$\Leftrightarrow$(ii)]{Rordam:Stable} from Rørdam's survey of his work on stability, and in particular his characterisations of stability together with Hjelmborg (\cite{Stability_RR0}).
    
The corona factorisation property follows from real rank zero and the fact that the Murray--von Neumann semigroup of $J$ is almost unperforated (as a consequence of being totally ordered and having cancellation of projections) by \cite[Corollary 5.9]{OPR:TAMS}.\footnote{Recall that an ordered semigroup $S$ is \emph{almost unperforated} if for any $x, y\in S$ satisfying $(n+1)x\leq ny$ for some $n\in \mathbb{N}$, we have $x\leq y$ in $S$. In their later work \cite{OPR:IMRN} Ortega, Perera and Rørdam characterised the corona factorisation property at the level of the Cuntz semigroup, deducing that it follows from having almost unperforated Cuntz semigroup. } 
\end{proof} 

The reduction to the separable setting is performed using Blackadar's notion of separable inheritability. 

\begin{defn}[{\cite[Section II.8.5]{Blackadar}}] 
\label{defn_si} 
A property $(P)$ of $C^*$-algebras is called \emph{separably inheritable} if the following statements are true, 
\begin{enumerate} [(i)] 
\item if $A$ is a $C^*$-algebra satisfying $(P)$ and $A_0$ is a separable $C^*$-subalgebra of $A$, there exists a separable $C^*$-subalgebra $B$ of $A$ which satisfies $(P)$ and $A_0\subseteq B $;
\item if $A_1 \hookrightarrow A_2 \hookrightarrow A_3 \hookrightarrow \cdots$ is an inductive limit of separable $C^*$-algebras with injective connecting maps, if each $A_n$ satisfies $(P)$, then $\underset{\longrightarrow}{\lim} \, A_n$ satisfies $(P)$.
\end{enumerate} 
\end{defn} 

The separably inheritable properties we use are listed in the lemma below.  All are proved in a standard way (illustrated in the proof in \cite[II.8.5.4]{Blackadar} that stable rank one is separably inheritable).\footnote{Blackadar notes that Properties (\ref{rr0_si}) and (\ref{k1=0_si}) are separably inheritable in \cite[II.8.5.5] {Blackadar}. Property (\ref{projMvN_si}) is noted to be separably inheritable in the second sentence of the proof of \cite[Proposition 4.1]{new-TWW}. A proof that Property (\ref{totord_si}) is separably inheritable is given by the first author in her DPhil thesis (\cite[Lemma 7.1.2]{Hua_thesis}). The `standard proof' works fine too for Property (\ref{trivial_k_si}), or alternatively one can deduce this from separable inheritability of Property (\ref{k1=0_si}) using the fact that for $C^*$-algebras $C,D$ with $C$ separable, any separable $C^*$-subalgebra of $D\otimes C$ is contained in a separable $C^*$-subalgebra of the form $D_0\otimes C$.  Using the model $K_0(A;\mathbb Z/n\mathbb Z)\cong K_1(A\otimes I_n)$, where $I_n\coloneqq\{f\in C([0,1],M_n):f(0)\in \mathbb C1_n,\ f(1)=0\}$, handles $K_0$ with coefficients, and either suspending, or using the model $K_1(A;\mathbb Z/n\mathbb Z)\cong K_1(A\otimes\mathcal O_{n+1})$ handles $K_1$ with coefficients.}  

\begin{lem} 
\label{sep_inh_properties} 
The following properties of a $C^*$-algebra $D$ are separably inheritable: 
\begin{enumerate} [(i)] 
\item \label{rr0_si} $D$ has real rank zero; 
\item \label{sr0_si} $D$ has stable rank one; 
\item $K_1(D)=0$;\label{k1=0_si}
\item \label{trivial_k_si} $K_{i} (D; \mathbb{Z} / n\mathbb{Z}) = 0$ for some $i\in \{0, 1\}$ and $n\in \mathbb{N}$; 
\item \label{MvN_si} $V(D)$ is totally ordered; \label{totord_si}
\item every projection in $D\otimes\mathcal K$ is Murray--von Neumann equivalent to a projection in $D$.\label{projMvN_si}
\end{enumerate} 
\end{lem} 

Recall too that the countable intersection of separably inheritable properties is again separably inheritable (\cite[II.8.5.3]{Blackadar}), so we can apply separable inheritability to any combination of the properties in Lemma \ref{sep_inh_properties}.  The following lemma sets out the separabilisation we will use in the uniqueness theorem (Theorem \ref{MainUniqueness}). As we want to use this to set up the relevant $K$-theoretic properties of $D$ which follow from $B^\omega$ being a II$_1$ factor, it is here where we first include the requirement that $\lim_{n\to\omega}\dim\pi_{\tau_n}(B_n)''=\infty$.

\begin{lem}\label{seplem}
Let $(B_n)_{n}$ be a sequence of unital $C^*$-algebras which have real rank zero, stable rank one, a tracial state $\tau_n$, totally ordered $V(B_{n})$ and $K_1(B_n)=0$. Write $B_{\omega}$ for $\prod_{\omega}B_{n}$, and $\tau_{B_\omega}$ for the (necessarily unique) trace on $B_\omega$ arising as the limit trace of the sequence $(\tau_n)_n$ inducing the trace-kernel extension \eqref{Trace_kernel}. Assume that $\lim_{n\to\omega}\dim\pi_{\tau_n}(B_n)''=\infty$. 

Let $A$ be a separable and unital $C^{\ast}$-algebra and $\phi,\psi\colon A\to B_\omega$ be nuclear, full and unital $^*$-homomorphisms with $\underline{K}(\phi)=\underline{K}(\psi)$.  Then there exists a separable subextension
\begin{equation}\label{sepsubext}
\begin{tikzcd}
0\arrow[r]&J\arrow[r,"j"]\arrow[d]&E\arrow[r,"q"]\arrow[d]&D\arrow[r]\arrow[d]&0\\
0\arrow[r]&J_B\arrow[r,"j_B"]&B_{\omega}\arrow[r,"q_B"]&B^\omega\arrow[r]&0
\end{tikzcd}
\end{equation} 
of the trace-kernel extension such that:
\begin{enumerate}
\item \label{condition_sep_full} $\phi(A)\cup \psi(A)\subseteq E$, and the corestrictions $\phi|^E,\psi|^E\colon A\to E$ of $\phi$ and $\psi$ to $E$ are unital, full, nuclear and satisfy $\underline{K}(\phi|^E)=\underline{K}(\psi|^E)$; \label{seplem1}
\item $K_{1}(J)$, $K_{1}(D)$ and $K_{i}(D; \mathbb{Z} / n\mathbb{Z})$ vanish for $i\in \{0,1\}$ and $n\geq 2$; \label{seplem2}
\item \label{condition_sep_J}$J$ has real rank zero and stable rank one, $V(J)$ is totally ordered and every projection in $J\otimes\mathcal K$ is Murray--von Neumann equivalent to a projection in $J$. \label{seplem3}
\end{enumerate}
\end{lem}

\begin{proof}
We start by using \cite[Proposition 1.9]{Schafhauser.Annals} to find a separable $C^*$-subalgebra $E_0\subset B_\omega$ containing $\phi(A)\cup \psi(A)$ and such that the corestrictions $\phi|^{E_0},\psi|^{E_0}$ are full, unital and nuclear maps. Then the same will hold when $E_0$ is replaced by any larger $C^*$-subalgebra of $B_\omega$. Since each of $K_0(A)$, $K_1(A)$, $K_0(A;\mathbb Z/n\mathbb Z)$ and $K_1(A;\mathbb Z/n\mathbb Z)$ for all $n\geq 2$ is countable, and $\underline{K}(\phi)=\underline{K}(\psi)$, we can find some larger separable $C^*$-subalgebra $E_1\supseteq E_0$ of $B_\omega$ such that the corestrictions to $E_1$ have $\underline{K}(\phi|^{E_1})=\underline{K}(\psi|^{E_1})$.\footnote{This follows from the same sort of `standard proof' used to show that having $K_*(\cdot;\mathbb Z/n\mathbb Z)=0$ is separably inheritable.} 

Proposition \ref{prop:TKquotient}, together with the assumption that the dimensions of $\pi_\tau(B_n)''$ do not remain bounded as $n\to\omega$, it follows that $B^\omega$ is a II$_1$ factor. Accordingly, we have $K_{1} (B^{\omega}) = 0$,  $K_{0} (B^{\omega}) \cong \mathbb{R}$, and $K_{i} (B^{\omega}; \mathbb{Z} / n\mathbb{Z}) = 0$ for $i\in \{0,1\}$ and $n\geq 2$ (following \eqref{eq:bockstein-new}). 

Since all the properties listed in items (\ref{seplem2}) and (\ref{seplem3}) above are separably inheritable (as noted in Lemma \ref{sep_inh_properties}), we can then apply  \cite[Proposition 1.6]{Schafhauser.Annals} to obtain a separable subextension \eqref{sepsubext} with $E_1\subseteq E$ satisfying (\ref{seplem2}) and (\ref{seplem3}). Since $E_1\subseteq E$, condition (\ref{seplem1}) also holds. 
\end{proof}

\section{A relatively purely large ideal in $\mathcal M(J)$} 
\label{pure_large_sec} 
In this section, we show that the separable subextensions $0\rightarrow J\rightarrow E\rightarrow D\rightarrow 0$ we built from the the trace-kernel extension in Section \ref{Sect:sep}, give rise to a purely large ideal $\mathcal I$ in the multiplier algebra $\mathcal M(J)$ (Theorem \ref{pure_large_intro}, or Theorem \ref{posi_ideal} below). 

As discussed in Section \ref{KL_section}, pure largeness was introduced by Elliott and Kucerovsky in \cite{EK.PJM} to abstractly characterise nuclearly absorbing extensions. As part of their ongoing abstract approach to classification theorems, Carri\'on, Gabe, Schafhauser, Tikuisis and the second named author simplified the definition of pure largeness for extensions in the case when the ideal is stable (which is a necessary hypothesis in the Elliott--Kucerovsky theorem), describing it in terms of Cuntz comparison: when $I$ is stable, $I\lhd E$ is \emph{purely large} if and only if for all $a\in I_+$ and $b\in E_+\setminus I$, we have $a\precsim b$. A proof of this fact (due to eventually appear in \cite{CGSTW2}) has been given by Bouwen and Gabe as \cite[Proposition 4.14]{Bouwen-Gabe-2024}.  This formulation of pure largeness motivates the following definition.\footnote{The original definition of Elliott and Kucerovsky is that $I\lhd E$ is purely large if for every $x\in E\setminus I$, the $C^*$-algebra $\overline{xIx^*}$ has a stable subalgebra which is full in $I$.} 

\begin{defn} 
\label{defn_PL}  
Let $J$ be a stable ideal in a $C^*$-algebra $I$. We call an extension 
\begin{equation} \label{extension-PL} 
0\longrightarrow I \stackrel{}{\longrightarrow} E\stackrel{q}{\longrightarrow} D \longrightarrow 0 
\end{equation} 
 \textit{purely large relative to} $J$ if for each $a\in J_{+}$ and $b\in E_{+} \setminus I$, we have $a\precsim b$ in $E$. 
\end{defn} 

Given a Cuntz subequivalence $a\precsim b$ it is not generally possible to control the norm of a sequence $(x_n)_{n}$ of elements with $x_n^*bx_n\to a$. However, control is possible when the Cuntz equivalence comes as the output of a (relatively) purely large extension. The following lemma is the relative version of \cite[Proposition 4.10]{Bouwen-Gabe-2024}, and the functional calculus argument which underpins it goes back at least to (\cite[Lemma 7]{EK.PJM}). As the proof is just a notational modification of \cite[Proposition 4.10]{Bouwen-Gabe-2024}, we omit it here; it is given in full in the first named author's DPhil thesis (\cite[Lemma 5.1.2]{Hua_thesis}).

\begin{lem}  \label{relative_purely_large_equi}
Let $J$ be a stable ideal in a $C^*$-algebra $I$ and suppose that an extension as in \eqref{extension-PL} is purely large relative to $J$.  Then for all $a\in J_+$ and $b\in E_+\setminus I$, there exists a sequence $(x_n)_n$ in $E$ with $\|x_n\|\leq (\|a\|/\|q(b)\|)^{1/2}$, and $x_n^*bx_n\to a$.
\end{lem} 

The motivating example of relatively purely large extensions for us comes from the `trace-class' ideals in multiplier algebras.  Let $J$ be the stabilisation of a simple separable unital $C^*$-algebra with unique trace and strict comparison. R\o{}rdam has shown in \cite{ideals_multiplier_finite} that the `trace-class' ideal $\mathcal T$ in $\mathcal M(J)$ below is the unique ideal between $J$ and $\calM(J)$. His work (and the generalisation to the case where $J$ is potentially stably projectionless in \cite{multiplier_strict_comparison}) shows that $\mathcal T$ is relatively purely large in $\mathcal M(J)$ with respect to $J$.  From our viewpoint (as discussed further in Section \ref{K1_Paschke_chapter}), relatively pure largeness forms an important component in Loreaux, Ng, and Sutradhar's $KK$-uniqueness theorem for pairs $(A,J)$ when $A$ is simple unital separable and nuclear, and $J$ is simple and has strict comparison with respect to the unique densely defined lower semicontinuous trace. 

\begin{prop} 
\label{unique_trace_PL} 
Let $J$ be a separable, simple and stable $C^*$-algebra with a unique (up to scaling) densely defined lower semicontinuous trace $\tau_J$  and strict comparison. Extending $\tau_J$ to a lower semicontinuous (but not densely finite) tracial weight on $\calM(J)_+$,\footnote{The extension is given by $\tau_J(x)=\lim_n\tau_J(e_nxe_n)$ for $x\in\mathcal M(J)_+$ and $(e_n)_n$ any approximate unit for $J$. It is independent of the choice of approximate unit.} one obtains a `trace-class' ideal in $\calM(J)$ given by 
\begin{equation} \label{unique_trace_PL:1}
\calT = \overline{\{x\in \calM(J): \tau_{J}(x^{\ast}x) < \infty\}}.
\end{equation} 
Then the corresponding extension $0\to \calT\to \calM(J)\to \calM(J)/\calT\to 0$ is purely large relative to $J$. 
\end{prop} 

\begin{proof} 
In \cite[Theorem 5.3]{multiplier_strict_comparison}, Kaftal, Ng and Zhang show that if $D$ is a $\sigma$-unital stable $C^*$-algebra with strict comparison with respect to a cone of densely defined lower semicontinuous traces with a finite dimensional base, then the multiplier algebra $\mathcal M(D)$ has the condition of \emph{strict comparison with respect to the traces extended from $D$}. Precisely, this means that for any nonzero $a, b\in \calM(D)_{+}$, we have $a\precsim b$ whenever: 
\begin{enumerate}[(i)] 
\item $a$ is in the ideal generated by $b$ in $\calM(D)$;\label{unique_trace_PL.1}
\item $d_{\tau} (a) < d_{\tau} (b)$ for all tracial weights $\tau$ on $\calM(D)$, which arise from extending a densely finite lower semicontinuous tracial weight $\tau$ on $D_+$, with $d_{\tau} (b) < \infty$. \label{unique_trace_PL.2}
\end{enumerate} 
This result applies to the $C^*$-algebra $J$ as in the statement of the proposition. 

Certainly the ideal $\mathcal T$ from \eqref{unique_trace_PL:1} contains $J$ (as $\tau_J$ is densely defined on $J$). By \cite[Lemma 2.6]{sc_KNZ}, a positive element $d$ in $\calM(J)$ is in $\calT$ if and only if $d_{\tau_{J}}((d-\delta)_{+})< \infty$ for any $\delta>0$. Therefore $1_{\calM(J)}\notin \calT$ and so $\calT$ is properly contained in $\calM(J)$. Take nonzero elements $b\in \calM(J)_{+}\setminus \calT$ and $a\in J_{+}$. Note that as $J$ is simple and $J\lhd \calM(J)$ is essential, it follows that every ideal in $\calM(J)$ contains $J$, and thus condition (\ref{unique_trace_PL.1}) above is automatic. Moreover, by the definition of $\calT$, we have $d_{\tau_{J}}(b) \geq \tau_{J}(b) = \infty$, which means that condition (\ref{unique_trace_PL.2}) above is trivially satisfied. Thus $a\precsim b$ by strict comparison with respect to extended traces, and hence the extension $0\to \calT\to \calM(J)\to \calM(J)/\calT\to 0$ is purely large relative to $J$. 
\end{proof} 





Our main objective in this section is the following, which will be used to obtain relative pure largeness needed for our $K_1$-injectivity theorem (Theorem \ref{K1_inj_main_4.2_intro}).  

\begin{thm} \label{posi_ideal} 
Let $J$ be a separable and stable $C^*$-algebra with real rank zero, stable rank one, $K_{1}(J) = 0$, and totally ordered $V(J)$. Then there exists a maximal proper ideal $\calI$ of $\calM(J)$ containing $J$ such that for any $b\in \calM(J)_{+}$, then $b\notin \calI$ if and only if $b\sim 1_{\calM(J)}$. Moreover, $\calI$ is the unique ideal with this property. In particular, the extension given by $\calI\lhd \calM(J)$ is purely large and $\calM(J) / \calI$ is simple and purely infinite. 
\end{thm} 

We will find the ideal $\mathcal I$ in Theorem \ref{posi_ideal} though Zhang's Proposition \ref{ideal_corr}, so we need to obtain a suitable order ideal in $V(\calM(J))$. This is achieved by means of the correspondence in Lemma \ref{dr-cgi} due to Goodearl, between equivalence classes in $V(\calM(J))$ with countably generated intervals in $V(J)$.  We recall the relevant definitions. 

\begin{defn} 
\label{defn_interval} 
Let $S$ be a partially ordered abelian monoid. A nonempty subset $I\subseteq S$ is
\begin{enumerate} [(i)] 
\item \textit{upward-directed} if for any $x, y\in I$, there exists some $z\in I$ such that $x, y\leq z$; 
\item an \textit{interval} if it is hereditary and upward-directed. 
\end{enumerate} 
An interval $I$ is \emph{countably generated} if there exists a countable \emph{cofinal} subset of $I$, i.e. there exists a sequence $(x_{n})_{n}\subseteq I$ such that for any $x\in I$, we have $x\leq x_{n}$ for some $n$. The set of countably generated intervals in $S$ is denoted by $\Lambda_{\sigma}(S)$. Given countably generated intervals $I_{1}$ and $I_{2}$ in $\Lambda_{\sigma}(S)$, then the sum $I_{1}+I_{2}$ as defined in \eqref{addition_subsets} by 
\begin{equation} 
I_{1}+I_{2} = \{x+y: x\in I_{1}, y\in I_{2}\} 
\end{equation} 
is a non-empty upward-directed subset of $S$.  When $S$ has the Riesz decomposition property, this sum of intervals is hereditary, and thus a countably generated interval. In this case, we give $\Lambda_{\sigma}(S)$ the \emph{algebraic order}:  $I_{1}\leq_{\text{alg}} I_{2}$ if there exists $I'\in \Lambda_{\sigma}(S)$ such that $I_{2} = I_{1}+I'$. In this way the countably generated intervals is a partially ordered abelian monoid. 
\end{defn}

In \cite[Section 1]{Goodearl}, Goodearl was primarily concerned with computing the $K_0$-group of multiplier of $C^*$-algebras of real rank zero and stable rank one -- which he achieves in \cite[Theorem 1.10]{Goodearl} -- generalising earlier work on $K$-theory for multipliers of AF-algebras by Elliott and Handelman (\cite{ElliottHandelman:PJM}).  We extract the computation of the Murray--von Neumann semigroup $V(\mathcal M(J))$ in terms of intervals in $V(J)$ from the lemmas used to prove \cite[Theorem 1.10]{Goodearl}.\footnote{As a warning, the Cuntz semigroup of a separable real rank zero $C^*$-algebra is also expressed in terms of countably generated intervals in its Murray--von Neumann semigroup -- but the order used to describe the Cuntz semigroup in terms of intervals is given by inclusion which is not the algebraic order; see \cite[Theorem 5.7]{RR0_Cuntz}. Accordingly $\mathrm{Cu}(J)$ and $V(\calM(J))$ are not generally isomorphic as ordered monoids.} Note that since in our case $J$ is stable, all elements in the Murray--von Neumann semigroups of $J$ and $\calM(J)$ are respectively given by Murray--von Neumann equivalence classes of projections in these algebras. No matrix amplifications are needed.\footnote{It is useful to note that Goodearl is not working with stable algebras (as then $K_0(\mathcal M(J))$ would be equal to $0$), and so states his intermediate results with the restriction that the intervals involved are algebraically complemented in a finite multiple of the dimension range. In our setting when $J$ is stable, this hypothesis vacuously holds for all intervals.} Since $J$ has real rank zero, Zhang's \cite[Theorem 1.1]{Riesz-dec}\footnote{This too is stated using the `dimension range'; since our $J$ is stable, the dimension range of $J$ is equal to the Murray--von Neumann semigroup $V(J)$.} ensures that $V(J)$ has the Riesz decomposition property and so $\Lambda_{\sigma}(V(J))$ is a partially ordered abelian monoid. 

\begin{lem}[Goodearl]
\label{dr-cgi} 
Let $J$ be a separable and stable $C^*$-algebra with real rank zero and stable rank one. There exists a well-defined ordered monoid isomorphism $\theta: V(\calM(J))\rightarrow \Lambda_{\sigma}(V(J))$ given by 
\begin{equation} \label{theta}
[e]\mapsto I_{e} \coloneqq \{[p]\in V(J): p\in eJe\}, 
\end{equation} 
for any projection $e\in \calM(J)$. 
\end{lem}

\begin{proof} 
It is shown in \cite[Lemma 1.6, Proposition 1.7]{Goodearl} that $\theta$ is an injective monoid homomorphism. Surjectivity is the combination of  \cite[Lemma 1.8 and Remark 1.9]{Goodearl}: given an interval $I\in \Lambda_\sigma(V(J))$, we have $I+V(J)=V(J)=\theta([1_{\mathcal M(J)}])$, so applying Remark 1.9 of \cite{Goodearl}, there exists a projection $e\in \calM(J)$ with $\theta([e])=I$.

Since the order on $V(\calM(J))$ is the algebraic order,\footnote{For projections $p,q\in M_\infty(\calM(J))$, one has $[p]\leq [q]$ if and only if there exists $x\in V(\calM(J))$ with $[p]+x=[q]$. Indeed, after finding a projection $p'\sim p$ with $p'\leq q$, one takes $x=[q-p']$.} the map $\theta$ is an order isomorphism when $\Lambda_\sigma(V(J))$ is also given the algebraic order.
\end{proof} 

Using Proposition \ref{ideal_corr} and Proposition \ref{dr-cgi}, we specify the ideal we need in $\calM(J)$ by giving an order ideal in $\Lambda_\sigma(V(J))$. This is done in Lemma \ref{proj-ideal} below, using the next lemma to obtain closure under addition.

\begin{lem} \label{non-dr-bound} 
Let $J$ be a separable and stable $C^*$-algebra with an approximate unit $(p_{n})_n$ of increasing projections and totally ordered $V(J)$. Then for any $I\in \Lambda_{\sigma}(V(J))$ with $I \neq V(J)$, there exists $n_{0}$ such that $I \subseteq I_{p_{n_{0}}}$. 
\end{lem} 

\begin{proof}
Suppose the conclusion does not hold, so that for each $n\in\mathbb N$, there exists $y_n\in I$ with $y_n\notin I_{p_n}$. By the total ordering of $V(J)$, we have $[p_n]\leq y_n$ and hence each $p_n\in I$. Since $(p_n)_n$ is an approximate unit of $J$ consisting of projections, it follows that every projection $p\in J$ has $[p]\leq [p_n]$ for large enough $n$, and hence $[p]\in I$. Finally, stability ensures that every element of $V(J)$ is represented by a projection from $J$, and hence $I=V(J)$.
\end{proof} 

\begin{lem} 
\label{proj-ideal} 
Let $J$ be a separable and stable $C^*$-algebra with real rank zero, stable rank one and totally ordered $V(J)$. Then $\calL\coloneqq V(\mathcal M(J))\setminus\{[1_{\calM(J)}]\}$ is an order ideal in $V(\calM(J))$. The ideal $\calI$ of $\calM(J)$ corresponding to $\calL$ given by Proposition \ref{ideal_corr} has the property that for any projection $e$ in $\calM(J)$, we have that $e\notin \calI$ if and only $[e] = [1_{\calM(J)}]$ in $V(\calM(J))$. Furthermore, $\calI$ is unique with this property and is a maximal proper closed ideal of $\calM(J)$. 
\end{lem} 

\begin{proof} 
To see that $\calL$ is hereditary, take projections $e, f$ in $\calM(J)$ such that $[e]\leq [f]$ and $[f]\in \calL$. Note that under the ordered monoid isomorphism of Lemma \ref{dr-cgi}, we have 
\begin{equation} 
I_{e} \leq_{\text{alg}} I_{f}\neq I_{1_{\mathcal M(J)}} = V(J). 
\end{equation} 
Accordingly $I_{e}\neq V(J)$ and thus $[e]\neq [1_{\calM(J)}]$. Thus, $\calL$ is a proper hereditary subset of $V(\calM(J))$ containing all equivalence classes of projections in $J$. 

To show that $\calL$ is closed under addition, take projections $e,f$ in $\calM(J)$ with $[e], [f]\in \calL$. Applying the ordered monoid isomorphism from Lemma \ref{dr-cgi}, again both $I_{e}, I_{f} \neq V(J)$. Let $(p_n)_n$ be an approximate unit of increasing projections for $J$ (which exists as $J$ has real rank zero). By the total ordering of $V(J)$, Lemma \ref{non-dr-bound} gives $n_{1}, n_{2}\in\mathbb N$ such that $I_{e}\subseteq I_{p_{n_{1}}}$ and $I_{f}\subseteq I_{p_{n_{2}}}$. This implies that $I_{e}+I_{f} \subseteq I_{p_{n_{1}}} + I_{p_{n_{2}}} = I_{p_{n_{0}}}$, where $n_{0}=\max(n_1,n_2)$. Then $I_e+I_f$ is a proper subset of $V(J)$, and hence applying $\theta^{-1}$, $[e]+[f]\neq [1_{\calM(J)}]$ in $V(\calM(J))$. 

By the definition in \eqref{ideal-def}, the closed ideal $\calI$ is given by 
\begin{equation} \label{I_0_defn}  
\calI \coloneqq \overline{\text{Ideal}} \{e\in \calM(J): e \text{ is a projection with } [e]\in \calL\}. 
\end{equation} 
As $\calI$ contains all projections in $J$ and $J$ has real rank zero, we have $J\subseteq \calI$. Proposition \ref{ideal_corr} ensures $\calI$ is uniquely determined by the projections it contains. As $\calL$ is the maximal proper order ideal of $V(\calM(J))$, it follows (again from Proposition \ref{ideal_corr}), that $\calI$ is a maximal proper ideal in $\calM(J)$. 
\end{proof} 

We complete the proof of Theorem \ref{posi_ideal} by extending the characterisation of those projections not in $\calI$ to all positive elements. Here we need to use the assumption that  $K_1(J)=0$, so that $\calM(J)$ also has real rank zero by a result of Lin (\cite{RR-LHX}).

\begin{proof}[Proof of Theorem \ref{posi_ideal}] 
Let $\calL$ and $\calI$ be as defined in Lemma \ref{proj-ideal} and take any $b\in\calM(J)_+$.  If $b\sim 1_{\calM(J)}$, then $1_{\calM(J)}\notin \calI$ implies $b\notin \calI$. For the other direction, suppose $b\notin \calI$. Since $J$ is separable with real rank zero, stable rank one and $K_{1}(J) = 0$, by \cite[Theorem 10]{RR-LHX}, the multiplier algebra $\calM(J)$ also has real rank zero. By \cite[Lemma 2.3]{proj-decomp}, for any  $\epsilon >0$, there exists mutually orthogonal projections $f_{1}, \cdots, f_{l}\in\calM(J)$ and $\alpha_{1}, \cdots, \alpha_{l}>0$, such that 
\begin{equation}
0\leq \sum_{i=1}^{l}\alpha_{i}f_{i} \leq b \quad \text{and} \quad \left\|\sum_{i=1}^{l}\alpha_{i}f_{i}-b\right\| < \epsilon. 
\end{equation}
Since $b\notin \calI$, there exists $\epsilon > 0$ such that the corresponding sum $b'=\sum_{i=1}^{l}\alpha_{i}f_{i}$ is not in $\calI$. Then the projection $e = \sum_{i=1}^{l} f_{i}$ is Cuntz equivalent to $b'$ and so is not in $\calI$. By Lemma \ref{proj-ideal}, we have $[e] = [1_{\calM(J)}]$ in $V(\calM(J))$ and thus 
\begin{equation}
\langle 1_{\calM(J)}\rangle = \langle e \rangle = \langle b' \rangle\leq \langle b \rangle \quad \text{in } \mathrm{Cu}(\calM(J)). 
\end{equation} 
In conclusion, every positive element in $\calM(J) \setminus  \calI$ is Cuntz equivalent to $1_{\calM(J)}$, which implies that the extension $\calI\lhd \calM(J)$ is purely large (relative to $J$). 

Uniqueness of such ideal follows from Lemma \ref{proj-ideal}.  Lastly, any non-zero positive element $\overline{x}$ in the quotient $\calM(J) / \calI$ has a positive lift $x\in \calM(J)\setminus \calI$. Since $x\sim 1_{\calM(J)}$, it follows that $\overline{x}\sim 1_{\calM(J) / \calI}$. Therefore, $\calM(J) / \calI$ is simple and purely infinite.
\end{proof}

\section{$K_{1}$-injectivity of relative commutants}  
\label{K1_Paschke_chapter} 
In this section we prove the main $K_1$-injectivity theorem. Recall that for a $C^*$-algebra $J$, we denote the canonical quotient map by $\pi: \calM(J)\rightarrow \calC(J)$. For an element $a\in \calM(J)$, we denote $\pi(a)$ by $\overline{a}$. Similarly, if $D$ is a $C^*$-subalgebra of $\calM(J)$, we denote $\pi(D)$ by $\overline{D}$, and for a $^*$-homomorphism $\phi\colon A\to \calM(J)$, we denote $\pi \circ \phi$ by $\overline{\phi}$. 

\begin{thm} \label{property_thm} 
Let $A$ be a unital and separable $C^*$-algebra, and let $J$ be a separable and stable $C^*$-algebra with the corona factorisation property. Moreover, suppose there exists a closed proper ideal $\calI \lhd\calM(J)$ containing $J$ such that the extension of $\calI$ in $\calM(J)$ is purely large relative to $J$ (see Definition \ref{defn_PL}). Let $\phi: A\rightarrow \calM(J)$ be a unital, full and weakly nuclear $^*$-homomorphism. Then $\calC(J) \cap \overline{\phi}(A)'$ is $K_{1}$-injective. 
\end{thm} 

Combining Theorem \ref{property_thm} with the existence of a
purely large, and thus relatively purely large ideal in Theorem \ref{posi_ideal}, gives the following Corollary (Theorem \ref{K1_inj_main_4.2_intro}).  

\begin{cor} \label{RR0-K1-inj}
Let $A$ be a unital and separable $C^*$-algebra. Let $J$ be a separable and stable $C^*$-algebra which has real rank zero, stable rank one, $K_{1}(J) = 0$ and totally ordered $V(J)$. Let ${\phi: A\rightarrow \calM(J)}$ be a unital, full and weakly nuclear $\ast$-homomorphism. Then $\calC(J)\cap \overline{\phi}(A)'$ is $K_{1}$-injective. 
\end{cor} 

Theorem \ref{property_thm} also applies when $J$ is simple and has a unique trace with respect to which it has strict comparison as the ``trace class ideal'' is the intermediate relatively purely large ideal needed by the theorem (see Proposition \ref{unique_trace_PL}). In this way, Theorem \ref{property_thm} extends the unique trace case of the $K_1$-injectivity theorem of Loreaux, Ng and Sutradhar (\cite[Theorem 3.28]{K1-injective}) by removing the simplicity assumption on the domain $A$, and generalizing from the assumption of nuclear domains to weakly nuclear maps. 

\begin{cor}
Let $A$ be a unital and separable $C^*$-algebra. Let $J$ be a simple, separable and stable $C^*$-algebra with the corona factorisation property and a unique trace with respect to which $J$ has strict comparison. Let $\phi: A\rightarrow \calM(J)$ be a unital, full and weakly nuclear $\ast$-homomorphism. Then $\calC(J)\cap \overline{\phi}(A)'$ is $K_{1}$-injective. 
\end{cor}  

The strategy behind the proof of Theorem \ref{property_thm} is heavily inspired by the work of Loreaux, Ng and Sutradhar, using our abstract relative pure largeness hypothesis. As set out in Lemma \ref{sufficient_K1}, to prove Theorem \ref{property_thm},\footnote{Note that weak nuclearity is  only needed when invoking Lemma \ref{sufficient_K1}, to get unitary equivalence modulo $J$ between unitally nuclearly absorbing maps. The proof of Lemma \ref{main} does not involve (weak) nuclearity.} it suffices to prove the lemma below, and this is the goal of the remainder of the section. As the argument to establish the lemma is somewhat involved, we give an expository outline in Subsection \ref{overall_plan_section}, and then give the details of the proof in Subsection \ref{main_proof_section}. In what follows, for a unital $C^*$-algebra $D$, we denote by $\calU(D)$ the group of unitaries in $D$ and $\calU^{0}(D)$ the group of unitaries homotopic to $1_{D}$ in $\calU(D)$. 

\begin{lem}  \label{main} 
Let $A$ be a unital and separable $C^*$-algebra, and let $J$ be a separable and stable $C^*$-algebra. Suppose there is a proper ideal $\mathcal I\lhd \mathcal M(J)$ containing $J$ such that the extension of $\calI$ is purely large relative to $J$. Let $\phi: A\rightarrow \calM(J)$ be a unital and full $^*$-homomorphism. Then for any unitary $\overline{u}\in \calC(J) \cap \overline{\phi}(A)'$, there exists a unitary $\overline{v}\in \calU^{0}(\calC(J) \cap \overline{\phi}(A)')$ such that the inclusion $\iota: C^{*}(\overline{\phi}(A), \overline{v}\overline{u}) \rightarrow \calC(J)$ is full. 
\end{lem}

\subsection{Outline of the proof of Lemma \ref{main}}  
\label{overall_plan_section} 

The following outline is based on our deconstruction of the arguments of \cite{K1-injective}. Throughout the rest of the section, we fix $A, J$ and $\phi\colon A\to\calM(J)$ all satisfying the assumptions of Lemma \ref{main}. We also fix a unitary $\overline{u}\in\calC(J)\cap \overline{\phi}(A)'$ with a contractive lift $u\in\calM(J)$. 

We start by discussing how to construct  unitaries in $\calU^{0} (\calC(J) \cap \overline{\phi}(A)')$. In the case of main interest to us, where $J$ has real rank zero, stable rank one and $K_1(J)$ vanishes, a consequence of Lin's theorem (\cite[Theorem 11]{RR-LHX}) gives a general form of unitaries in $\calU^0(\calC(J))$: for every unitary $\overline{v}\in\calU^{0}(\calC(J))$, there exists an increasing approximate unit $(p_{n})_{n}$ of $J$ consisting of projections (by convention we take $p_0=0$) and a sequence of scalars $(\alpha_{n})_{n}$ in $\mathbb{T}$ such that $\overline{v}$ is the image under $\pi$ of the unitary 
\begin{equation} 
\sum_{n=1}^{\infty} \alpha_{n} (p_{n} - p_{n-1}),
\end{equation} 
which strictly converges (see Lemma \ref{almost_ortho_approx}) in $\calM(J)$.  

In general, a unitary in $\calC(J)$ constructed of the form above need not commute with $\overline{\phi}(A)$, but it will (after passing to a subsequence) if the approximate unit of projections $(p_{n})_{n}$ are \emph{quasicentral} with respect to $\phi(A)$, meaning that $\|[\phi(a), p_{n}]\| \rightarrow 0$ as $n\rightarrow \infty$ for any $a\in A$. However, the existence of a quasicentral approximate unit of projections is a very strong assumption and rarely satisfied, and we instead work with \emph{almost idempotent quasicentral approximate units} $(e_{n})_{n}$ of $J$ with respect to $\phi(A)$ (see Subsection \ref{Multi_rr0_sec}). These always exist (see \cite[Theorem I.9.16]{Davidson}), and for any sequence $\alpha = (\alpha_{n})_{n}$ in $\mathbb{T}$, the element of the bidiagonal form (with the convention $e_{0}=0$) 
\begin{equation} 
v_{\alpha} \coloneqq \sum_{n=1}^{\infty} \alpha_{n} (e_{n} - e_{n-1}) 
\end{equation} 
converges in the strict topology by Lemma \ref{almost_ortho_approx}. While $v_{\alpha}$ is not necessarily a unitary in $\calM(J)$, its image $\overline{v}_{\alpha}\coloneqq \pi(v_\alpha)$ becomes a unitary in $\calC(J)$ provided the $\alpha_n$ get close together. The following proposition, which essentially combines the results of \cite[Lemma 3.2, Lemma 3.3]{K1-injective}, sets this out precisely.

\begin{prop}
\label{unitary_commutant} 
Let $J$ be a separable $C^*$-algebra, and let $S$ be a separable $C^*$-subalgebra of $\mathcal M(J)$. Then there exists an almost idempotent approximate unit $(e_n)_{n}$ for $J$ such that for any bounded sequence $(\alpha_n)_n$ in $\mathbb{C}$, the strictly convergent series
\begin{equation}
\label{v_alpha} 
v_\alpha\coloneqq\sum_{n=1}^\infty \alpha_n (e_n-e_{n-1}), 
\end{equation} 
is an element in $\calM(J)$, whose image $\overline{v}_{\alpha}$ in $\calC(J)$ lies in the relative commutant $\mathcal C(J)\cap\overline{S}'$.  Moreover, if $(\alpha_{n})_{n}$ is a sequence in $\mathbb{T}$ such that $\lim_{n\to\infty}|\alpha_n-\alpha_{n+1}|=0$, then $\overline{v}_\alpha$ is a unitary in $\calC(J)\cap\overline{S}'$.
\end{prop} 

\begin{proof} 
Fix $(x_{k})_{k}$ a countable dense subset of $S$. Take an almost idempotent approximate unit $(e_{n})_{n}$ of $J$ with $\|[e_{n}, x_{k}]\|<1/2^{n+k}$ for any $n, k\in \mathbb{N}$ (by taking a subsequence of an almost idempotent quasicentral approximate unit given by \cite[Theorem I.9.16]{Davidson}). Then direct compuations (see \cite[Lemma 3.3]{K1-injective} and \cite[Lemma 3.2]{K1-injective}) show that $\overline{v}_{\alpha}$ is a unitary in $\calC(J) \cap \overline{S}'$.
\end{proof} 

We now fix such an approximate unit $(e_n)_{n}$ as in Proposition \ref{unitary_commutant} for $S=\phi(A)$ for the rest of the argument,\footnote{We will be explicit about this in the statements of the lemmas which follow.} and use the notation $v_{\alpha}$ for the element given by \eqref{v_alpha} for a sequence $\alpha = (\alpha_{n})_{n}$ in $\mathbb{T}$ and $\overline{v}_\alpha\coloneqq \pi(v_\alpha)$. To prove Lemma \ref{main}, it suffices to find a sequence $\alpha = (\alpha_n)_n$ such that $\overline{v}_{\alpha}$ lies in $\calU^{0}(\mathcal C(J)\cap \overline\phi(A)')$ and the inclusion $\iota: C^*(\overline{\phi}(A),\overline{v}_{\alpha} \overline{u})\rightarrow \calC(J)$ is full. The first part of this is the easy condition to achieve: one ensures that the path obtained by connecting successive terms of the sequence $(\alpha_n)_{n}$ never wraps fully round the circle, so that it can be continuously deformed back to the constant sequence.\footnote{In \cite{K1-injective}, Loreaux, Ng and Sutradhar called such a sequence  $\alpha = (\alpha_{n})_{n}$ in $\mathbb{T}$ \emph{unit oscillating} when the arc connecting adjacent terms of the sequence never crosses the unit.} The following proposition captures the essential idea of this part of Loreaux, Ng and Sutradhar's argument. 

\begin{prop} 
\label{unit_osci} 
Let $A$ and $J$ be separable $C^*$-algebras with $A$ unital. Let $\phi: A\rightarrow \calM(J)$ be a unital $^*$-homomorphism and take $S = \phi(A)$. Let $(e_n)_{n}$ be the approximate unit given by Proposition \ref{unitary_commutant}. If $(\theta_n)_{n}$ is a sequence in $[0,2\pi)$ with $\lim_{n\rightarrow\infty}|\theta_{n+1}-\theta_{n}| = 0$, then $\overline{v}_{\alpha}$ is a unitary in $\calU^0(\calC(J)\cap\overline\phi(A)')$, where $\alpha = (e^{i\theta_{n}})_{n}$. 
\end{prop} 

\begin{proof} 
For each $t\in [0,1]$, define $\alpha_{n}(t) \coloneqq e^{it\theta_{n}}$ and take the sequence $\alpha(t) = (\alpha_{n}(t))_{n}$ in $\mathbb{T}$. Since $\lim_{n\rightarrow \infty}|\theta_{n+1}-\theta_{n}| = 0$, it follows that $\lim_{n\rightarrow \infty}|\alpha_{n+1}(t)-\alpha_{n}(t)| = 0$ for each $t\in [0,1]$. By the choice of $(e_{n})_{n}$ and Proposition \ref{unitary_commutant}, we have that $\overline{v}_{\alpha(t)}$ is a unitary in $\calC(J)\cap \overline{\phi}(A)'$ for each $t\in [0,1]$. Moreover, $|\alpha_{n}(s) - \alpha_{n}(t)| < 2\pi |s-t|$ for any $n\in \mathbb{N}$ and thus $t\mapsto \overline{v}_{\alpha(t)}$ is a continuous path of unitaries, by Lemma \ref{almost_ortho_approx}, connecting $1_{\calC(J)}$ and $\overline{v}_{\alpha(1)} = \overline{v}_{\alpha}$. 
\end{proof} 

The main challenge is to arrange for the inclusion map $\iota$ to be full. In \cite{K1-injective}, taking advantage of the simplicity assumption of $A$, and using a density argument, the authors showed that it is sufficient to verify fullness of $h_{k}(\overline{v}_{\alpha}\overline{u})$ in $\calC(J)$ for a dense collection $(h_k)_k$ of nonzero positive contractions in $C(\mathbb{T})$. Similarly, to prove that $\iota$ is full, we verify the fullness of $h_k(\overline{v}_\alpha \overline{u})\overline{\phi}(a_k)$ in $\calC(J)$ for countably many positive functions $(h_{k})_{k}$ in $C(\mathbb{T})$ and non-zero positive elements $(a_{k})_{k}$ in $A$, which are fixed for the rest of this subsection. This is a consequence of the following lemma, applied to $D$ being $\overline{\phi}(A)$, $C$ being $\calC(J)$ and the unitary being $\overline{v}_{\alpha} \overline{u}$ for the fixed unitary $\overline{u}$ and some chosen sequence $\alpha$ in $\mathbb{T}$. The lemma is proved using a standard density argument, in exactly the same fashion as the version of the lemma given in the second authors DPhil thesis \cite[Lemma 6.2.7]{Hua_thesis}.

\begin{lem} \label{reduce-countable} 
Let $D\subseteq C$ be an inclusion of $C^*$-algebras with $D$ separable, and let $u\in C\cap D'$ be a unitary.  There exist sequences $(d_k)_k$ of non-zero positive contractions in $D$, and $(h_k)_k$ of norm one positive functions in $C(\mathbb T)$, such that if $h_k(u)d_k$ is full for each $k$, then the inclusion map $\iota\colon C^*(D,u)\to C$ is full.
\end{lem} 

It remains to construct $\alpha$ such that $h_{k}(\overline{v}_{\alpha}\overline{u})\overline{\phi}(a_{k})$ is full in $\calC(J)$, for each $k$. We would like to do this by considering lifts of these elements in $\calM(J)$, since it is easier to manipulate elements in $\calM(J)$ as compared with $\calC(J)$. However, while $v_{\alpha} u$ is a lift of the unitary $\overline{v}_\alpha \overline{u}$, it is not necessarily a unitary itself, so one cannot make sense of the functional calculus $h(v_\alpha u)$ for $h\in C(\mathbb{T})$. For this reason, it is useful to work with \emph{Laurent polynomials} $\hat{h}$: 
\begin{equation}
\hat{h}(z) = \sum_{s=0}^{S_{1}} \beta_{s} z^{s} + \sum_{s=1}^{S_{2}}\beta_{-s}\bar{z}^{s},
\end{equation} 
where $S_{1}, S_{2}\in \mathbb{N}$ and $\beta_{s}\in \mathbb{C}$, as there is no ambiguity regarding the definition of 
\begin{equation}
\hat{h}(x) = \sum_{s=0}^{S_{1}}\beta_{s}x^{s} + \sum_{s=1}^{S_{2}}\beta_{-s}(x^{\ast})^{s} 
\end{equation} 
for an element $x$ in any $C^*$-algebra. Since Laurent polynomials, when restricted to $\mathbb{T}$,
are uniformly dense in $C(\mathbb{T})$, we approximate each of the $h_{k}$ given by Lemma \ref{reduce-countable} by a Laurent polynomial $\hat{h}_{k}$. Taking care with the approximations, it 
will suffice to construct $\alpha$ such that $ \hat{h}_{k}(v_{\alpha} u) \phi(a_{k})$ is full in $\calM(J)$ for each $k\in \mathbb{N}$. 

Temporarily fix $k$. Then as $h_k$ is a non-zero positive function, there is a finite collection $\mu_1,\dots,\mu_{s}\in\mathbb T$ and $\kappa>0$ so that the sum of the translates $\sum_{i=1}^{s}h_k(\mu_{i}t)>\kappa$ for all $t\in\mathbb T$. Since $a_{k}\neq 0$ and $\phi$ is full, then $\overline{\phi}(a_{k})$ is full in $\calC(J)$ and thus $\sum_{i=1}^{s}h_k(\mu_{i}\overline{u}) \overline{\phi}(a_{k})\geq \kappa \overline{\phi}(a_{k})$ is full in $\calC(J)$. Hence at least one of the lifts of $h_k(\mu_{i}\overline{u}) \overline{\phi}(a_{k})$ does not lie in the ideal $\calI$ specified by Lemma \ref{main}. Take $\lambda_{k}$ to be a coefficient among $\mu_{1}, \cdots, \mu_{s}$ such that the lift of $h_k(\lambda_{k}\overline{u}) \overline{\phi}(a_{k})$ does not lie in $\calI$. 

Next, we use the abstract hypothesis on $\calI$ that the extension of $\calI$ in $\calM(J)$ is purely large relative to $J$. This ensures that any positive lift of $h_k(\lambda_{k}\overline{u}) \overline{\phi}(a_{k})$ Cuntz dominates all the elements in $J$. Using the fact that $\hat{h}_{k}(\lambda_{k} u)\phi(a_{k})$ approximates a positive lift of $h_k(\lambda_{k}\overline{u}) \overline{\phi}(a_{k})$,\footnote{Note that $\hat{h}_{k}(\lambda_{k} u)\phi(a_{k})$ need not be positive itself.} for  starting indices $n$ and $m$, and any $b\in J_{+}$, we can find a contraction $x\in \overline{bJ (e_{m'} - e_{m})}$ for some $m' > m$ and some $n'>n$, with
\begin{equation} 
\label{find_approx} 
x \hat{h}_{k}(v_\alpha u) \phi(a_{k})x^* \approx b,\footnote{This means that two elements are close in norm, and quantifiers will be specified in later proofs.} 
\end{equation} 
for any sequence $\alpha = (\alpha_n)_{n}$ which takes the constant value $\lambda_k$ on the block of indices from $n$ to $n'$. See Lemma \ref{finite_block_approx} for the details of this part of the proof. 

We will apply the procedure above to each $k$ infinitely many times. This will produce $n_1<n_1'<n_2<n_2'<\dots$, dividing the index set into infinitely many \emph{blocks} $B_r=\{n_r,n_r+1,\dots,n_r'\}$ separated by \emph{gaps} of length $n_{r+1} - n_{r}'$, which can be made as long as we like.  We will define the values $\alpha_n$ on the blocks to be the constants $\lambda_k$ (with each value $\lambda_k$ being assigned to infinitely many distinct blocks), and fill in the values of $\alpha_{n}$ in the gaps between the blocks so that the condition of Proposition \ref{unit_osci} is satisfied (possible as the gap length can be made arbitrarily long).  Accordingly $\overline{v}_\alpha$ is indeed a unitary in $\calU^{0}(\calC(J)\cap\overline\phi(A)')$.  For a fixed $k$, we can find elements $(x_\ell)_{\ell=1}^\infty$ with orthogonal supports\footnote{This is possible by taking the $b$ in \eqref{find_approx} to be $e_{\ell} - e_{\ell-1}$, and ensuring the starting index $m$ used at each stage of the induction is at least one more than the `finishing index' $m'$ of the previous stage.} so that 
\begin{equation}
    x_{\ell}\hat h_k(v_\alpha u)\phi(a_k)x_{\ell}^*\approx e_{\ell}-e_{\ell-1}.
\end{equation} 
The orthogonality of the supports of the $x_{\ell}$, and the almost idempotent nature of $(e_n)_{n}$ ensure that $x=\sum_{\ell} x_{\ell}$ is strictly convergent, and (taking suitable care with the estimates),
\begin{equation}
    x\hat{h}_k(v_\alpha u)\phi(a_k)x^*\approx 1_{\mathcal M(J)}.
\end{equation}
This ensures that $h_{k}(\overline{v}_{\alpha} u)\overline{\phi}(a_{k})$ is full in $\calC(J)$. 

By ensuring every index $k$ is assigned to infinitely many distinct blocks (and all the corresponding $e_{m'}-e_m$ are mutually orthogonal), the procedure above can be done for all indices simultaneously. This allows us to find a single sequence $\alpha$.

\subsection{Proof of Lemma \ref{main}}  \label{main_proof_section} 

The following technical but straightforward\footnote{It amounts to standard calculations with approximate units, and is readily verified when $\hat{h}(z)$ is a monomial: either $z^s$ or $\bar{z}^s$.} lemma for working with Laurent polynomials is stated as \cite[Lemma 3.15]{K1-injective}. As remarked there, a version of the lemma is proved as \cite[Lemma 4.2, Lemma 4.3]{ess-codim} in the case that $(e_n)_n$ consists of projections, and the general proof works in a very similar fashion. For elements $a, b\in A$ and $\epsilon > 0$, we denote $\|a-b\| < \epsilon$ by $a\approx_{\epsilon} b$. 

\begin{lem}[{\cite[Lemma 3.15]{K1-injective}}]
\label{finite_block_Laurent} 
Let $J$ be a $C^*$-algebra with an almost idempotent approximate unit $(e_{n})_{n}$ and a contraction $u$ in $\calM(J)$. For any Laurent polynomial $\hat{h}$, $\epsilon > 0$, integers $n_{0}, m_{0}>0$, $b\in J$, $a\in \calM(J)$, the following properties hold, 
\begin{enumerate} [(i)] 
\item There exists an integer $m_{1}> m_{0}$ such that 
\begin{equation} 
(1-e_{m_{1}})  \hat{h}(v_{\alpha}u) \approx_{\epsilon} (1-e_{m_{1}})\hat{h}(v_{\beta}u), 
\end{equation}  
for any sequences $\alpha = (\alpha_{n}),_{n} \beta = (\beta_{n})_{n}\subseteq \mathbb{T}$ with $\alpha_{n} = \beta_{n}$ for $n\geq n_{0}$. 
\item There exists an integer $n_{1} > n_{0}$ such that 
\begin{equation} 
\hat{h}(v_{\alpha}u)b \approx_{\epsilon} \hat{h}(v_{\beta}u)b, 
\end{equation}  
for any sequences $\alpha = (\alpha_{n})_{n}, \beta = (\beta_{n})_{n}\subseteq \mathbb{T}$ with $\alpha_{n} = \beta_{n}$ for $n\leq n_{1}$. 
\item There exists an integer $m_{2} > m_{0}$ such that for any sequences $\alpha = (\alpha_{n})_{n}\subseteq \mathbb{T}$, we have 
\begin{equation} 
\|e_{m_{0}}\hat{h}(v_{\alpha}u) a (1-e_{m_{2}}) \| < \epsilon \text{ and } \|(1-e_{m_{2}})\hat{h}(v_{\alpha}u) a e_{m_{0}}\| < \epsilon. 
\end{equation}  
\end{enumerate} 
\end{lem}

We now turn to each of the inductive steps in the proof of Lemma \ref{main}. For $\epsilon>0$ and a positive contraction $a$, we write $(a-\epsilon)_+$ for the element $g_\epsilon(a)$, where $g_\epsilon(t)=\max(t-\epsilon,0)$. We let $f_\epsilon$ be the positive continuous function on $[0, \infty)$ given by
\begin{equation} 
\label{stair_function} 
f_{\epsilon}(t) \coloneqq
\begin{cases} 
t/\epsilon, & 0\leq t\leq \epsilon; \\
1, & t\geq \epsilon,
\end{cases} 
\end{equation}
so that $f_\epsilon(a)(a-\epsilon)_+=(a-\epsilon)_+$ for all positive contractions $a$. For completeness, we (re)state all the assumptions involved in the lemma. 

\begin{lem} \label{finite_block_approx}  
Let $A$ be a unital and separable $C^*$-algebra and let $J$ be a stable $C^*$-algebra with an almost idempotent approximate unit $(e_{n})_{n}$. Suppose there exists a proper ideal $\calI\subseteq\calM(J)$ containing $J$ such that the extension of $\calI$ in $\calM(J)$ is relatively purely large with respect to $J$. Let $\phi: A \rightarrow\calM(J)$ be a unital and full $^*$-homomorphism. Let $\overline{u}$ be a unitary in $\calC(J)\cap \overline{\phi}(A)'$ with a contractive lifting $u\in \calM(J)$. 

Let $a\in A_{+}$ be a nonzero positive contraction and let $h\in C(\mathbb T)_+$ have norm one. There exists $\delta>0$ and $\lambda\in\mathbb T$ such that for any $m_{0}, n_{0}\in \mathbb{N}$, any positive contraction $b_{0}\in J$, $\epsilon>0$ and Laurent polynomial $\hat{h}$ with $\|f_{1/3}\circ h-\hat{h}|_{\mathbb{T}}\|_{\infty} < {\epsilon}$, there exist $m_{2}>m_{1}>m_{0}$, $n_{1}>n_{0}$ and a contraction $x\in \overline{b_{0}J(e_{m_{2}}-e_{m_{1}})}$ such that 
\begin{equation} 
x\hat{h}(v_{\alpha}u) f_{\delta}(\phi(a)) x^{\ast}\approx_{ \epsilon}  b_{0},  
\end{equation} 
whenever $\alpha = (\alpha_{n})_{n}$ is a sequence in $\mathbb{T}$ with $\alpha_{n} = \lambda$ for $n_{0}\leq n \leq n_{1}$. 
\end{lem}   

\begin{proof} 
Fix $a\in A$ to be a non-zero positive contraction and $h\in C(\mathbb T)_+$ with norm one. Since $a\neq 0$ and $\phi$ is full, there exists $\delta>0$ such that $(\phi(a) - \delta)_{+} = \phi((a-\delta)_{+})$ is full in $\calM(J)$. As $h\in C(\mathbb{T})$ is positive and has norm one, by choosing a neighborhood in $\mathbb{T}$ where the value of $h$ is strictly bigger than $1/3$ and rotating it by sufficiently many points on $\mathbb{T}$, we can find $\mu_{1}, \cdots, \mu_{s}\in \mathbb{T}$ such that 
\begin{equation} 
0 < \kappa \leq \sum_{i=1}^{s} (h-1/3)_{+}(\mu_{i} t), \ \ t\in \mathbb{T}, 
\end{equation}  
for some constant $\kappa>0$. Since ${\overline{u}}\in \calC(J)$ commutes with $\overline{\phi}(A)$, then 
\begin{equation} 
\label{bdd_below} 
0 \leq \kappa \cdot (\overline{\phi} (a) - \delta)_{+} \leq \sum_{i=1}^{s} (h-1/3)_{+}(\mu_{i} \overline{u}) (\overline{\phi} (a) - \delta)_{+}. 
\end{equation} 
Since $(\phi(a) - \delta)_{+}\in \calM(J)$ is a full element, the sum on the right-hand side of (\ref{bdd_below}) is full in $\calC(J)$, and so does not lie in the proper ideal $\pi (\calI)\lhd \mathcal C(J)$, where $\pi: \calM(J)\rightarrow \calC(J)$ is the quotient map. Thus there exists $i_{0}\in \{1, \dots, s\}$ such that 
\begin{equation} 
\overline{d} \coloneqq (h-1/3)_{+}(\mu_{i_{0}} \overline{u}) (\overline{\phi} (a) - \delta)_{+}\notin \pi (\calI). 
\end{equation} 
Fix $\lambda\coloneqq \mu_{i_0}$, and write 
\begin{equation}
\overline{c} \coloneqq (f_{1/3}\circ h)(\lambda \overline{u})f_\delta(\overline{\phi}(a)), 
\end{equation} 
where $f_{1/3}$ and $f_{\delta}$ are the positive functions of norm one defined in \eqref{stair_function}. Then $\overline{c}$ and $\overline{d}$ are contractions in $\calC(J)_{+}$ satisfying $\overline{c}\overline{d}=\overline{d}$ since $\overline{u}\in \calC(J)\cap \overline{\phi}(A)'$. Moreover, applying the canonical map $\mathcal C(J)\to\mathcal M(J)/\mathcal I$, the image of $\overline{c}$ in $\mathcal M(J)/\mathcal I$ acts as the unit on the image of $\overline{d}$ (which is non-zero as $\overline{d}\notin\pi(\mathcal I))$.  Accordingly 
\begin{equation}
\label{induction.normcontrol}
\|\overline{c}+\mathcal I\|=1\text{ in }\mathcal M(J)/\mathcal I.
\end{equation} 

Let $m_0,n_0\in\mathbb N$, $\epsilon>0$, $b_0\in J$ be a positive contraction and $\hat{h}$ be a Laurent polynomial with $\|f_{1/3}\circ h-\hat{h}|_{\mathbb T}\|_\infty<\epsilon$. Let $c\in \calM(J)$ be a contractive positive lift of $\overline{c}$. For any sequence $\alpha = (\alpha_{n})_{n}$ with $\alpha_n = \lambda$ for all $n\geq n_{0}$, we have $\overline{v}_{\alpha}=\lambda 1_{\mathcal C(J)}$ and thus
\begin{equation} \label{Corona2} 
\pi \big( \hat{h} (v_{\alpha} u)\big) = \hat{h} (\overline{v}_{\alpha} \overline{u} ) = \hat{h} (\lambda \overline{u}) = \pi \big(\hat{h} (\lambda u)\big). 
\end{equation} 
Take $0 < \mu < \epsilon - \|f_{1/3}\circ h-\hat{h}|_{\mathbb T}\|_\infty$, then by the choice of $\mu$ and $\hat{h}$, 
\begin{equation} \label{corona3}
\pi(c) = \overline{c} \approx_{\epsilon-\mu} \hat{h}(\lambda \overline{u}) f_{\delta}(\overline{\phi}(a)) \stackrel{\eqref{Corona2}}{=} \pi(\hat{h}(v_\alpha u)f_\delta(\phi(a))). 
\end{equation} 
As $(e_{n})_{n}$ is an almost idempotent approximate unit of $J$, we can approximately lift \eqref{corona3} to elements in the multiplier algebra. By applying Lemma \ref{finite_block_Laurent} (i), there exists $m_{1}>m_{0}$ such that 
\begin{equation} \label{corona4}
(1-e_{m_{1}})c(1-e_{m_{1}})\approx_{\epsilon - \mu} (1-e_{m_{1}})\hat{h}(v_\alpha u)f_\delta(\phi(a))(1-e_{m_{1}}), 
\end{equation} 
for any sequence $\alpha = (\alpha_{n})_{n}$ with $\alpha_n = \lambda$ for all $n\geq n_{0}$. 

Note that as $(1-e_{m_1})c(1-e_{m_1})$ is a lift of $\overline{c}$ in $\calM(J)$, \eqref{induction.normcontrol} implies that
\begin{equation}
\|q((1-e_{m_1})c(1-e_{m_1}))\| = \|\overline{c} + \calI \|=1\text{ in }\mathcal M(J)/\mathcal I, 
\end{equation} 
where $q: \calM(J) \rightarrow \calM(J)/\mathcal I$ is the quotient map. As $\mathcal I$ is relatively purely large in $\calM(J)$ with respect to $J$, Lemma \ref{relative_purely_large_equi} provides a contraction $y\in J$ with 
\begin{equation} 
y(1-e_{m_1})c(1-e_{m_1})y^*\approx_{\mu/2} b_0.
\end{equation} 
Combining this with \eqref{corona4}, we have
\begin{equation} \label{multiplier1} 
    y(1-e_{m_{1}}) \hat{h}( v_{\alpha} u) f_{\delta}(\phi(a)) (1-e_{m_{1}})y^*\approx_{\epsilon-\mu/2} b_0,
\end{equation}
for all sequences $(\alpha_n)_{n}$ with $\alpha_n=\lambda$ when $n\geq n_{0}$. Since $b_{0}^{1/\ell} b_{0} \rightarrow b_{0}$ when $\ell \rightarrow \infty$ and $(e_{n})_{n}$ is an approximate unit of $J$, there exists $\ell_{0}\in \mathbb{N}$ and $m_{2}>m_{1}$ such that when replacing $y(1-e_{m_{1}})$ in \eqref{multiplier1} by the contraction 
\begin{equation} 
x \coloneqq b_{0}^{1/\ell_{0}} y (e_{m_{2}}-e_{m_{1}})\in \overline{b_{0}B(e_{m_{2}}-e_{m_{1}})}, 
\end{equation} 
the approximation \eqref{multiplier1} still holds, namely, 
\begin{equation}  \label{multiplier2} 
x \hat{h}(v_{\alpha} u) f_{\delta}(\phi(a)) x^{\ast} \approx_{\epsilon - \mu /2} b_{0}, 
\end{equation} 
for all sequences $(\alpha_n)_{n}$ with $\alpha_n=\lambda$ when $n\geq n_{0}$. Lastly, since $x\in J$, by Lemma \ref{finite_block_Laurent} (ii), there exists $n_{1}>n_{0}$ such that 
\begin{equation}
x\hat{h}(v_\alpha u)\approx_{\mu/2} x\hat{h}(v_\beta u), 
\end{equation} 
whenever the sequences $(\alpha_n)_n$ and $(\beta_n)_n$ in $\mathbb T$ have $\alpha_n=\beta_n$ for all $n\leq n_{1}$. 
Combining this with \eqref{multiplier2}, we obtain 
\begin{equation}
    x \hat{h}(v_{\alpha} u) f_{\delta} (\phi(a)) x^*\approx_{\epsilon} b_0, 
\end{equation} 
for any $\alpha = (\alpha_{n})_{n}$ in $\mathbb{T}$ with $\alpha_{n} = \lambda$ for $n_{0}\leq n\leq n_{1}$. 
\end{proof} 

We now inductively apply Lemma \ref{finite_block_approx} to prove the main technical lemma (Lemma \ref{main}). 

\begin{proof}[Proof of Lemma \ref{main}]
Fix a unitary $\overline{u}\in \calC(J)\cap \overline{\phi}(A)'$ with a contractive lift $u$ in $\calM(J)$. Let $(e_n)_{n}$ be an almost idempotent approximate unit of $J$, sufficiently quasicentral with respect to $\phi(A)$ as given by Proposition \ref{unitary_commutant}, and write $e_0=0$. Fix $0< \epsilon < 1/7$. 

Let $(h_k)_{k}$ and $(a_k)_{k}$ be sequences of norm one positive functions in $C(\mathbb T)$ and non-zero positive contractions in $A$ respectively given by Lemma \ref{reduce-countable}. For each $k$, we fix a Laurent polynomial $\hat{h}_{k}$ so that $\|h_{k} - \hat{h}_{k}\|_{\infty} < \epsilon$, and take $\delta_{k}>0$ and $\lambda_{k}\in \mathbb{T}$ satisfying the conditions mentioned in Lemma \ref{finite_block_approx}. Finally, fix a bijection $\theta\colon\mathbb N\times \mathbb{N}\to\mathbb N$ such that for any $k\in \mathbb{N}$ and $\ell_{1}<\ell_{2}$, then $\theta(k, \ell_{1}) < \theta(k, \ell_{2})$. In the following proof, we perform an induction on $r$ and $\theta(k, \ell) = r$ means that the $r$-th inductive step is the $\ell$-th time we use the pair $(h_{k}, a_{k})$, and $e_{\ell} - e_{\ell-1}$ is approximated, see \eqref{MainProofC4} for details.

Combining Lemma \ref{finite_block_Laurent} and Lemma \ref{finite_block_approx}, it is possible to inductively construct
\begin{itemize} 
\item contractions $x_{r} \in J$ for every $r\in \mathbb{N}$; 
\item increasing sequences of natural numbers $n_1<n_1'<n_2<n_2'<\dots$ and $m_1<m_1'<m_2<m_2'<\dots$ such that $n_{r}\geq n_{r-1}' + r$, 
\end{itemize}
with the following properties for each $r\in\mathbb{N}$ and $k,\ell\in\mathbb{N}$ satisfying $\theta(k, \ell)= r$: 
\begin{enumerate}[(i)] 
\item\label{MainProofC3} $x_r\in \overline{(e_{\ell}-e_{\ell-1})J(e_{m_{r}'}-e_{m_r})}$; 
\item\label{MainProofC4} for any sequence $\alpha = (\alpha_n)_{n}$ in $\mathbb T$ with $\alpha_n=\lambda_k$ for $n_r\leq n\leq n_r'$, one has
\begin{equation}
x_r \hat{h}_{k}(v_{\alpha}u) f_{\delta_{k}}(\phi(a_{k})) x_r^*\approx_{\epsilon} e_{\ell}-e_{\ell-1};
\end{equation}
\item\label{MainProofC2} for any sequence $\alpha = (\alpha_n)_{n}$ in $\mathbb T$, the element $\hat{h}_{k}(v_{\alpha}u) f_{\delta_{k}}(\phi(a_{k}))$ is approximately block diagonal and its ``off-diagonal'' terms are norm small, in the sense that
\begin{align} \label{cross-approx1} 
\Big\| e_{m_{r-1}'+1} \hat{h}_{k}(v_{\alpha}u) f_{\delta_{k}}(\phi(a_{k})) (1-e_{m_r-1}) \Big\| &< {\frac{\epsilon}{\ell\cdot 2^{\ell+1}}} \text{ and} \\ 
\label{cross-approx2} 
\Big\| (1-e_{m_r-1}) \hat{h}_{k}(v_{\alpha}u) f_{\delta_{k}}(\phi(a_{k})) e_{m_{r-1}'+1} \Big\| &< \frac{\epsilon}{\ell\cdot 2^{\ell+1}}. 
\end{align} 
\end{enumerate} 
Indeed, assume these objects have been constructed for the first $r-1$ inductive steps (in the case $r=1$, we can formally commence with $n_0'=m_0'=0$). Suppose that $k, \ell$ are natural numbers such that $\theta(k, \ell) = r$. Then Lemma \ref{finite_block_Laurent} (iii) enables us to find some $\tilde{m}_r > m_{r-1}'$ such that for any $\alpha = (\alpha_{n})_{n}$ in $\mathbb{T}$, 
\begin{align}
\label{cross_approx3}
\Big\| e_{m_{r-1}'+1} \hat{h}_{k} (v_{\alpha} u) f_{\delta_{k}} (\phi(a_{k})) (1-e_{\tilde{m}_r}) \Big\| &< \frac{\epsilon}{\ell\cdot 2^{\ell+1}},\\
\label{cross_approx4} 
\Big\| (1-e_{\tilde{m}_r}) \hat{h}_{k} (v_{\alpha} u)f_{\delta_{k}}(\phi(a_{k})) e_{m_{r-1}'+1} \Big\| &< \frac{\epsilon}{\ell\cdot 2^{\ell+1}}. 
\end{align} 
The inequalities above also hold if we replace $\tilde{m}_r$ by any integer $m > \tilde{m}_r$, since $(1-e_{\tilde{m}_r})(1-e_{m}) = 1-e_{m}$ by the assumption that $(e_{n})_{n}$ is almost idempotent. Now we can apply Lemma \ref{finite_block_approx} to $\tilde{m}_r$ and some $n_{r}\coloneqq n_{r-1}'+r+1$ as starting indices, and find $m_r' > m_{r} > \tilde{m}_r$ (so that \eqref{MainProofC2} holds by the choice of $\tilde{m}_{r}$), together with $n_{r}' > n_{r}$ and $x_r$ such that \eqref{MainProofC3} and \eqref{MainProofC4} hold. 

For each $r\in \mathbb{N}$, suppose that $k, \ell$ are the natural numbers that $\theta(k, \ell) = r$. Then define $\beta_n=\lambda_{k}$ for those $n_r\leq n\leq n_r'$ lying in the $r$-th block. Since the gap length $n_{r} - n_{r-1}'>r$, we can fill in the values of $\beta_{n}$ in the gap so that they slowly move from the value in the $(r-1)$-th block to the value in the $r$-th block without winding around the circle. More precisely, suppose the constant value in the $(r-1)$-th block and $r$-th block are $\mu$ and $\nu$ respectively. Take $\theta_{n_{r-1}'} = \log \mu$ and $\theta_{n_{r}} = \log \nu$ in $[0,2\pi)$. For $n_{r-1}' \leq n\leq n_{r}$, define 
\begin{equation} 
\theta_{n} = \log \mu + \frac{n-n_{r-1}'}{n_{r} - n_{r-1}'} (\log \nu - \log \mu), 
\end{equation} 
and $\beta_{n} = e^{\theta_{n}}$. Since $n_{r} - n_{r-1}' > r$, it follows that $|\beta_{n} - \beta_{n+1}| < 2\pi / r$ for $n_{r-1}'\leq n\leq n_{r}$. By Proposition \ref{unit_osci}, we have constructed a sequence $\beta = (\beta_{n})_{n}$ in $\mathbb{T}$ such that $\overline{v}_\beta$ is a unitary in $\mathcal U^0(\mathcal C(J)\cap \overline\phi(A)')$.

It remains to show that the inclusion $\iota: C^{*}(\overline{\phi}(A), \overline{v}_{\beta}\overline{u})\rightarrow \calC(J)$ is full. By our choices of $(h_{k})_{k}$ and $(a_{k})_{k}$ from Lemma \ref{reduce-countable}, it suffices to show that for each $k\in \mathbb{N}$, the element $h_{k}(\overline{v}_{\beta} \overline{u}) \overline{\phi}(a_{k})$ is full in $\calC(J)$. Fix some $k\in \mathbb{N}$. Since $x_{r}$ is contractive for any $r\in\mathbb{N}$ and satisfies \eqref{MainProofC3}, the bidiagonal sum
\begin{equation} 
x \coloneqq \sum_{\ell=1}^{\infty} x_{\theta(k, \ell)} 
\end{equation} 
is strictly convergent in $\mathcal M(J)$ with $\|x\|\leq 2$ by Lemma \ref{almost_ortho_approx}. We will estimate 
\begin{equation} \label{big_sum} 
x\hat{h}_{k}(v_\beta u)f_{\delta_k}(\phi(a_k))x^*=\sum_{\ell_1,\ell_2=1}^\infty x_{\theta(k, \ell_{1})}\hat{h}_{k}(v_\beta u)f_{\delta_{k}} (\phi(a_{k}))x_{\theta(k, \ell_{2})}^*. 
\end{equation} 
Firstly, we show that the sum of the ``off-diagonal'' terms, namely the terms with $\ell_{1}\neq \ell_{2}$, is norm convergent with small norm. Suppose $\ell_{1}<\ell_{2}$, then $r_{1} \coloneqq \theta(k, \ell_{1}) < r_{2} \coloneqq\theta(k, \ell_{2})$ and by construction, 
\begin{equation} 
m_{r_{1}} < m_{r_{1}}' < m_{r_{2}-1}'+1 \leq m_{r_{2}} < m_{r_{2}}'. 
\end{equation} 
Since $x_{r_{1}}$ satisfies \eqref{MainProofC3} and $(e_{n})_{n}$ is almost idempotent, then $x_{r_{1}} e_{m_{r_{2}-1}'+1} = x_{r_{1}}$. Likewise, $x_{r_{2}}$ satisfies \eqref{MainProofC3} and $(e_{m_{r_{2}}'}-e_{m_{r_{2}}})(1-e_{m_{r_{2}}-1}) = e_{m_{r_{2}}'}-e_{m_{r_{2}}}$, it follows that $x_{r_{2}}(1-e_{m_{r_{2}}-1}) = x_{r_{2}}$. Putting these together, we have 
\begin{align} 
&&{}& x_{\theta(k, \ell_{1})}\hat{h}_k(v_\beta u)f_{\delta_k}(\phi(a_{k}))x_{\theta(k, \ell_{2})}^{*} \\
&&={}& x_{r_{1}}\hat{h}_k(v_\beta u)f_{\delta_k}(\phi(a_k))x_{r_{2}}^*\\ 
&&={}&
x_{r_{1}}e_{m_{r_{2}-1}'+1}\hat{h}_k(v_\beta u)f_{\delta_k}(\phi(a_k))(1-e_{m_{r_{2}}-1})x_{r_{2}}^*.
\end{align} 
Then by property \eqref{MainProofC2}, we have 
\begin{equation} 
\big\|x_{\theta(k, \ell_{1})}\hat{h}_k(v_\beta u)f_{\delta_k}(\phi(a_{k}))x_{\theta(k, \ell_{2})}^{*}\big\| < 
\frac{\epsilon}{\ell_{2}\cdot 2^{\ell_{2}+1}}.  
\end{equation} 
Similar calculations work in an identical fashion when $\ell_{1} > \ell_{2}$ and we get 
\begin{equation} 
\big\|x_{r_{1}}e_{m_{r_{2}-1}'+1}\hat{h}_k(v_\beta u)f_{\delta_k}(\phi(a_k))(1-e_{m_{r_{2}}-1})x_{r_{2}}^* \big\| < 
\frac{\epsilon}{\ell_{1}\cdot 2^{\ell_{1}+1}}. 
\end{equation} 
Summing up all of the ``off-diagonal'' terms gives us 
\begin{equation}\label{cross_approx_final} 
\bigg\|\sum_{\ell_1\neq \ell_2}x_{\theta(k, \ell_{1})}\hat{h}_{k}(v_\beta u)f_{\delta_{k}}(\phi(a_{k}))x_{\theta(k, \ell_{2})}^*\bigg\| < 2\sum_{\ell_2=1}^{\infty}\sum_{\ell_1<\ell_2}\frac{\epsilon}{\ell_2\cdot 2^{\ell_2+1}}<\epsilon. 
\end{equation}
For the ``diagonal'' terms in \eqref{big_sum}, condition \eqref{MainProofC4} and the definition of $\beta$ gives
\begin{equation} 
x_{\theta(k, \ell)} \hat{h}_{k} (v_{\beta} u) f_{\delta_{k}}(\phi(a_{k}))  x_{\theta(k, \ell)} ^{\ast}\approx_{\epsilon} e_{\ell}-e_{\ell-1}, 
\end{equation}  
for every $\ell\in \mathbb{N}$. Then Lemma \ref{almost_ortho_approx} gives, 
\begin{equation} 
\label{diagonal_approx_final} 
\sum_{\ell=1}^{\infty} x_{\theta(k, \ell)} \hat{h}_{k} (v_{\beta} u) f_{\delta_{k}}(\phi(a_{k}))  x_{\theta(k, \ell)} ^{\ast}\approx_{2\epsilon} \sum_{\ell=1}^{\infty} (e_{\ell}-e_{\ell-1}) = 1_{\calM(J)}, 
\end{equation} 
where the series converges strictly since it is bidiagonal. Combining the estimates \eqref{cross_approx_final} for the ``off-diagonal'' terms and \eqref{diagonal_approx_final} for these ``diagonal'' terms gives, 
\begin{equation}  
\label{multi_est_final} 
x \hat{h}_{k} (v_{\beta} u)f_{\delta_{k}}(\phi(a_k)) x^{\ast} \stackrel{\eqref{cross_approx_final}}{\approx_{\epsilon}} \sum_{\ell=1}^{\infty} x_{\theta(k, \ell)} \hat{h}_{k} (v_{\beta} u) f_{\delta_{k}}(\phi(a_{k}))  x_{\theta(k, \ell)} ^{\ast} \stackrel{\eqref{diagonal_approx_final}}{\approx_{2\epsilon}} 1_{\calM(J)}.
\end{equation}   

Apply the quotient map $\pi: \calM(J) \rightarrow \calC(J)$ to \eqref{multi_est_final}, and we obtain
\begin{equation}   
\overline{x} \hat{h}_{k} (\overline{v}_{\beta} \overline{u})f_{\delta_{k}}(\overline\phi(a_k)) \overline{x}^{\ast} \approx_{3\epsilon} 1_{\calC(J)}. \label{approx_final_6} 
\end{equation}   
Since $\|h_{k} - \hat{h}_{k}\|_{\infty} < \epsilon$ and $\|\overline{x}\| \leq \|x\| \leq 2$, we have 
\begin{equation}   
\overline{x} h_{k} (\overline{v}_{\beta} \overline{u})f_{\delta_{k}}(\overline\phi(a_k)) \overline{x}^{\ast} \approx_{7\epsilon} 1_{\calC(J)}, 
\end{equation}   
which implies that $h_{k} (\overline{v}_{\beta} \overline{u})f_{\delta_{k}}(\overline{\phi}(a_k))$ is full in $\calC(J)$ as we chose $\epsilon < 1/7$. Notice that $f_{\delta_{k}}(\overline{\phi}(a_{k}))$ is Cuntz equivalent to $\overline{\phi}(a_{k})$ in $\overline{\phi}(A)$. Because $\overline{v}_{\beta} \overline{u} \in \calC(J)\cap \overline{\phi}(A)'$, it commutes with elements witnessing $f_{\delta_{k}}(\overline{\phi}(a_{k}))\sim \overline{\phi}(a_{k})$ in $\overline{\phi}(A)$, and thus $h_{k} (\overline{v}_{\beta} \overline{u}) f_{\delta_{k}}(\overline{\phi}(a_{k}))$ is Cuntz equivalent to $h_{k}(\overline{v}_{\beta} \overline{u}) \overline{\phi}(a_{k})$ in $\calC(J)$. As a result, $h_{k}(\overline{v}_{\beta} \overline{u}) \overline{\phi}(a_{k})$ is full in $\calC(J)$. 
\end{proof} 

Combining Theorem \ref{KK_KL_uniqueness} and Corollary \ref{RR0-K1-inj}, gives our $KK_{\mathrm{nuc}}$ and $KL_{\mathrm{nuc}}$-uniqueness theorems (Theorem \ref{Intro:KKUnique}). 

\begin{thm}  
\label{KL-uniqueness} 
Let $A$ be a unital and separable $C^*$-algebra. Let $J$ be a separable and stable $C^*$-algebra which has real rank zero, stable rank one, $K_{1}(J) = 0$ and totally ordered $V(J)$. Let $(\phi, \psi) \colon A\rightrightarrows \calM(J)\rhd J$ be a Cuntz pair where both $\phi$ and $\psi$ are unital, weakly nuclear and unitally nuclearly absorbing. 
\begin{enumerate} [(i)]
\item If $[\phi, \psi]_{KK_{\mathrm{nuc}}} = 0$, there exists a norm-continuous path $(u_{t})_{t\geq 0}$ of unitaries in $J^\dagger$ such that 
\begin{equation} 
\|u_t(\phi(a))u_t^*-\psi(a)\|\to 0,\quad a\in A. 
\end{equation} 
\item If $[\phi, \psi]_{KL_{\mathrm{nuc}}} = 0$, there exists a sequence $(u_{n})_n$ of unitaries in $J^\dagger$ such that 
\begin{equation} 
\|u_n(\phi(a))u_n^*-\psi(a)\|\to 0,\quad a\in A. 
\end{equation} 
\end{enumerate} 
\end{thm}  

In the next section, Theorem \ref{KL-uniqueness} will be used to obtain uniqueness results for maps (Theorem \ref{II1_factor:intro} and Theorem \ref{quasidiagonality}), through the abstract classification framework. 

\section{Uniqueness theorems} 
\label{classification_chapter} 
In the first subsection, we prove the technical version of our main uniqueness theorem for $^*$-homomorphisms into ultraproducts of allowed codomains (see Section \ref{Sect:sep} for a discussion of our codomain hypotheses), and in the second we deduce the main results for uniqueness of embeddings into ultraproducts of finite factors as given in the introduction. 

\subsection{Main uniqueness theorem}

Here is the technical version of our main theorem.
\begin{thm} 
\label{MainUniqueness} 
Let $A$ be a separable, unital and exact $C^{\ast}$-algebra satisfying the UCT. Let $(B_n)_{n}$ be a sequence of unital $C^*$-algebras which have real rank zero, stable rank one, a tracial state $\tau_n$, totally ordered $V(B_{n})$ and $K_1(B_n)=0$. Write $B_{\omega}$ for $\prod_{\omega}B_{n}$, and $\tau_{B_\omega}$ for the (necessarily unique) trace on $B_\omega$ arising as the limit trace of the sequence $(\tau_n)_{n}$.  Assume that $\lim_{n\to\omega}\dim\pi_{\tau_n}(B_n)''=\infty$. 

Given unital, full and nuclear $^*$-homomorphisms $\phi,\psi\colon A\to B_\omega$ with $\tau_{B_\omega}\circ\phi=\tau_{B_\omega}\circ\psi$, $K_0(\phi)=K_0(\psi)$ and $K_0(\phi;\mathbb Z/n\mathbb Z)=K_0(\psi;\mathbb Z/n\mathbb Z)$ for all $n\geq 2$, there exists a unitary $u\in B_\omega$ with $\psi=\mathrm{Ad}\,u\circ\phi$.
\end{thm}

The proof essentially consists of following Schafhauser's uniqueness theorem (\cite[Proposition 4.3]{Schafhauser.Annals}), gluing in our $KL$-uniqueness theorem (Theorem \ref{KL-uniqueness}) in place of the $\mathcal Q$-stable $KK$-uniqueness theorem (\cite[Proposition 2.7]{Schafhauser.Annals}).\footnote{Comparing to the proof of \cite[Proposition 4.3]{Schafhauser.Annals}, the deunitisation of unitally nuclearly absorbing maps to make them nuclearly absorbing is no longer necessary, as our $KL_{\mathrm{nuc}}$-uniqueness theorem (Theorem \ref{KL-uniqueness}) is proved for unitally nuclearly absorbing maps.}

The other formal difference is that our classification invariant uses $K_0$ together with $K_0$ with coefficients in $\mathbb Z/n\mathbb Z$ for all $n\geq 2$, rather than just $K_0$. In Schafhauser's theorem, the tensorial factor of $\mathcal Q$ in the codomains ensures that $K$-theory with $\mathbb Z/n\mathbb Z$ coefficients vanishes, and hence $K_0$ carries the only $K$-theoretic information associated to a morphism. However there is no real difference here: Schafhauser's methods show that uniqueness corresponds to the vanishing of the class in $KL(A,J_B)$ obtained from the pair $(\phi,\psi)$ (after we adjust so that $\phi$ and $\psi$ agree modulo $J_B$). Under the rest of the hypotheses involved, Dadarlat and Loring's universal multicoefficient theorem shows that this amounts to $\phi$ and $\psi$ agreeing on total $K$-theory. (Just as in \cite{Schafhauser.Annals}, and in contrast to the results of \cite{CGSTW}, there is no algebraic $K_1$ component). In our situation, all $K_1$-groups (with coefficients) of the codomains vanish (as explained in Lemma \ref{prop:B_omega}), and hence the classification is by $K_0$ with all $\mathbb Z/n\mathbb Z$ coefficients. Something very similar happens if $\mathcal Q$ is replaced by another UHF-algebra of infinite type, such as the CAR algebra $M_{2^\infty}$ in Schafhauser's theorem: one gets uniqueness provided the invariant is augmented to include $K_0$ with suitable coefficient groups (in the case of $M_{2^{\infty}}$-stable codomains $B$, one should use the groups $K_0(\cdot;\mathbb Z/n\mathbb Z)$, where $n$ is co-prime to $2$).  This was known to Schafhauser at the time of writing \cite{Schafhauser.Annals}; given for example in his BIRS talk in 2017. 

\begin{proof} 
Fix unital, full and nuclear $^*$-homomorphisms $\phi,\psi\colon A\to B_\omega$ with $\tau_{B_{\omega}}\circ\phi=\tau_{B_{\omega}}\circ\psi$, $K_{0}(\phi) = K_0(\psi)$ and $K_0(\phi;\mathbb Z/n\mathbb Z)=K_0(\psi;\mathbb Z/n\mathbb Z)$ for all $n\geq 2$. We work with the trace-kernel extension 
\begin{equation} 
0\longrightarrow J_B\stackrel{j_B}{\longrightarrow} B_\omega\stackrel{q_B}{\longrightarrow} B^\omega \longrightarrow 0
\end{equation} 
described in Section \ref{Sect:sep}. Since each $\tau_n$ is the unique trace on $B_n$ (by Proposition \ref{prop.codomain.strictcomp}), $B^\omega$ is a finite von Neumann factor with trace $\tau_{B^\omega}$ satisfying $\tau_{B_\omega}=\tau_{B^\omega}\circ q_B$ as set out in Proposition \ref{prop:TKquotient}. By assumption, we have 
\begin{equation} 
\tau_{B^\omega}\circ q_B\circ\phi=\tau_{B_{\omega}}\circ \phi = \tau_{B_{\omega}}\circ\psi = \tau_{B^\omega}\circ q_B\circ\psi. 
\end{equation} 
Then, by the uniqueness of (weakly) nuclear maps which induce the same trace in \cite[Proposition 1.1]{Schafhauser.Annals}, which is a consequence of Connes' theorem, $q_B\circ\phi,q_B\circ\psi\colon A\to B^\omega$ are approximately unitarily equivalent, and hence unitarily equivalent (via reindexing in $B^\omega$ as $A$ is separable).  As unitaries in II$_1$ factors are all of the form $e^{2\pi ih}$, for some self-adjoint contraction $h$, they lift to $B_\omega$, and we can replace $\psi$ by a unitary conjugate so that $q_B\circ\phi=q_B\circ\psi$.  

By Lemma \ref{prop:B_omega}, we have $K_1(B_\omega;\mathbb Z/n\mathbb Z)=0$ for all $n\geq 2$, so that $\underline{K}(\phi)=\underline{K}(\psi)$. We now separablise using Lemma \ref{seplem} to obtain a separable unital subextension 
\begin{equation}
\begin{tikzcd}
0\arrow[r]&J\arrow[r,"j"]\arrow[d]&E\arrow[r,"q"]\arrow[d]&D\arrow[r]\arrow[d]&0\\
0\arrow[r]&J_B\arrow[r,"j_B"]&B_{\omega}\arrow[r,"q_B"]&B^\omega\arrow[r]&0
\end{tikzcd}
\end{equation} 
such that
\begin{enumerate}
\item \label{condition_sep_full} $\phi(A)\cup \psi(A)\subseteq E$, and the corestrictions $\phi|^E,\psi|^E\colon A\to E$ of $\phi$ and $\psi$ to $E$ are unital, nuclear, full, and satisfy $\underline{K}(\phi|^E)=\underline{K}(\psi|^E)$; 
\item $K_{1}(J)$, $K_{1}(D)$ and $K_{i}(D; \mathbb{Z} / n\mathbb{Z})$ vanish for $i\in \{0,1\}$ and $n\geq 2$; 
\item \label{condition_sep_J}$J$ has real rank zero and stable rank one, $V(J)$ is totally ordered and every projection in $J\otimes\mathcal K$ is Murray--von Neumann equivalent to a projection in $J$. 
\end{enumerate}

Let $\lambda\colon E\to \mathcal M(J)$ be the canonical unital $^*$-homomorphism. Since $q_B\circ \phi = q_{B}\circ \psi$, we have $q\circ \phi|^E = q\circ \psi|^E$ and hence 
\begin{equation}
(\lambda\circ\phi|^E,\lambda\circ\psi|^E)\colon A\rightrightarrows \calM(J)\rhd J
\end{equation}
is a Cuntz pair inducing a class $\kappa\in KL_{\mathrm{nuc}}(A,J)$ which satisfies 
\begin{equation}
KL_{\mathrm{nuc}}(A,j)(\kappa)=[\phi|^E]_{KL_{\mathrm{nuc}}(A,E)}-[\psi|^E]_{KL_{\mathrm{nuc}}(A,E)}.\footnote{This is the computation of \cite[Proposition 2.1]{Schafhauser.Annals}, followed by the quotient map from $KK_{\mathrm{nuc}}(A,I)$ to $KL_{\mathrm{nuc}}(A,I)$.}
\end{equation}
Since $A$ satisfies the UCT, Dadarlat and Loring's universal multicoefficient theorem from \cite{Dadarlet-Loring} (see Theorem \ref{UMCT}) gives an isomorphism, 
\begin{equation}
\Gamma^{(I)}\colon KL_{\mathrm{nuc}}(A,I)\rightarrow \text{Hom}_{\Lambda}(\underline{K}(A),\underline{K}(I)), 
\end{equation} 
which is natural in $I$ for any separable $C^*$-algebra $I$. Naturality and the fact that 
\begin{equation} 
\Gamma^{(E)}([\phi|^E]_{KL_{\mathrm{nuc}}(A,E)})=\underline{K}(\phi|^E)=\underline{K}(\psi|^E)=\Gamma^{(E)}([\psi|^E]_{KL_{\mathrm{nuc}}(A,E)})
\end{equation} 
gives 
\begin{equation}\underline{K} (j) \circ \Gamma^{(J)}(\kappa)= \Gamma^{(E)} \circ KL_{\mathrm{nuc}}(A, j) (\kappa) = 0.
\end{equation}

As $K_1(D)=0$, the six-term exact sequence shows that $K_0(j)$ is injective.  Likewise, as $K_i(D;\mathbb Z/n\mathbb Z)=0$, the six-term exact sequence for $K$-theory with coefficients in $\mathbb Z/n\mathbb Z$ shows that $K_i(j;\mathbb Z/n\mathbb Z)$ is injective for all $i\in \{0,1\}$ and $n\geq 2$. Combining this with the fact that $K_1(J)=0$, we get injectivity of $\underline{K}(j)$, and hence (as $\Gamma^{(J)}$ is an isomorphism)  $\kappa = [\lambda\circ\phi|^E,\lambda\circ\psi|^E]_{KL_{\mathrm{nuc}}(A, J)} =0$.

By Lemma \ref{separableideals}, $J$ is stable and has the corona factorisation property, and so $\lambda\circ \phi|^E$ and $\lambda\circ \psi|^E$ are unitally nuclearly absorbing by Elliott--Kucerovsky's generalized Weyl--von Neumann theorem given as Theorem \ref{nuc_absorb_equiv}. Since $\phi|^E$ and $\psi|^E$ are nuclear, so too are $\lambda\circ \phi|^E$ and $\lambda\circ \psi|^E$. Then our $KL_{\mathrm{nuc}}$-uniqueness theorem (Theorem \ref{KL-uniqueness}), shows that there exists a sequence of unitaries $(u_n)_{n}$ in $J^{\dagger}$ such that 
\begin{equation}
\|u_n(\lambda\circ  \phi|^E(a))u_n^*-\lambda\circ  \psi|^E(a)\|\to 0,\quad a\in A. 
\end{equation}
Since $\lambda$ restricts to the identity map on $J^\dagger$, it follows that  
\begin{equation} 
\|u_n  \phi|^E(a)u_n^*- \psi|^E(a)\|\to 0,\quad a\in A. 
\end{equation} 
Thus the unital $^*$-homomorphisms $\phi, \psi: A\rightarrow B_{\omega}$ are approximately unitarily equivalent.  As the codomain is an ultraproduct and $A$ is separable, it is standard that the approximate unitary equivalence automatically upgrades to an exact unitary equivalence (by reindexing, or Kirchberg's $\epsilon$-test).
\end{proof} 

\subsection{Uniqueness of embeddings into ultraproducts of finite factors} 
We now formally collect the main results from the introduction, starting with Theorem \ref{intro:ultraproductthm}.  Here, the case when the dimension condition in Theorem \ref{MainUniqueness} fails amounts to the classification of injective maps $A\to M_r$ for some matrix algebra $M_r$, which is standard. 
\begin{proof}[Proof of Theorem \ref{intro:ultraproductthm}]
Let $(\mathcal M_n)_{n}$ be a sequence of finite von Neumann factors. These satisfy the codomain hypotheses of Theorem \ref{MainUniqueness}. There is a dichotomy between the case when $\omega$-many of the $\mathcal M_n$ are type II$_1$ factors, in which the dimension condition $\lim_{n\to\omega}\pi_{\tau_n}(\mathcal M_n)''$ is automatic, or the case when $\omega$-many of the $\mathcal M_n$ are finite type I factors, i.e. $\mathcal M_n\cong M_{k_n}$ for $\omega$-many $n$. In the latter case, either the dimension condition holds, and the result follows from Theorem \ref{MainUniqueness}, or there exists some $r\in\mathbb N$ such that $\lim_{n\to\omega}k_n=r$. In the remaining case, 
the $C^*$-ultraproduct and von Neumann ultraproduct of $(\mathcal M_n)_{n}$ are both identified with $M_r$, and the result is immediate from the classification of maps into matrix algebras.\end{proof}

Corollary \ref{intro:qdcor} on the uniqueness of quasidiagonality is immediate from Theorem \ref{intro:ultraproductthm}.  It remains to deduce Theorem \ref{II1_factor:intro}. As with all of our results, this works in the level of generality of nuclear maps from exact domains as stated below (which includes Theorem \ref{II1_factor:intro} as a special case).

\begin{thm}\label{II1_factor} 
Let $A$ be a separable, unital and exact $C^*$-algebra satisfying the UCT, and let $\calM$ be a II$_{1}$-factor with trace $\tau_{\calM}$. Let $\phi, \psi: A\rightarrow \calM$ be unital, nuclear and faithful $^*$-homomorphisms such that $\tau_{\calM} \circ \phi = \tau_{\calM} \circ \psi$, then there exists a sequence of unitaries $(u_{n})_{n}$ in $\calM$ such that 
\begin{equation}\label{II1_factor.1} 
\|u_n \phi(a)u_n^*-\psi(a)\|\to 0, \quad a\in A. 
\end{equation} 
\end{thm} 

\begin{proof} 
We apply Theorem \ref{intro:ultraproductthm} with $\mathcal M_n=\mathcal M$ and denote the inclusion of $\calM$ into its norm ultrapower by $\iota: \calM\rightarrow \mathcal M_\omega$. Given unital, faithful and nuclear maps $\phi,\psi:A\to\mathcal M$ (which are full since $\calM$ is simple) with $\tau_{\mathcal M}\circ\phi=\tau_{\mathcal M}\circ \psi$, then $\iota\circ \phi, \iota\circ \psi: A\rightarrow \calM_{\omega}$ are unital, full and nuclear and agree on $\tau_{\calM_{\omega}}$. As traces determine projections in $\mathcal M$, we have $K_0(\phi)=K_0(\psi)$, and so $K_0(\iota\circ\phi)=K_0(\iota\circ\psi)$. Also $K_0(\iota\circ\phi;\mathbb Z/n\mathbb Z)=K_0(\iota\circ\psi;\mathbb Z/n\mathbb Z)$, as these maps factor through $\mathcal M$ which has $K_0(\mathcal M;\mathbb Z/n\mathbb Z)=0$. Then $\iota\circ\phi$ and $\iota\circ\psi$ are unitarily equivalent by Theorem \ref{intro:ultraproductthm}, which implies that $\phi$ and $\psi$ are approximately unitarily equivalent.
\end{proof}

Combining Theorem \ref{II1_factor} and \cite[Proposition 1.1]{Schafhauser.Annals} as a consequence of Connes' theorem, we get the equivalence of unitary equivalence in norm and $2$-norm for maps into II$_{1}$-factors. 

\begin{cor} 
    Let $A$ be a separable, unital and exact $C^*$-algebra satisfying the UCT, and let $\mathcal M$ be a II$_1$ factor.  Let $\phi,\psi\colon A\to\mathcal M$ be unital, faithful and nuclear $^*$-homomorphisms. Then the following are equivalent:
\begin{enumerate}
    \item $\tau_{\calM}\circ\phi=\tau_{\calM}\circ\psi$;
    \item there exists a sequence of unitaries $(v_n)_{n}$ in $\mathcal M$ with $\|v_n\phi(a)v_n^*-\psi(a)\|_2\to 0$ for all $a\in A$; 
    \item there exists a sequence of unitaries $(u_n)_{n}$ in $\mathcal M$ with $\|u_n\phi(a)u_n^*-\psi(a)\|\to 0$ for all $a\in A$. 
\end{enumerate}
\end{cor} 

We end by recording how to recapture Hadwin, Li and Liu's result classifying maps from inductive limits of type I algebras into II$_1$ factors from Theorem \ref{II1_factor}. 

\begin{cor}[{Hadwin--Li--Liu, \cite{LHL:OM}}]
\label{cor:HLL_reproof} 
     Let $A$ be a separable and unital $C^*$-algebra which is an inductive limit of type I $C^*$-algebras and let $\mathcal M$ be a II$_1$-factor. Any two unital $^*$-homomorphisms $\phi,\psi\colon A\to\mathcal M$ with $\tau_\mathcal M\circ\phi=\tau_\mathcal M\circ\psi$ are approximately unitarily equivalent in the norm topology.  
\end{cor}
\begin{proof} 
As $\tau_{\calM}\circ \phi = \tau_{\calM}\circ \psi$ and $\tau_{\calM}$ is faithful, $\phi$ and $\psi$ have the same kernel $I$.  Then $A/I$ is also an inductive limits of type $I$ $C^*$-algebras,\footnote{This is a folklore combintaion of standard facts, but we have not found a precise reference. Firstly quotients of type I $C^*$-algebras are again type I. For this reason we may as well assume that the connecting maps in the inductive limit are injective.  As ideals in a $C^*$-algebra inductive limit are inductive (see \cite[Lemma III.4.1]{Davidson}, for example), it follows that $A/I$ is an inductive limit of type I $C^*$-algebras in just the same way that quotients of AF $C^*$-algebras are AF (see \cite[Theorem III.4.4]{Davidson}, for example).} and hence satisfies the UCT (this goes back to \cite{RS:UCT}).  The result now follows by applying Theorem \ref{II1_factor} to maps $\hat{\phi}, \hat{\psi}: A/I\rightarrow \calM$ induced from $\phi$ and $\psi$. 
\end{proof}

\appendix 
\section{Classification of maps into type III factors} 
\label{Appendix_A} 
In this appendix, we explain how to obtain the following classification of maps into type III factors from Kirchberg's classification theorems. We thank Jamie Gabe and Chris Schafhauser for bringing the $\infty+1=\infty$ trick in the $KK$-calculation below to our attention which means that the UCT is not needed in this case.

\begin{thm}\label{Appendix:Thm}
Let $A$ be a separable, unital and exact $C^*$-algebra, and let $\mathcal M$ be a type III factor with separable predual. Then any two unital and nuclear $^*$-homomorphisms $A\to\mathcal M$ are approximately unitarily equivalent if and only if they have the same kernel.
\end{thm}

We will obtain Theorem \ref{Appendix:Thm} from the following classification theorem (we could equally have used a version of the theorem characterising approximate unitary equivalence in terms of $KL_{\mathrm{nuc}}$).
\begin{thm}[Kirchberg]\label{Kirchberg}
Let $A$ be a separable, unital and exact $C^*$-algebra, and let $B$ be a unital, simple purely infinite $C^*$-algebra.  Then two unital, injective and nuclear $^*$-homomorphisms $\phi,\psi:A\to B$ are asymptotically unitarily equivalent (i.e. there is a continuous path of unitaries $(u_t)_{t\geq 0}$ in $B$ with $u_t\phi(a)u_t^*\to \psi(a)$ as $t\to\infty$ for all $a\in A$) if and only if $KK_{\mathrm{nuc}}(\phi)=KK_{\mathrm{nuc}}(\psi)$.
\end{thm}
 
As with many results of this type, unfortunately Kirchberg never published full details of Theorem \ref{Kirchberg}. It can be found as Theorem 8.3.3 (iii) of R{\o}rdam's book (\cite{Rordam:Book}) where it is called `Kirchberg's Classification Theorem'. For the details, R\o{}rdam refers to Kirchberg's manuscript \cite{Kirchberg99}, and Kirchberg's Fields Institute preprint, which morphed into his huge book project, the unfinished version of which was published posthumously in 2025 (\cite{Kirchberg:Manuscript}). This demonstrates that these results were known to Kirchberg back in the 1990s. 
 
Gabe's new treatment \cite{Gabe:Crelle,Gabe:MAMS} now gives a complete proof of Theorem \ref{Kirchberg} and many more results. Though it is not stated explicitly in \cite{Gabe:MAMS}, Theorem \ref{Kirchberg} is a special case of \cite[Theorem B]{Gabe:MAMS} which gives such a classification result for unital, full, nuclear and strongly $\mathcal O_\infty$-stable $^*$-homomorphisms\footnote{See \cite[Section 3.3]{Gabe:Crelle} and \cite[Section 4]{Gabe:MAMS} for a definition and discussion of $\mathcal O_\infty$-stable maps.} from a separable unital $C^*$-algebra $A$ into a unital $C^*$-algebra $B$. For simple (non-zero) $B$, fullness is the same as injectivity, and for such a nuclear map to exist, $A$ will need to be exact. When $A$ is separable and exact, and $B$ is simple and purely infinite, all nuclear $^*$-homomorphisms are strongly $\mathcal O_\infty$-stable by \cite[Corollary 9.8]{Gabe:MAMS}, and in this way \cite[Theorem B]{Gabe:MAMS} encompasses Theorem \ref{Kirchberg}. This approach uses tensorial absorption results (and the complete version of the argument for \cite[Corollary 9.8]{Gabe:MAMS} factors through the ideas used in Kirchberg and R\o{}rdam's extensions of Kirchberg's Geneva theorems to strongly purely infinite nuclear $C^*$-algebras \cite{KR:Adv}).

Alternatively, and closer to the ideas we used to prove Theorem \ref{II1_factor:intro}, Bouwen and Gabe have very recently shown how to use trace-kernel classification techniques to give a new approach to purely infinite classification results.  They obtain tensorial absorption results as a consequence of classification, rather than as an input to it. Theorem 6.16 of \cite{Bouwen-Gabe-2024} gives a classification of maps from separable nuclear $C^*$-algebras into simple purely infinite $C^*$-algebras in terms of (a slight modification of) $KL$. 

\begin{proof}[Proof of Theorem \ref{Appendix:Thm}]
Certainly if $\phi$ and $\psi$ are approximately unitarily equivalent, they have the same kernel. Conversely if they have the same kernel $I$, we can replace $A$ by $A/I$ and it suffices to show that two unital, nuclear and injective maps $\phi,\psi:A\to\mathcal M$ are approximately or even asymptotically unitarily equivalent, which will follow from Theorem \ref{Kirchberg} by showing $KK_{\mathrm{nuc}}(\phi)=0$.

Fix a family $(v_n)_n$ of isometries in $\mathcal M$ with pairwise orthogonal ranges such that $\sum_{n=1}^\infty v_nv_n^*=1$ (with strong convergence). Then given a unital and nuclear $^*$-homomorphism $\phi:A\to\mathcal M$, one can define the infinite repeat $\phi^{(\infty)}:A\to \mathcal M$ by $\phi(x)=\sum_{n=1}^\infty v_n\phi(x)v_n^*$. By construction we will have $[\phi^{(\infty)}]+[\phi]=[\phi^{(\infty)}]$ in $KK_{\mathrm{nuc}}(A,\mathcal M)$, and hence $[\phi]=0$ in $KK_{\mathrm{nuc}}(A,\mathcal M)$.
\end{proof}

\end{document}